\documentclass[a4paper, 12pt]{amsart}

\usepackage[pdftex]{graphicx} 

\usepackage{amssymb}
\usepackage{amsmath}
\usepackage{amsthm}
\usepackage{amscd}
\usepackage{comment}
\usepackage[all,cmtip]{xy}
\usepackage{bm}

\usepackage{color}

\usepackage[utf8]{inputenc}

\usepackage{tikz}
\usetikzlibrary{positioning}
\usetikzlibrary{intersections}
\usetikzlibrary{calc, quotes, angles}

\theoremstyle{plain}
  \newtheorem{thm}{Theorem}[section]
  \newtheorem{lem}[thm]{Lemma}
  \newtheorem{slem}[thm]{Sublemma}
  \newtheorem{cor}[thm]{Corollary}
  \newtheorem{prop}[thm]{Proposition}
  \newtheorem{clm}[thm]{Claim}
  
  \newtheorem{ass}[thm]{Assertion}

\theoremstyle{definition}
  \newtheorem{defn}[thm]{Definition}
  \newtheorem{conj}[thm]{Conjecture}
  \newtheorem{ex}[thm]{Example}
  
  \newtheorem{prob}[thm]{Problem}

  \newtheorem{rem}[thm]{Remark}

\theoremstyle{remark}

\makeatletter

\newcommand{\wangle}[0]{\tilde{\angle}}
\newcommand{\diam}[0]{\mathrm{diam}\,}

\newcommand{\e}[0]{\epsilon}

\newcommand{\vol}[0]{\mathrm{vol}}

\newcommand{\ca}{\mathcal}
\newcommand{\ve}{\varepsilon}

\newcommand{\pa}{\partial}

\newcommand{\ol}{\overline}

\newcommand{\inte}[0]{{\rm int}}

\newcommand{\pmed}[0]{\par\medskip}
\newcommand{\psmall}[0]{\par\smallskip}
\newcommand{\pbig}[0]{\par\bigskip}
\newcommand{\n}[0]{\noindent}

\newcommand{\R}[0]{\mathbb R}

\newcommand{\beq}[0]{\begin{equation}}
\newcommand{\eeq}[0]{\end{equation}}
\newcommand{\beqq}[0]{\begin{equation*}}
\newcommand{\eeqq}[0]{\end{equation*}}
\newcommand{\bali}[0]{\begin{align}}
\newcommand{\eali}[0]{\end{align}}
\newcommand{\balii}[0]{\begin{align*}}
\newcommand{\ealii}[0]{\end{align*}}

\newcommand{\benu}[0]{\begin{enumerate}}
\newcommand{\eenu}[0]{\end{enumerate}}

\begin{document}

\title[Limits of manifolds with boundary]{Limits of manifolds with boundary II
}

\author{Takao Yamaguchi \and Zhilang Zhang }

%

\address{Takao Yamaguchi, Institute of Mathematics, University of Tsukuba, Tsukuba 305-8571, Japan}
%
%
\email{takao@math.tsukuba.ac.jp}

\address{Zhilang Zhang, School of Mathematics, Foshan University, Foshan, China}
 \email{zhilangz@fosu.edu.cn}


\thanks{This work was supported by JSPS KAKENHI Grant Numbers 18H01118, 21H00977  and  NSFC No.11901089}

\date{\today}

\begin{abstract} 
In this  paper, as a continuation of [34], we consider the Gromov-Hausdorff convergence and collapsing in the family of compact Riemannian manifolds with boundary satisfying lower bounds on the sectional curvatures of  manifolds,  boundaries and the second fundamental forms of boundaries, and an upper diameter bound.  We describe the local geometric structure of the limit spaces, and establish  some stability results including Lipschitz homotopy stability. 
\end{abstract}

\maketitle

\setcounter{tocdepth}{2}

\tableofcontents

\section{Introduction} \label{sec:intro}

For $n\ge 2, \kappa, \nu\in\R$, 
$\lambda\ge 0$ and $d>0$, let 
$\ca M(n,\kappa,\nu, \lambda, d)$ denote the set of all isometry classes of 
$n$-dimensional compact  Riemannian manifolds $M$ with boundary such that   
the  sectional curvatures of $M$ and $\pa M$, the second fundamental forms of $\pa M$
and the diameters satisfy 
\beq\label{eq:curv-assum}
      K_M \ge \kappa, \,\,\, K_{\pa M} \ge \nu,\,\,\,
 \Pi_{\partial M} \ge -\lambda, \,\,\,\diam(M)\le d,
\eeq
where  we do not need the condition on the lower curvature bound for $\pa M$
when $n=\dim M=2$.
In \cite{wong0}, Wong proved the 
 precompactness of $\ca M(n,\kappa,\nu, \lambda, d)$ 
with respect  to the Gromov-Hausdorff distance.

The purpose of the present paper is to 
determine  the  local geometric structure of  the limit spaces, and establish some results on local and global stabilities.
Let us consider a convergence
\beq \label{eq:convMtoN-intro}
\ca M(n,\kappa,\nu,\lambda,d)\ni M_i\to N\in
\ol{\ca M(n,\kappa,\nu,\lambda,d)}.
\eeq 
Note that the case of inradius collapse,
where the inradii ${\rm inrad}(M_i)$ of $M_i$
converge to $0$, the limit space $N$
is known to be an Alexandrov space
with curvature $\ge \nu$
(\cite{YZ:inradius}). Therefore in most cases, we assume the general case that 
the inradii ${\rm inrad}(M_i)$ of $M_i$ are
uniformly bounded away from $0$.
 In this case, we call \eqref{eq:convMtoN-intro} a {\it non-inradius collapse/convergence}, 
and say that $M_i$  non-inradius collapses/converges to $N$ 
as in \cite{YZ:part1}.
The {\em boundary} $N_0$ of  the limit space $N$ is
defined as 
\[   \text{
       $N_0:=\lim_{i\to\infty} \pa M_i$\quad under \eqref{eq:convMtoN-intro}.}
\]
Note that the  {\em interior} 
\[
      {\rm int} N:=N\setminus N_0
\]
of $N$ satisfies the local Alexandrov curvature condition $\ge \kappa$,
although it is not complete. 
Obviously, it is not the case for points of $N_0$.
Thus one of our main concerns is on the local structure around the  points of $N_0$.

From  \eqref{eq:curv-assum}, we may assume that 
the intrinsic metric $(\pa M_i)^{\rm int}$ of
$\pa M_i$ 
 converges to some
Alexandrov space $C_0$. We may also assume that the natural map
$\eta_i: (\pa M_i)^{\rm int} \to (\pa M_i)^{\rm ext}$
to the extrinsic metric,
converges to a surjective $1$-Lipschitz map 
$$
     \eta_0:C_0\to N_0.
$$
It is shown in \cite{YZ:inradius}  that 
$\# \eta_0^{-1}(x)\le 2$  for 
all $x\in N_0$, and $\eta_0$ 
induces a $1$-Lipschitz map  $d\eta_0:\Sigma_p(C_0) \to \Sigma_x(N)$ for any $p\in \eta_0^{-1}(x)$.
 
We call a point $x\in N_0$ {\it single}  (resp. {\it double})    if $\# \eta_0^{-1}(x)=1$
(resp. $\# \eta_0^{-1}(x)=2$).
We denote by $N_0^1$ (resp. by $N_0^2$) 
the set of all single points (resp. double points)
in $N_0$.
For $k=1$ or $2$,
let ${\rm int}\, N_0^k$ denote the interior of $N_0^k$ in $N_0$, and let 
$\pa N_0^k=\bar N_0^k\setminus {\rm int}\, N_0^k$ be the topological boundary of 
$N_0^k$ in $N_0$.
We set  
$$
\mathcal S^k :=\pa N_0^k \cap N_0^k.
$$ 

 In our previous paper \cite{YZ:part1},
we made clear the infinitesimal structure
of $N$.
For instance, $N$ becomes {\it infinitesimally Alexandrov} which implies that 
for any $x\in N_0$, the space of directions
$\Sigma_x(N)$ as well as the intrinsic
metric $\Sigma_x(N_0)^{\rm int}$ of 
$\Sigma_x(N_0)\subset\Sigma_x(N)$ 
is an 
Alexandrov space with curvature $\ge 1$.
Moreover the tangent cones, denoted
$T_x(N)$ and $T_x(N_0)$, of $N$ and $N_0$
at $x$ are isometric to the Euclidean cone
over $\Sigma_x(N)$ and $\Sigma_x(N_0)$
respectively (see \cite[Theorem 1.1]{YZ:part1}).

\psmall 
 Now we state our main results.
Let $m:=\dim N$ denote the topological dimension of $N$.

\pmed
\n
{\bf Results on almost isometries.}\,
\pmed
Let a sequence $M_i$ in $\ca M(n,\kappa,\nu,\lambda, d)$ non-inradius collapses/converges  to a geodesic space $N$.
By \cite[Lemma 4.18]{YZ:part1}, we have 
$$
\dim N=\dim_H N=m, \qquad   \dim N_0=\dim_H N_0=m-1.
$$
We call a point $x\in N_0$ {\it metrically regular} if  the intrinsic metric
$\Sigma_x(N_0)^{\rm int}$ of $\Sigma_x(N_0)$
in $\Sigma_x(N)$
 is isometric to $\mathbb S^{m-2}$. 
Otherwise we call $x$ {\it metrically singular}.
We denote by $N_0^{\rm reg}$ (resp. by $N_0^{\rm sing}$) the set of all metrically regular 
(resp. all metrically singular) points of  $N_0$.
\begin{defn}\label{defn:thin-thick-brpoint}
We say that a point $x$ of $N_0$ is {\it thick} 
(resp.\,\,{\it thin}) in $N$ if $\Sigma_x(N_0)\neq\Sigma_x(N)$
(resp. $\Sigma_x(N_0)=\Sigma_x(N)$).
A thick point $x$ is called a {\it boundary regular 
point} of $N$ if $\Sigma_x(N)$ is isometric 
to the unit hemisphere $\mathbb S^{m-1}_+$ (see also Definition \ref{defn:thick}).

\end{defn}

Note that if $x$ is a boundary regular point of $N$, then $x\in N_0^{\rm reg}$.
 However the converse is not true as 
one see from an example where $N$ is a square and 
$N_0$ is the boundary of the square
(see Example \ref{ex:convexbody}).

Let ${\ca Tk}$ (resp. ${\ca Tn}$)
denote the set of all thick points
(resp. all thin points), and let ${\ca Tk}^{\rm reg}$
be the set of all boundary regular points of $N$.
We set ${\ca Tn}^{\rm reg}:={\ca Tn}\cap N_0^{\rm reg}$.
We show that  ${\ca Tk}^{\rm reg}\cup {\ca Tn}^{\rm reg}$ has full measure in 
$N_0$ with respect to the $(m-1)$-dimensional 
Hausdorff measure (see Lemma \ref{lem:dimTKNc}).

\pmed\n
{\bf Local structure of $N_0$.}\,
For $x\in N_0$, let  $B^{N_0^{\rm ext}}(x,r)$ (resp. $B^{N_0^{\rm int}}(x,r))$ denote the metric $r$-ball around $x$ in $N_0$ with respect to 
the extrinsic metric (resp. the intrinsic metric) induced from $N$.

The local structure of $N_0$ at any regular point of $N_0$ can be determined as follows.
We denote by $\tau_x(\,\,)$ a function
depending on $x$ such that $\lim_{\e\to 0}\tau_x(\e)=0$. 
We say that a map $f:X\to X'$ between metric 
spaces is an {\it $\e$-almost isometry} if 
it is bijective and satisfies 
$||f(x),f(y)|/|x,y|-1| <\e$ for all $x,y\in X$.

\begin{thm}\label{thm:regular-ball}
For any $x\in N_0^{\rm reg}$, there exists $\e>0$ satisfying the following$:$
\begin{enumerate}
\item If $x\in N_0^1$  in addition, then  $\mathring{B}^{N_0^{\rm int}}(x,\e)$  is $\tau_x(\e)$-almost isometric to an open subset of $\R^{m-1}\,;$
\item  If $x\in {\ca Tk}^{\rm reg}\cup
 ({\ca Tn}^{\rm reg}\cap N_0^1)$  in addition to $(1)$, then both
$\mathring{B}^{N_0^{\rm ext}}(x,\e)$ and
$\mathring{B}^{N_0^{\rm int}}(x,\e)$ are
$\tau_x(\e)$-almost isometric to open subsets of $\R^{m-1}\,;$
\item If $x\in N_ 0^2$  in addition, then both $\mathring{B}^{N_0^{\rm ext}}(x,\e)$ and $\mathring{B}^{N_0^{\rm int}}(x,\e)$ are the union of 
two subsets of $N_0$ that are $\tau_x(\e)$-almost isometric to  open subsets of $\R^{m-1}$
with respect to the extrinsic and the intrinsic metrics of $N_0$ respectively.
\end{enumerate}
\end{thm}

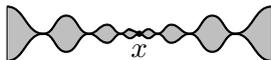
\pmed
\begin{figure}
\begin{center}
\begin{tikzpicture}
[scale = 0.6]

\filldraw [fill=lightgray, opacity=.1] 
 (0,0) to  [out=00, in=180]  (0.2,0.1)
to [out=00, in=180]  (0.4,0)
to  [out=00, in=180] (0.7,0.2)
to   [out=00, in=180] (1.1,0)
to   [out=00, in=180] (1.6,0.4)
to   [out=00, in=180] (2.1,0)
to   [out=00, in=180] (2.9,0.6)
to  [out=-90, in=90] (2.9,-0.6)
to  [out=180, in=00] (2.1,0)
to  [out=180, in=00] (1.6,-0.4)
to  [out=180, in=00] (1.1,0)
to  [out=180, in=00] (0.7,-0.2)
to  [out=180, in=00] (0.4,0)
to  [out=180, in=00] (0.2,-0.1)
to  [out=180, in=00] (0,0);

\draw [-, thick] (0.2,0.1) to [out=180, in=00] (0,0);
\draw  [-, thick] (0.4,0) to [out=180, in=00] (0.2, 0.1);
\draw  [-, thick] (0.7,0.2) to [out=180, in=00] (0.4,0);
\draw  [-, thick] (1.1,0) to [out=180, in=00] (0.7,0.2);
\draw [-, thick] (1.6,0.4) to [out=180, in=00] (1.1,0);
\draw  [-, thick] (2.1,0) to [out=180, in=00] (1.6,0.4);
\draw  [-, thick] (2.9,0.6) to [out=180, in=00] (2.1,0);
\draw  [-, thick] (0.2,-0.1) to [out=180, in=00] (0,0);
\draw  [-, thick] (0.4,0) to [out=180, in=00] (0.2, -0.1);
\draw  [-, thick] (0.7,-0.2) to [out=180, in=00] (0.4,0);
\draw [-, thick] (1.1,0) to [out=180, in=00] (0.7,-0.2);
\draw  [-, thick] (1.6,-0.4) to [out=180, in=00] (1.1,0);
\draw [-, thick] (2.1,0) to [out=180, in=00] (1.6,-0.4);
\draw  [-, thick] (2.9,-0.6) to [out=180, in=00] (2.1,0);

\filldraw [fill=lightgray, opacity=.1] 
(-2.9,0.6) to [out=0, in=180] (-2.1,0)
to   [out=0, in=180] (-1.6,0.4)
 to [out=0, in=180] (-1.1,0)
 to [out=0, in=180] (-0.7,0.2)
to [out=0, in=180] (-0.4,0)
to [out=0, in=180] (-0.2, 0.1)
to [out=0, in=180] (0,0)
 to [out=180, in=00] (-0.2,-0.1) 
 to [out=180, in=00] (-0.4,0) 
 to [out=180, in=00]   (-0.7,-0.2)
 to [out=180, in=00]  (-1.1,0)
 to [out=180, in=00]  (-1.6,-0.4)
 to [out=180, in=00]  (-2.1,0)
 to [out=180, in=00]   (-2.9,-0.6) 
to  (-2.9,0.6); 

\draw [-, thick] (-0.2,0.1) to [out=180, in=00] (0,0);
\draw  [-, thick] (-0.4,0) to [out=0, in=180] (-0.2, 0.1);
\draw  [-, thick] (-0.7,0.2) to [out=0, in=180] (-0.4,0);
\draw  [-, thick] (-1.1,0) to [out=0, in=180] (-0.7,0.2);
\draw [-, thick] (-1.6,0.4) to [out=0, in=180] (-1.1,0);
\draw  [-, thick] (-2.1,0) to [out=0, in=180] (-1.6,0.4);
\draw  [-, thick] (-2.9,0.6) to [out=0, in=180] (-2.1,0);
\draw  [-, thick] (-0.2,-0.1) to [out=0, in=180] (0,0);
\draw  [-, thick] (-0.4,0) to [out=0, in=180] (-0.2, -0.1);
\draw  [-, thick] (-0.7,-0.2) to [out=0, in=180] (-0.4,0);
\draw [-, thick] (-1.1,0) to [out=0, in=180] (-0.7,-0.2);
\draw  [-, thick] (-1.6,-0.4) to [out=0, in=180] (-1.1,0);
\draw [-, thick] (-2.1,0) to [out=0, in=180] (-1.6,-0.4);
\draw  [-, thick] (-2.9,-0.6) to [out=0, in=180] (-2.1,0);

\fill (0,0) circle (2pt) node [below] {\small{$x$}};

\end{tikzpicture}
\end{center}
\caption{}
\end{figure}

\pmed\n
{\bf Local structure of $N$.}\,
As observed in \cite{wong0} (see also \cite{YZ:part1}), any space $N\in\overline{\ca M}$ can be embedded in an Alexandrov 
space $Y$.
We denote by   $S^{N^{\rm ext}}(p,r)$ and $B^{N^{\rm ext}}(p,r)$ the extrinsic  $r$-sphere and $r$-ball
around $x$ in $N$, i.e., 
\[
        S^{N^{\rm ext}}(p,r):=S^Y(p,r)\cap N,\quad B^{N^{\rm ext}}(p,r):=B^Y(x,r)\cap N.
\]  

The local structure of $N$ around any
thick point can be described  by  
the local topological stability and the result on almost isometry, as follows.

\begin{thm}\label{thm:regular-ballX(intro)}
If $x\in N_0$ is thick, then for any small 
enough $0<r\le r_x$, 
we have the following:
\begin{enumerate}
\item $\mathring{B}^{N^{\rm ext}}(x,r)$  is homeomorphic  to $T_x(N)\,;$ 
\item 
If $x$ is a  boundary regular 
point in addition,  
$\mathring{B}^{N^{\rm ext}}(x,r)$  is $\tau_x(r)$-almost isometric to an open subset of $\R_+^{m}$.
\end{enumerate}
\end{thm}

\begin{rem}\label{rem:regular-ball}
(1)\,
Theorem  \ref{thm:regular-ballX(intro)} does not hold  for a thin point $x\in N_0$ in general.
\par\n
(2)\, Theorem  \ref{thm:regular-ballX(intro)} (2) does not hold  for a  point $x\in{\ca Tk}\cap N_0^{\rm reg}$ in general.
See Example \ref{ex:convexbody}.
\end{rem}
For the regular thin points of $N_0$, we obtain the following result on
almost isometries.

\begin{thm}\label{thm:almostisometry(intro)}
For any regular thin point $x$ of $N_0$,
there are $\e=\e_x>0$ and 
a  $\tau_x(\e)$-Lipschitz function
$g:B^{\R^{m-1}}(0,\e)\to\R_+$ such that 
a closed neighborhood $W$ of $x$ in $N$
is $\tau_x(\e)$-almost isometric to the closed subset $G_g$ of $\R^m$ defined by 
the graph of $g$ as
\[
G_g=\{ (y,t)\in B^{\R^{m-1}}(0,\e)\times\R_+\,|\, 0\le t\le
g(y)\}.
\]
\end{thm} 
\pmed
The following  is a figure concerning  Theorem \ref{thm:almostisometry(intro)} on the local structure around the regular thin points.

\pmed
\begin{figure}
\begin{center}
\begin{tikzpicture}
[scale = 0.6]

\draw [-, thick] (0,0) to [out=00, in=180] (4,0);
\draw  [-, thick] (0,0) to [out=180, in=0] (-4,0);
\draw [-, thick] (0,0) to [out=00, in=180] (1,-0.4);
\draw [-, thick] (1,-0.4) to [out=00, in=180] (2,0);
\draw [-, thick] (2,0) to [out=00, in=135] (4,-1);
\draw [-, thick] (0,0) to [out=180, in=00] (-1,-0.4);
\draw [-, thick] (-1,-0.4) to [out=180, in=00] (-2,0);
\draw [-, thick] (-2,0) to [out=180, in=45] (-4,-1);
\draw [-, thick] (-4,-1) -- (-4,0);
\draw [-, thick] (-3.5,-0.55) -- (-3.5,0);
\draw [-, thick] (-3.1,-0.3) -- (-3.1,0);
\draw [-, thick] (4,-1) -- (4,0);
\draw [-, thick] (3.5,-0.55) -- (3.5,0);
\draw [-, thick] (3.1,-0.3) -- (3.1,0);
\draw [-, thick] (-1.2,-0.4) -- (-1.2,0);
\draw [-, thick] (-0.8,-0.4) -- (-0.8,0);
\draw [-, thick] (1.2,-0.4) -- (1.2,0);
\draw [-, thick] (0.8,-0.4) -- (0.8,0);

\fill (-4.2,-0.5) circle (0pt) node [left]{\small{$W$}};
\fill (0,0) circle (2pt) node [above]{\small{$x$}};
\fill (-2.7,0) circle (0pt) node [above] {\tiny{Alex. boundary}};
\fill (-3.5,-1) circle (0pt) node [below] {\small{$N_0$}};
\fill (0,0.8) circle (0pt) node [above] {\tiny{1)}\,\,\small{$x\in N_0^1\cap N_0^{\rm reg} $}};

\draw [-, thick] (7,0.7) to [out=-45, in=180] (9,0);
\draw [-, thick] (9,0) to [out=00, in=180] (10,0.2);
\draw [-, thick] (10,0.2) to [out=00, in=180] (11,0);
\draw [-, thick] (11,0) to [out=00, in=180] (12,0.2);
\draw [-, thick] (12,0.2) to [out=00, in=180] (13,0);
\draw [-, thick] (13,0) to [out=00, in=225] (15,0.7);
\draw [-, thick] (7,-0.7) to [out=45, in=180] (9,0);
\draw [-, thick] (9,0) to [out=00, in=180] (10,-0.2);
\draw [-, thick] (10,-0.2) to [out=00, in=180] (11,0);
\draw [-, thick] (11,0) to [out=00, in=180] (12,-0.2);
\draw [-, thick] (12,-0.2) to [out=00, in=180] (13,0);
\draw [-, thick] (13,0) to [out=00, in=135] (15,-0.7);
\draw [-, thick] (7, -0.7) -- (7,0.7);
\draw [-, thick] (7.4,-0.4) -- (7.4,0.4);
\draw [-, thick] (7.8,-0.2) -- (7.8,0.2);
\draw [-, thick] (15, -0.7) -- (15,0.7);
\draw [-, thick] (14.6,-0.4) -- (14.6,0.4);
\draw [-, thick] (14.2,-0.2) -- (14.2,0.2);
\draw [-, thick] (9.8,-0.2) -- (9.8,0.2);
\draw [-, thick] (10.2,-0.2) -- (10.2,0.2);
\draw [-, thick] (11.8,-0.2) -- (11.8,0.2);
\draw [-, thick] (12.2,-0.2) -- (12.2,0.2);

\fill (15.2,0) circle (0pt) node [right]{\small{$W$}};
\fill (11,0) circle (2pt) node [above]{\small{$x$}};
\fill (15,-0.7) circle (0pt) node [below] {\small{$N_0$}};
\fill (11,0.8) circle (0pt) node [above] 
{\tiny{2)}\,\,\small{$x\in N_0^2\cap N_0^{\rm reg}$}};

\draw [->, thick]  (2,-2) -- (3.5,-3.5);
\draw [->, thick] (8,-2) -- (6.5,-3.5);
\fill (2.9,-2.5) circle (0pt) node [right] 
{\tiny{almost isometries}};

\draw [-, thick] (1,-5) -- (9,-5);
\draw [-, thick] (1,-3.8) to [out=-45, in=180] (3,-5);
\draw [-, thick] (3,-5) to [out=00, in=180] (4,-4.6);
\draw [-, thick] (4,-4.6) to [out=00, in=180] (5,-5);
\draw [-, thick] (5,-5) to [out=00, in=180] (6,-4.6);
\draw [-, thick] (6,-4.6) to [out=00, in=180] (7,-5);
\draw [-, thick] (7,-5) to [out=00, in=225] (9,-3.8);
\draw [-, thick] (1,-5) -- (1,-3.8);
\draw [-, thick] (1.5,-5) -- (1.5,-4.3);
\draw [-, thick] (1.9,-5) -- (1.9,-4.65);
\draw [-, thick] (9,-5) -- (9,-3.8);
\draw [-, thick] (8.5,-5) -- (8.5,-4.3);
\draw [-, thick] (8.1,-5) -- (8.1,-4.65);
\draw [-, thick] (3.8,-5) -- (3.8,-4.6);
\draw [-, thick] (4.2,-5) -- (4.2,-4.6);
\draw [-, thick] (5.8,-5) -- (5.8,-4.6);
\draw [-, thick] (6.2,-5) -- (6.2,-4.6);

\fill (1,-4.4) circle (0pt) node [left]{\small{$G_g$}};
\fill (5,-5) circle (2pt) node [below]{\tiny{$0$}};
\fill (8,-5) circle (0pt) node [below]{\small{$B^{\R^{m-1}}(0,\e)$}};

\end{tikzpicture}
\end{center}
\caption{Almost isometries at regular thin points}
\end{figure}
\pbig

\begin{rem} \label{rem:almost-graph}
In the situation of Theorem \ref{thm:almostisometry(intro)},
we have the following. 
\begin{enumerate}
\item If $x$ is a regular thin point contained in $N_0^1$,
then  $\{ t=0\}$   
corresponds to $N_0\cap W$
and $\{ t=g(y)>0\}$ corresponds to
the local Alexandrov boundary points
in $W\setminus N_0\,;$
\item If $x$ is a regular thin point in $N_0^2$, 
then 
$\{ t=0, t=g(y)\}$ (resp. $\{ g(y)=0\}$) corresponds to  $N_0\cap W$ (resp. to $N_0^2\cap W$).
\end{enumerate}
\end{rem}  

\pmed
\n
{\bf Local contractibility.}\,
We begin with the local contractibility of $N$.
 We say that a space $V$ is {\it $(C,\e)$-Lipschitz strongly 
contractible} to a point $x_0\in V$ if
there exists a $C$-Lipschitz homotopy 
$F:V\times [0,\e]\to V$ such that 
$F(\,\cdot\,,0)={\rm id}_V$, 
$F(\,\cdot\,,\e)= x_0$ and $F(x_0,t)=x_0$ for all $t\in [0,\e]$ (see \cite{FMY}).
From here on, we consider the $L_1$-distance on 
the product $V\times [0,\e]$.
 
By Perelman's local stability theorem 
(\cite{Pr:alexII}), any Alexandrov space
is locally contractible.
Based on concave functions constructed in 
\cite{Per} and \cite{Kap}, the Lipschitz 
contractibility of Alexandrov spaces
was shown in \cite{MY:SLC}.
In our situation, 
based on these results, we obtain 
the following.

\begin{thm} \label{thm:contractible}
The limit space  $N$ is locally Lipschitz contractible.  More precisely,  
for any $x\in N$ and $\zeta>0$, there exists $\e_{x,\zeta}>0$ such that
for  any $0<\e\le \e_{x,\xi}$, there exists a  closed domain $V$ 
containing $B(x,\e)$  that is $(1+\zeta, \e)$-Lipschitz strongly contractible to $x$.
\end{thm}

\begin{rem}\label{rem:size-ball}
In Theorem \ref{thm:contractible}, from the 
condition stated there, we have 
\[
      B^N(x,\e)\subset V\subset B^N(x, (1+\zeta)\e).
\]
\end{rem}

\pmed\n
{\bf Global stabilities.}\, Now we state 
our results on the global stability:
the Lipschitz homotopy stability and the volume convergence.
\pmed
\n
{\bf Lipschitz homotopy stability.}\,

For Alexandrov spaces, the quantitative 
 Lipschitz homotopy stability was obtained in 
 Fujioka-Mitsuishi-Yamaguchi \cite{FMY}.
First we discuss it in our situation.

We often abbreviate  $\ca M(n,\kappa,\nu,\lambda,d)$ as  $\ca M$ for simplicity. 
Let $\overline{\ca M}$ denote the closure of
$\ca M$
with respect to the Gromov-Hausdorff distance.
For each integer $1\le m\le n$ and $v>0$, we denote by $\ol{\ca M}(m)$ (resp.
$\ol{\ca M}(m,v)$) the set of all
$N\in\ol{\ca M}$ such that 
$\dim N=m$ (resp. $\dim N=m$ and $\ca H^m(N)>v$).

For $C>0$, the {\it $C$-Lipschitz homotopy 
distance} $d_{\text{$C$-$LH$}}(X,X')$ (see \cite{FMY}) between 
metric spaces $X$ and $X'$
is defined as 
the infimum of $\epsilon>0$ such that there are $\epsilon$-approximations $f:X\to X'$ and $g:X'\to X$ that are $C$-Lipschitz, together with $C$-Lipschitz maps $F:X\times [0,\epsilon]\to X$ and $G:X'\times [0,\epsilon]\to X'$ satisfying
\[
F(\,\cdot,0)=g\circ f,\quad F(\,\cdot,\epsilon)={\rm id}_X,\quad G(\,\cdot,0)=f\circ g,\quad G(\,\cdot,\epsilon)={\rm id}_{X'}.
\]
 We also say that $X$ and $X'$ have \emph{the same $(C,\e)$-Lipschitz homotopy type}
if $d_{\text{$C$-$LH$}}(X,X')<\e$.
Note that this is not an equivalence relation.


\begin{thm} \label{thm:Lip-homotopy}
For any $1\le m\le n$ and $v>0$, there exists
$C>1$ depending on $n,\kappa,\nu,\lambda,d,
m,v$ satisfying the following$:$
For arbitrary $N,N'\in \ol{\ca M}(m,v)$,
if the Gromov-Hausdorff distance $d_{GH}(N,N')<C^{-1}$, then we have 
\[
     d_{\text{$C$-$LH$}}(N,N')<Cd_{GH}(N,N').
\]
\end{thm}

\begin{cor} \label{cor:finiteness}
There exists
$C>1$ depending on $n,\kappa,\nu,\lambda,d,m,v$ such that 
for any $\e>0$, the set $\ol{\ca M}(m,v)$ contains only finitely many $(C,\e)$-Lipschitz homotopy types.
\end{cor}

\begin{rem} \label{rem:LP-stability}
(1)\, 
It should be noted that in Theorem \ref{thm:Lip-homotopy}, $N$ and $N'$ are not necessarily homeomorphic to each other, and that $\ol{\ca M}(m,v)$
may contain infinitely many 
homeomorphism classes.

\par\n
(2)\,
Theorem \ref{thm:Lip-homotopy} is new even for the noncollapsing case when $m=n$.
More precisely, 
even if we fix an element $N\in\ol{\ca M}$ with 
$\dim N=n$, \cite{wong0}  yields only  the 
homotopy type stability (see also Proposition \ref{prop:homotopy=}).
\par\n 
(3)\,
More strongly, if we fix a Riemannian manifold $N$ in 
$\ca M(n,\kappa,\nu,\lambda,d)$, we 
obtain a diffeomorphism stability
(see Proposition \ref{prop:w-stability}).
\end{rem}
\pmed

\n
{\bf Local Lipschitz homotopy stability.}\,
Here we describe the general local structure results
in terms of Lipschitz homotopy stabiity.

The following result shows that any small
extrinsic metric sphere 
around each point of $X$ is Lipschitz homotopy equivalent to
the space of directions at the point. 
More explicitly, we have
\pmed 
\begin{thm}\label{thm:metric-sphere}
For any $1\le m\le n$ and $v>0$, there exists
$C$ depending on $n,\kappa,\nu,\lambda,d,
m,v$ satisfying the following$:$
For any $N\in\overline{\ca M}(m)$ and $p\in N$ with 
\beq\label{eq:volume-Sigma}
\begin{cases}
v\le \ca H^{m-1}(\Sigma_p(N))<\infty,
           \quad &\text{if $p\in{\ca Tk}\cup {\rm int} N$} \\
 v\le \ca H^{m-2}(\Sigma_p(N))<\infty,
           \quad &\text{if $p\in {\ca Tn}$},
\end{cases}
\eeq
there exists $r_p>0$ such that for each $0<r\le r_p$,
we have 
\[
      d_{\text{$C$-$LH$}}(S^{N^{\rm ext}}(p,r)/r, \Sigma_p(N))<\tau_p(r),
\]
where $S^{N^{\rm ext}}(p,r)/r$ denotes the rescaled metric $\frac{1}{r}S^{N^{\rm ext}}(p,r)$ equipped with the intrinsic metric.
\end{thm}

\psmall
Next we discuss the Lipschitz homotopy stability 
of small punctured balls at a given point of $N$.
For $p\in N$, we set 
\begin{gather*}
  B^{N^{\rm ext}}_*(p,r):=B^{N^{\rm ext}}(p,r) \setminus \{ p\}, \quad T_{p}^rN_*:=B^{T_p(N)}(o_p,r)\setminus \{ o_p\}.
\end{gather*}
\psmall  
\begin{thm}\label{thm:punctured-ball}
For any $1\le m\le n$ and $v>0$, there exists
$C$ depending on $n,\kappa,\nu,\lambda,d,
m,v$ satisfying the following$:$
For any $N\in\overline{\ca M}(m)$ and $p\in N$
satisfying \eqref{eq:volume-Sigma},
there exists $r_p>0$ such that for each $0<r\le r_p$
we have 
\[
      d_{\text{$C$-$LH$}}(B^{N^{\rm ext}}_*(p,r)/r, T_p^1N_*)<\tau_p(r),
\]
where $B^{N^{\rm ext}}_*(p,r)/r$ denotes the rescaled metric $\frac{1}{r}B^{N^{\rm ext}}_*(p,r)$.
\end{thm}
\pmed

\begin{rem}\label{rem:Alex-CBA}
For Alexandrov spaces $N$, Perelman's stability
theorem shows that $(B^N(p, r),p)$ is homeomorphic 
to $(T_p^1N,o_p)$.
For GCBA (geodesically complete, curvature bounded above) spaces $N$, Lychack and Nagano 
\cite[Theorem 1.12] {NL:geodesically}
proved that $S^N(p,r)$  is homotopy equivalent to 
$\Sigma_p(N)$, and Kramer \cite{Kram} proved that
$B^N_*(p, r)$  is homotopy equivalent to 
$\Sigma_p(N)$.

Our proofs show that 
Theorems \ref{thm:metric-sphere} and 
\ref{thm:punctured-ball} hold true for both
Alexandrov spaces and GCBA spaces
and are new even in those cases.  See also \cite[Remark 4.6]{FMY}.
\end{rem}

\begin{prob} \label{prob:ext->int}
In Theorems \ref {thm:metric-sphere} and 
\ref{thm:punctured-ball}, we considered the extrinsic metric spheres and the extrinsic punctured balls.
This is because of lack of knowledge about the intrinsic geometry of the limit space $N$.
For instance, we do not know if a 
$N$-minimal geodesic has direction at the initial point
in some cases (see \cite[Problem 5.7]{YZ:part1}).
It would be an interesting 
problem to determine whether we can replace them by the intrinsic metric spheres and the intrinsic punctured balls.
\end{prob} 

\pbig
\n
{\bf Volume convergence.}\,
Finally  we exhibit a result on   
the convergence of  Hausdorff measures of 
 $\pa M$ in $\ca M(n,\kappa,\nu,\lambda,d)$.
We denote by $\mu^m_N$ the $m$-dimensional Hausdorff measure of $N$.

Let a sequence $M_i \in \ca M(n,\kappa,\nu,\lambda,d)$ converge to a 
compact geodesic space $N$ with respect to the  Gromov-Hausdorff  distance. 
We easily verify   
$\lim_{i \to\infty}\mu^n_{M_i}=   \mu^n_N$
with respect to the measured Gromov-Hausdorff
topology (see Lemma \ref{lem:vol-conv}).

\begin{thm} \label{thm:vol-conv}
Under the above situation,  we have 
\[
\lim_{i \to\infty} \mu^{n-1}_{\pa M_i}=
   \mu^{n-1}_{N_0} +\iota_*(\mu^{n-1}_{N_0^2}),
\]
where $\iota:N_0^2\to N_0$ denotes the inclusion.
\end{thm}
\pmed\n

The organization of the paper is as follows.

In section \ref{sec:prelim},  we first recall basic materials on 
Alexandrov spaces and Wong's extension procedure for 
Riemannian manifold with boundary. 

 In Section \ref{sec:non-inradius}, 
we recall from \cite{YZ:inradius}
and \cite{YZ:part1} the basic properties of  the limit spaces $N$ and the gluing map $\eta_0:  C_0\to N_0$, which will be 
needed in later sections.

In Section \ref{sec:infinitesimal},
we exhibit some 
results related with the infinitesimal structure 
of the limit spaces obtained in 
\cite{YZ:part1} that will be 
needed in later sections.

We begin the study of the local 
structures of $N$ and $N_0$ from Section \ref{sec:almost-stability}.
 In Section \ref{sec:almost-stability},
we prove Theorems  \ref{thm:regular-ballX(intro)}  and
\ref{thm:almostisometry(intro)}.
Theorem \ref{thm:contractible} is proved 
in Section \ref{sec:loc=cont}, 
where we begin the discussion on the  Lipschitz homotopy.

In  Section \ref{sec:LHS}, we discuss the
Lipschitz homotopy stability  and prove
Theorem \ref{thm:Lip-homotopy}.
Theorem \ref{thm:conv-cover} is a key result,
and is applied in the next section, too.

  In Section \ref{sec:LLHS},
we prove Theorems \ref {thm:metric-sphere}  and \ref{thm:punctured-ball}.
We also discuss a diffeomorphism stability there.

In  Section \ref{sec:behavior}, we discuss 
 the convergence of volumes of 
boundaries 
and an obstruction to collapse.
 Theorem \ref{thm:vol-conv} is proved there.

\psmall\n 
{\bf Acknowledgements}. The authors would like to thank Xiaochun Rong and Ayato Mitsuishi  for their valuable comments on the first version of our  paper.

 
\setcounter{equation}{0}
\section{Preliminaries}\label{sec:prelim}

\par\n
{\bf Notations and conventions} 
\psmall

 Throughout the paper, we use the same notations and conventions as in \cite{YZ:part1}. 
In particular, the following ones are often used:
\begin{itemize}
\item For two points $x$ and $y$ in a metric space
$X$, $d(x,y)$, $|x,y|$ or $|x,y|_X$ denote the distance 
between them$\,;$
\item For a point $x$ in a metric space
$X$, $B(x,r)$ or $B^X(x,r)$ denotes the closed metric $r$-ball around $x$.
Similarly, $S(x,r)$ or $S^X(x,r)$ denotes the metric  $r$-sphere around $x\,;$
\item $\gamma_{x,y}^X$ denotes a shortest curve
(an $X$-minimal geodesic) in a metric space $X$
of unit speed parametrized on $[0,|x,y|]\,;$
\item For a metric space $X=(X,d)$ and $\lambda>0$,
$\lambda X$ denotes the rescaled metric space
$(X,\lambda d)\,;$
\item For a closed subset $A$ of a geodesic space $X$, $A^{\rm int}$ denotes the intrinsic metric (length metric) on $A$ induced from $X$, and $A^{\rm ext}$
 denotes the extrinsic metric $A$ induced from $X$
that is the restriction of the distance of $X$ to $A\,;$
\item $\angle(\,\,\,,\,\,)$ denotes the angle$\,;$
\item $o_i$  denotes a sequence of positive numbers
satisfying  $\lim_{i\to\infty}o_i=0\,;$
\item $\tau(\e_1,\ldots,\e_n)$ (resp.  $\tau_{a_1,\ldots,a_m}(\e_1,\ldots,\e_n)$)  denotes a function 
satisfying  $\lim_{\e_1,\ldots,\e_n\to 0}
\tau(\e_1,\ldots,\e_n)=0$
(resp. a function depending on $a_1,\ldots,a_m$ 
satisfying  $\lim_{\e_1,\ldots,\e_n\to 0}
\tau_{a_1,\ldots,a_m}(\e_1,\ldots,\e_n)=0$ for fixed 
$a_1,\ldots,a_m$).
\end{itemize}  
Additionally, we use the following.
\pmed\n
$\bullet$\,For a closed subset $K$ of a metric space $X$, let  $A(K,r)$ or $A^X(K,r)$ denote
the  closed metric annulus around $K$ defined as 
$A(K,r,R)=B(K,R)\setminus\mathring{B}(K,r)$.

%
%

In this section, we provide some preliminary
arguments on Alexandrov spaces,  the GH-convergence, gradient flows and the extension of manifolds with
boundary.

\pmed\n
{\bf Alexandrov spaces.}\,
For the basic of Alexandrov spaces, 
we refer the readers to \cite{BGP},  \cite{BBI}, \cite{AKP}.
Let $X$ be an $m$-dimensional Alexandrov space
with curvature $\ge\kappa$.
For closed subsets $A$, $B$ of $X$ and $x, y\in X\setminus(A\cup B)$, consider a geodesic triangle on  the $\kappa$-plane $M_{\kappa}^2$
having the side-length $(|A,x|, |x,B|, |B, A|)$ whenever it exists.
Then $\tilde\angle AxB$ denotes the angle of this comparison triangle at the vertex
corresponding to $x$.

A point $x\in X$ is called {\em regular} if the space of directions, $\Sigma_x(X)$, at $x$ is isometric to the unit sphere  $\mathbb{S}^{m-1}$.
Otherwise we call $x$ {\em singular}. We denote by $X^{\rm reg}$ (resp. $X^{\rm sing}$)
the set of all regular points (resp. singular points) of $X$. 

For  $\delta>0$ and $1\le k\le m$, a system of $k$ pairs of points, $\{(a_i,b_i)\}_{i=1}^k$ is called a
$(k,\delta)$-\emph{strainer} at $x\in X$ if it satisfies
\begin{align*}
   \tilde\angle a_ixb_i > \pi - \delta, \quad &
                \tilde\angle a_ixa_j > \pi/2 - \delta, \\
      \tilde\angle b_ixb_j > \pi/2 - \delta,\quad &
                \tilde\angle a_ixb_j > \pi/2 - \delta,
\end{align*}
for all $1\le i\neq j\le k$.
The constant $\min_{1\le i,j\le k}\{|a_i,x|, |b_j,x|\}$
is called the \emph{length} of the strainer.
If  $x\in X$ has a $(k,\delta)$-strainer, then 
it is called $(k,\delta)$-strained.

The following lemma is useful in the geometry of
Alexandrov spaces.

\begin{lem}$($\cite[Lemma 5.6]{BGP}$)$
\label{lem:comparison}
Let $x$ be $(1,\delta)$-strained with $(1,\delta)$-strainer $(a,b)$ with length $\ell$. 
For any $y,z\in B(x,\e)$ we have
\[
  |\wangle ayz -\angle ayz|<\tau(\delta,\e/\ell),\quad
  |\wangle ayz +\wangle byz -\pi|<\tau(\delta,\e/\ell).
\]
\end{lem}  

If  $x\in X$  is $(m,\delta)$-strained, it is called 
a {\em $\delta$-regular} point.
We denote by $X_\delta^{\rm reg}$ the set  of all 
$\delta$-regular points of $X$.

\begin{thm} $($\cite[Theorem 9.4]{BGP}$)$ \label{thm:almost-isometry}
Let $x$ be $\delta$-regular with 
$(m,\delta)$-strainer $\{ (a_i,b_i)\}_{i=1}^m$.
If $\delta \ll 1/m$, then the map $f:\mathring{B}(x,\e)\to \R^m$ 
defined by 
\[
         f(y)=(|a_1,y|, \ldots, |a_m,y|)
\]
provides a $\tau(\delta,\e/\ell)$-almost isometry
onto an open subset of $\mathbb R^m$, where $\ell:=\min_{1\le i\le m}\{|a_i,x|, |b_i,x|\}$.
\end{thm}

When $X$ has nonempty boundary $\partial X$,
the double $D(X)$ of $X$ is  also
an Alexandrov space with curvature $\ge \kappa$
(see \cite{Pr:alexII}). 
A boundary point $x\in\partial X$ is called {\em $\delta$-boundary regular} if $x$ is $\delta$-regular
in $D(X)$.

In Section \ref{sec:almost-stability}, we need the following result on the dimension of the interior singular point sets.

\begin{thm} [\cite{BGP}, \cite{Per}, cf. \cite{OS}]   \label{thm:dim-sing} 
Let $m:=\dim X$. Then we have 
$$
            \dim_H(X^{\rm sing}\cap {\rm int}X) \le m-2,  \,\,\dim_H(\partial X)^{\rm sing} \le m-2,
$$  
where $(\partial X)^{\rm sing}:=D(X)^{\rm sing}\cap \pa X$.
\end{thm}  
 \pmed

%
%
%

\pmed\n 
{\bf Gromov-Hausdorff convergence.}\,
For metric spaces $X$ and $X'$,
a not necessarily continuous map $f:X\to X'$ is called
$\e$-approximation if 
\begin{itemize}
\item $||f(x), f(x')|-|x,x'|| < \epsilon$ for all $x, x' \in X\,;$ 
\item $f(X)$ is $\e$-dense in $Y$.
\end{itemize} 
The Gromov-Hausdorff distance $d_{GH}(X,Y)<\e$
iff there are $\e$-approximations $\varphi:X\to Y$ and 
$\psi:Y\to X$ (see \cite{GLP} for more details).
 
For compact subsets $A_1,\ldots,A_k\subset X$
and $A'_1,\ldots,A'_k\subset X'$,
the Gromov-Hausdorff distance 
$d_{GH}((X,A_1,\ldots,A_k),(X',A'_1,\ldots,A'_k))<\e$ 
iff there are $\e$-approximations $\varphi:X\to X'$ and 
$\psi:X'\to X$ such that the restrictions
$\varphi|_{A_i}$ and $\psi|_{A'_i}$ give
$\e$-approximations between $A_i$ and $A'_i$ for any
$1\le i\le k$.

In the geometry of Alexandrov spaces,
the following topological stability theorem by Perelman plays a crucial role. 

\begin{thm} [\cite{Pr:alexII}, \cite{Per}, cf.\cite{Kap:stab}]      \label{thm:stability}
If a sequence $X_i$ of $m$-dimensional compact Alexandrov spaces 
with curvature $\ge \kappa$ Gromov-Hausdorff converges to an $m$-dimensional compact Alexandrov space $X$,
then $X_i$  is homeomorphic to $X$ for large enough $i$.
\end{thm}

A direction of a minimal geodesic from $p$ to 
a closed subset $A$ of an  Alexandrov space $X$ is  denoted by $\uparrow_p^A$.
The set of all directions $\uparrow_p^A$ 
from $p$ to $A$ is denoted by 
$\Uparrow_p^A$.
For  an open subset $U$ of $X$,  
let $f=(f_1,\ldots, f_k):U\to\R^k$ be the map defined by $f_i(x)=|A_i,x|$, where 
$A_i$  are closed subsets of $X$. For $c,\e>0$, 
the map $f$ is called {\it $(c,\e)$-regular} at $p\in U$ (see \cite{Pr:alexII})
if there is a point $w\in X$ satisfying
\begin{itemize}
\item $\angle(\Uparrow_p^{A_i},\Uparrow_p^{A_j})>\pi/2-\e\,;$
\item $\angle(\Uparrow_p^{w},\Uparrow_p^{A_i})>\pi/2+c$
\end{itemize}
for all $1\le i\neq j\le k$.

The following theorem is a respectful version of 
Theorem \ref{thm:stability}, which is a special case of 
\cite[Theorem 4.3]{Pr:alexII}.

\begin{thm}[\cite{Pr:alexII}] 
\label{thm:stability-respectful}
For an Alexandrov space $X$
and closed subsets $A_1,\ldots,A_k, B$ of $X$,
let $f_\alpha(x):=|A_\alpha,x|$, $h(x):=|B,x|$ 
and  let $I_\alpha$ be a closed interval contained in $f_\alpha(X)$\, $(1\le\alpha\le k)$, and set
$\mathbb I^k:=I_1\times\cdots\times I_k$.
Let $K$ be a compact subset of $f^{-1}(\mathbb I^k)$, and 
suppose that 
\begin{itemize}
\item 
$f:=(f_1, \dots, f_k)$ is $(c,\delta)$-regular on  $f^{-1}(\mathbb I^k)\,;$
\item $(f,h)$  is $(c,\delta)$-regular on $K$.
\end{itemize}
If 
$\delta$ is small enough compared with 
$c$ and $f^{-1}(\mathbb I^k)$,
then there exists $\e=\e_{A_1,\ldots,A_k,B}>0$ satisfying the following:
For an Alexandrov space $X'$ with curvature $\ge -1$ and the same dimension as $X$, and for closed
subsets
$A_1',\ldots, A'_k,B'$ of $X'$, suppose that the Gromov-Hausdorff distance between $(X,A_1,\dots, A_k,B)$ and $(X', A'_1, \dots, A'_k,B')$ is less than $\e$. Then 
there exists a homeomorphism 
$\varphi : f^{-1}(\mathbb I^k) \to (f')^{-1}(\mathbb I^k)$ such that 
\begin{enumerate}
\item $f' \circ \varphi = f$ on $f^{-1}(\mathbb I^k)\,;$
\item $(f',h') \circ \varphi = (f,h)$ on $K$, 
\end{enumerate}
where
$f'$ and $h'$ are the distance maps
defined by $A_1',\ldots,A_k'$ and $B'$ as 
$f$ and $h$.
\end{thm}

\pmed\n
{\bf Gradient flows.}\,
We shortly recall the properties of the gradient flows of semiconcave functions on  Alexandrov spaces. See \cite{Pet} for more details.

In what follows, for simplicity we assume that functions are defined on the whole space.
Let $X$ be an Alexandrov space and $f:X\to\mathbb R$.

If $X$ has no boundary, $f$ is said to be \textit{$\lambda$-concave} if its restriction to any unit-speed minimal geodesic $\gamma(t)$ is $\lambda$-concave, i.e., $f\circ\gamma(t)-(\lambda/2)t^2$ is concave.
If $X$ has boundary, we require the same condition for the tautological extension of $f$ to the double  $D(X)$.
A \textit{semiconcave} function is locally a $\lambda$-concave function, where $\lambda$ depends on each point.
A \textit{strictly concave} function is a $\lambda$-concave function with $\lambda<0$.

A typical example of a semiconcave function is a distance function $|A,\cdot\,|$ from a closed subset $A\subset X$.
Note that the concavity of $|A,\cdot\,|$ at $x\in X\setminus A$ is bounded above by a constant depending only on a lower bound of $|A,x|$ and the lower curvature bound of $X$.

Let $f:X\to\mathbb R$ be a $\lambda$-concave function and $p\in X$.
The directional derivative $df=d_pf$ is a $0$-concave function defined on the tangent cone $T_p(X)$.
The \textit{gradient} $\nabla_pf\in T_p(X)$ of $f$ at $p$ is characterized by the following two properties:
\beq
\begin{aligned}\label{eq:gradient}
\begin{cases}
 &df(v)\le\langle\nabla_pf,v\rangle \quad\text{for any $v\in T_p(X)$}\\
 & df(\nabla_pf)=|\nabla_pf|^2.
\end{cases}
\end{aligned}
\eeq

where $|\cdot|$ denotes the norm, i.e., the distance to the vertex $o_p$ of the cone and $\langle\cdot,\cdot\rangle$ is the scalar product defined by the norm and the angle at the vertex as in Euclidean plane.
The concavity of $df$ guarantees the existence and uniqueness of the gradient.

The \textit{gradient curve} of a semiconcave function $f$ is a curve whose right tangent vector is unique and equal to the gradient of $f$.
For any $\lambda$-concave function $f$, the gradient curve $\Phi(x,t)$ starting at $x\in X$ exists and is unique for all $t\ge0$. 
The map $\Phi:X\times[0,\infty)\to X$ is called the \textit{gradient flow} of $f$.

The following two results are useful in the present paper.

\begin{prop}[{\cite[2.1.4]{Pet}}]\label{prop:lip}
Let $\Phi:M\times[0,\infty)\to M$ be a the  gradient flow of a $\lambda$-concave function
$f:X\to\mathbb R$. 
Then $\Phi_t:=\Phi(\cdot,t):X\to X$ is $e^{\lambda t}$-Lipschitz for any $t\ge 0$.
In particular, if $f$ is $C$-Lipschitz, the restriction of $\Phi$ to $M\times[0,T]$ is $\max\{C,e^{\lambda T},1\}$-Lipschitz.
\end{prop}

\begin{lem} $($\cite[Lemma 2.5]{FMY}$)$
\label{lem:grad-df}
Let $f$ be a $(-\lambda)$-concave function on $M$, where $\lambda>0$.
If $p\in M$ is the unique maximum point of $f$, then for any $x\in M\setminus\{p\}$, we have
\[|\nabla_xf|\ge df(\uparrow_x^p)\ge\lambda|px|/2.\]
\end{lem}
\begin{proof}
The $(-\lambda)$-concavity implies
\[df(\uparrow_x^p)\ge\frac{f(p)-f(x)+\lambda|px|^2/2}{|px|}.\]
\end{proof}

\pmed\n
{\bf Manifolds with boundary and gluing.}\,
A Riemannian manifold with boundary is not necessarily an Alexandrov space
unless $\Pi_{\pa M}\ge 0$.
Assuming 
\beq \label{eq:curv-assum2}
 K_M \ge \kappa, \,\,\,
 |\Pi_{\partial M}|\le \lambda,
\eeq
Wong (\cite{wong0}) carried out a gluing of $M$ and 
 warped  cylinders along their boundaries  in such a way that 
the resulting manifold  becomes an Alexandrov space having totally geodesic boundary,
based on the gluing theorem by Kosovski$\breve{i}$ \cite{Kos}.
As remarked in \cite{YZ:part1},
under the new weaker curvature conditions
in 
\eqref{eq:curv-assum},
we  can still carry out the gluing construction
without  additional argument.


Suppose $M$ is an $n$-dimensional complete  Riemannian manifold satisfying 
\eqref{eq:curv-assum}.
Then for arbitrarily  $t_0>0$ and  $0<\e_0 < 1$
there exists a monotone non-increasing function  $\phi: [0,t_0]\to (0,1]$ 
satisfying
\beq       \label{eq:phi}
\left\{
 \begin{aligned}
     \phi''(t)+K\phi(t)\leq0, \,\,\,&\phi(0)=1, \,\,\,\phi(t_0)=\e_0,\\
     -\infty<\phi'(0)\leq-\lambda,\,\, \,&\phi'(t_0)=0,
 \end{aligned}
   \right.
\eeq
for some constant  $K=K(\lambda,\ve_0,t_0)$.
Consider the warped product $C_M:= [0, t_0]\times_{\phi}\partial M$ with the metric on  $[0, t_0]\times\partial M$ defined by 
\[
   g(t,x)=dt^2+\phi^2(t)g_{\partial M}(x).
\]
From the construction, we have 
\beq  \label{eq:gluing-cond}
\begin{cases}
\begin{aligned}
&K_M\ge\kappa, \,\,\, K_{C_M}\ge c(\nu,\lambda,\e_0, t_0), \\ 
&II^M_{\pa M}+II^{C_M}_{\{ 0\}\times \pa M}\ge 0
\,\,\, \text{along \,$\pa M=\{ 0\}\times \pa M$}.
\end{aligned}
\end{cases}
\eeq
It follows from \eqref{eq:gluing-cond} 
and \cite{Kos} that the glued space
\[
                       \tilde M := M\cup_{\pa M} C_M
\]
along $\partial M$ and  $\{ 0\}\times\partial M$ satisfies the following.

\begin{prop} [\cite{wong0}] \label{prop:extendAS}
For  any $M$  satisfying \eqref{eq:curv-assum}, 
the following holds$:$
\begin{enumerate}
 \item $\tilde M$  is an  Alexandrov space with curvature 
  $\ge \tilde{\kappa}$,  
         where $\tilde{\kappa}=\tilde{\kappa}(\kappa,\nu, \e_0,\lambda,t_0);$
 \item the extrinsic metric $M^{\rm ext}$ of $M$ in $\tilde M$ is $L$-bi-Lipschitz homeomorphic to $M$ 
         for the uniform constant $L=1/\e_0;$
\item  $\diam(\tilde M)\le \diam(M) + 2t_0$.
\end{enumerate}
\end{prop}

\begin{prop} [\cite{wong0}]\label{prop:cpn-diam}
For any $M\in\ca M(n,\kappa,\nu,\lambda,d)$,
we have the following:
\begin{enumerate}
 \item There exists a constant $D=D(n,\kappa,\nu,\lambda,d)$ such that  any component $\pa^\alpha M$  of $\pa M$ has intrinsic diameter bound
    \[
             \diam((\pa^\alpha M)^{\rm int})\leq D;
    \]
 \item $\partial M$ has at most $J$ components, where $J = J(n, \kappa,\nu, \lambda, d)$.
\end{enumerate}
\end{prop}

\pmed


\setcounter{equation}{0}

\section{Basic structure of limit spaces} \label{sec:non-inradius}

 From here on, we consider the situation that 
$M_i \in \ca M(n,\kappa,\nu,\lambda, d)$  
 non-inradius collapses/converges to a  geodesic space $N$. 
Let $\tilde M_i$ be the extension of $M_i$ 
constructed in the previous section.
Passing to a subsequence, we have the following convergences  in the sense of Section \ref{sec:prelim}:
\begin{enumerate}
 \item $(M_i, \pa M_i)$ converges to 
$(N, N_0)\,;$ 
\item $(\tilde M_i, M_i,\partial M_i)$ converges to a triple $(Y,X,X_0)$ with $Y\supset X\supset X_0$,
%
where $Y$ is an Alexandrov space with curvature $\ge \tilde\kappa$.
\end{enumerate}
 
\pmed\n
{\bf Limit spaces.}\, 
 First we recall 
basic properties of $X$ and $Y$ and the relations between them, established in 
\cite{wong0}, \cite{YZ:inradius}, \cite{YZ:part1}
except Lemma \ref{lem:dX-est}.

 In  view of \eqref{eq:gluing-cond},
passing to a subsequence, we may assume that $C_{M_i}$ converges 
to some compact Alexandrov space $C$ with curvature $\ge c(\nu,\lambda,\ve_0, t_0)$. 
Obviously we have  
\beq \label{eq:defC}
 C=[0, t_0]\times_{\phi}C_0, 
   \,\, \, \, C_0:=\lim_{i\to\infty} (\pa M_i)^{\rm int},
\eeq
where  $C_0$ is an Alexandrov space with curvature $\ge \nu$.

For simplicity we set  
\beq \label{eq:defC0}
    C_0=\{0\}\times C_0,\,\,\,\,  
    C_t:=\{t\}\times C_0\subset C,
\eeq
and 
\[
       {\rm int}\, X:= X\setminus X_0.
\]
Let $C_{M_i}^{\rm ext}$ denote the extrinsic metric 
of $C_{M_i}$
induced from $\tilde M_i$, which is defined as 
the restriction of the metric of $\tilde M_i$.
Since the identity map $\iota_i: C_{M_i}\to C_{M_i}^{\rm ext}$ is 1-Lipschitz,
we have a 1-Lipschitz map $\eta: C\to Y$ in the limit.
Note that $\eta:C\to Y\setminus {\rm int}\, X$ is surjective.

From now on,  we consider the gluing map 
\[
\eta_0: = \eta|_{\{ 0 \}\times C_0} : C_0\to X_0,
\]
which is a surjective $1$-Lipschitz map with respect to the
intrinsic metrics. 
Note that $\eta_0$ is the limit of $1$-Lipschitz map $\eta_{0,i}:(\pa M_i)^{\rm int} \to (\pa M_i)^{\rm ext}$, where $(\pa M_i)^{\rm ext}$ denotes the extrinsic  
metric induced from $\tilde M_i$.

%

\begin{lem} [{\cite[Lemma 3.1]{YZ:inradius}}]  \label{lem:loc-isom}
The map $\eta:C\setminus C_0\to Y\setminus X$ is a bijective local isometry.
\end{lem}

Let $\tilde\pi:C\to C_0$ and $\pi:Y\to X$ be the canonical projections. 
To be precise, $\pi$ is defined as
\[
          \pi(y) = \eta_0\circ \tilde \pi\circ \eta^{-1}(y), \quad (y\in Y\setminus X).
\]   
For any $p\in C_0$ with $x:=\eta_0(p)$, let 
$$
\tilde\gamma_p^{+}(t) := (p,t), \quad
\gamma_x^+(t):=\eta(\tilde\gamma(t)), \quad t\in [0,t_0].
$$
We call the geodesic $\tilde\gamma_p^{+}$
 (resp. $\gamma_x^+$) the {\it perpendicular to} $C_0$  at $p$
(resp. a perpendicular to $X_0$ at $x$).
The maps $\tilde \pi$ and $\pi$ are the projections along perpendiculars.
We also call 
$$
\tilde\xi_p^+:=\dot{\tilde\gamma}_p^+(0), \quad
\xi_x^+:=\dot\gamma_x^+(0)
$$
the perpendicular directions at $p$ and  $x$ respectively.
 When $\#\eta_0^{-1}(x)=2$, we use the notation $\gamma_x^+(t)$ and 
$\gamma_x^-(t)$ to denote the two perpendiculars at $x$. We set
\beq \label{eq:xipm}
\xi_x^{+} := \dot\gamma_x^+(0) \quad
\xi_x^{-} := \dot\gamma_x^-(0).
\eeq
Obviously, we have  $|\eta(p,t), X|=t$  for  all $(p,t)\in C$.  We set
\beq \label{eq:CtY}
            C_t^Y:= \eta(C_t).
\eeq
The projection $\pi:Y\to X$
also provides a Lipschitz strong deformation of $Y$ to $X$. Thus we have 

\begin{prop}[{\cite[Theorem 1.5]{wong0}}] \label{prop:retractionYX}
$Y$ has the same Lipschitz homotopy type as $X$.
\end{prop}

\psmall

Let $X^{\rm int}$ denote the intrinsic metric of $X$.  

\psmall
\begin{lem}[{\cite[2.1 (iii)]{wong0}}] \label{lem:XbiLipX}
The canonical map $X\to X^{\rm int}$ is an
$L$-bi-Lipschitz homeomorphism.
\end{lem}  

\begin{prop} [{\cite[Proposition 3.9, Remark 3.12]{YZ:inradius}}]\label{prop:intrinsic}
There exists an isometry $f: X^{\rm int} \to N$ such that 
$f(X_0) =N_0$.
\end{prop}

From now on, in view of Proposition \ref{prop:intrinsic},
 we often identify 
$N=X^{\rm int}$ and  $N_0=X_0$,
or shortly $N=X$ when no confusion 
arises.

\psmall
We always assume $m=\dim Y$.
From a uniform positive lower bound for  ${\rm inrad}(M_i)$, we immediately have 
\beq
      \dim X = m,   \label{eq:dimX} \,\,\, {\rm int}\, X \neq\emptyset,
\eeq
and that $X_0$ coincides with the topological boundary $\partial X$ of $X$ in $Y$.

Let 
$X\cup_{\eta_0} ([0,t_0]\times_{\phi}C_0)$ denote the geodesic space obtained by the result of
gluing of the two geodesic spaces $X^{\rm int}$ and $[0,t_0]\times_{\phi}C_0$ by the map 
$\eta_0:\{ 0\}\times C_0\to X^{\rm int}_0$.

\begin{lem} [{\cite[Proposition 3.11]{YZ:inradius}}]   \label{lem:YX0}
 $Y$ is isometric to the geodesic space 
\[
      X\cup_{\eta_0}([0,t_0]\times_{\phi}C_0),
\]
where $(0,x)\in \{ 0\}\times C_0$ is identified with $\eta_0(x)\in X_0$ for each $x\in C_0$.  
\end{lem}
\pmed


\begin{defn}
For $k=1,2$, set
\[
    X_0^k:=\{ x \in X_0\,|\, \# \eta_0^{-1}(x)=k\,\}, \,\,\, C_0^k:=\eta_0^{-1}(X_0^k).
\] 
In view of $N=X^{\rm int}$, we often identify 
$N_0 =X_0,\,\,\, N_0^k = X_0^k$.
Define $\mathcal S^k\subset X_0^k$ in the same way, and set 
$\tilde{ \mathcal S^k} :=\eta_0^{-1}(\mathcal S^k)$. We also set
$\ca S=\ca S^1\cup\ca S^2$ and 
$\tilde{\ca S}=\tilde{\ca S}^1\cup\tilde{\ca S}^2$.
\end{defn}
         
\pmed

The following  is immediate from
the warped product structure of $C$.

\begin{lem} [{\cite[Lemma 4.5]{YZ:inradius}}]
For any $p\in C_0$, 
$\Sigma_p(C)$ is isometric to the half-spherical suspension $\{ \dot{\tilde\gamma}_p^{+}(0) \}*\Sigma_p(C_0)$.
\end{lem}

For $x\in X_0$, $\Sigma_x(X_0)$ and $\Sigma_x(X)$ are defined as  closed subsets of
the Alexandrov space  $\Sigma_x(Y)$ as usual. 

\begin{lem}[{\cite[Lemma 3.11]{YZ:part1}}] \label{lem:not-perp}
For any $x\in X_0$, 
and for any $\xi\in\Sigma_x(Y)$
with $\angle(\xi_x^+,\xi)<\pi/2$ and $Y$-geodesic $\gamma$ with
$\dot\gamma(0)=\xi$, we have the following$:$
\begin{enumerate}
\item $\gamma(t)\in Y\setminus X$ for all small enough $t>0\,;$
\item  The curve $\sigma(t)=\pi\circ \gamma(t)$ defines a 
unique direction $[\sigma]\in\Sigma_x(X_0)$ in the sense
$[\sigma]=\lim_{t\to 0} \uparrow_x^{\sigma(t)}$, and satisfies 
\[
 \angle(\xi_x^{+},\xi)+\angle(\xi, [\sigma])=\angle(\xi_x^{+}, [\sigma])=\pi/2\,;
\]
\item Under the convergence $(\frac{1}{t} Y, x)\to (K_x(Y),o_x)$ as 
$t\to 0$, $\sigma(t)$ converges to a minimal geodesic
$\sigma_\infty$ from $o_x$ in the direction $[\sigma]$.
\end{enumerate}
\end{lem}

For $x\in X_0^1$, we consider the {\it radius
of $\Sigma_x(Y)$ viewed from $\xi_x^+$} defined as 
\[
{\rm rad}(\xi_x^+):=\max \{ \angle(\xi_x^+, \xi)\,|\,\xi\in\Sigma_x(Y)\}.
\]

\begin{lem}[{\cite[Lemma 3.13]{YZ:part1}}] \label{lem:single-interior}
For  any $x\in X_0^1$ with ${\rm rad}(\xi_x^+)>\pi/2$, 
take $\xi\in\Sigma_x(Y)$ such that 
$\angle(\xi^+_x,\xi)>\pi/2$ and the geodesic $\gamma$
in the direction $\xi$ is defined. Then we have 
\begin{enumerate}
\item $x\in {\rm int}\,X_0^1\,;$
\item $\gamma((0,\e))\subset
{\rm int} X$ for some $\e>0$.
\end{enumerate}
\end{lem}

\begin{lem}[{\cite[Lemma 3.12]{YZ:part1}}] \label{lem:double-sus}
If $x\in X_0^2$, then $\Sigma_x(Y)$ is isometric to the spherical suspension $\{ \xi_x^+, \xi_x^-\}*\Sigma$, where $\xi_x^+, \xi_x^-$ are as in 
\eqref{eq:xipm},
and $\Sigma$
is isometric to  any of $\Sigma_p(C_0)$, $\Sigma_x(X_0)$ and $\Sigma_x(X)$ with $\eta_0(p)=x$.
\end{lem}

From Lemmas \ref{lem:not-perp}, \ref{lem:single-interior}
and \ref{lem:double-sus},  we obtain the following.

\begin{prop}[{\cite[Proposition 3.17]{YZ:part1}}]    \label{prop:perp+horiz}
Let $x\in X_0$.
\begin{enumerate}
 \item For every  $\xi\in\Sigma_x(Y)\setminus \Sigma_x(X)$ which is not a perpendicular direction, there is a unique perpendicular direction  $\xi_x^{+}\in \Sigma_x(Y)$ at $x$
and a unique  $v\in\Sigma_x(X_0)$ such that
      \begin{equation*}
         \angle(\xi_x^{+},\xi)+\angle(\xi, v)=\angle(\xi_x^{+}, v)=\pi/2.
       \end{equation*}
 \item For any perpendicular direction $\xi_x^+\in\Sigma_x(Y)$, we have
     \[
              \Sigma_x(X_0) = \{ v\in\Sigma_x(Y)\,|\,\angle(\xi_x^{+}, v)=\pi/2\}.
    \]
 \item If $x\in X_0^1$ and $\xi_x^+$ is the perpendicular direction at $x$, then  
    \[
              \Sigma_x(X) = \{ v\in\Sigma_x(Y)\,|\,\angle(\xi_x^{+}, v)\ge \pi/2\}.
    \]
\end{enumerate}
\end{prop}

\psmall

\begin{lem}$($\cite[Lemma 3.25]{YZ:part1}$)$ \label{lem:deviation}
For arbitrary $x, y\in X$ and any  minimal geodesic
$\mu:[0,\ell]\to Y$ joining them, let $\sigma=\pi\circ\mu$
and set $\rho(t)=|\mu(t), X|$. Then we have
\begin{enumerate}
 \item  $\rho(t)\le Ct|x,y|_Y$, where $C=C(\lambda)$. In particular,
\[
     \max \rho\le O(|x,y|_Y^2);
\]
 \item $\angle(\mu'(0), [\sigma])\le O(|x,y|_Y);$
 \item$\displaystyle {\left| \frac{L(\sigma)}{L(\mu)} - 1\right|< O(|x,y|_Y^2)}$.
\end{enumerate}
\end{lem}  

\begin{lem}[{\cite[Lemma 3.21]{YZ:part1}}]\label{lem:S1-closed}
${\ca S}^1$ is closed.
\end{lem}

We consider the distance function $d_X$ from $X$.

\begin{lem}\label{lem:dX-est}
For any $x\in X_0$ and $y\in B^Y(x,\e)$, setting $s:=|y,X|$, we have 
\[
  \angle(\dot\gamma^Y_{y,x}(0), \Sigma_y(d_X^{-1}([0,s])))<\tau_x(\e),
\]
\end{lem}
\begin{proof}
By contradiction, suppose the conclusion does not hold. Then we have a sequence $y_i\to x$ such that 
\beq\label{eq:angle(dX)}
\angle(\dot\gamma^Y_{y_i,x}(0), \Sigma_{y_i}(d_X^{-1}([0,s_i])))\ge \omega>0,
\eeq
where $s_i:=|y_i,X|$ and $\omega$ is a constant 
independent of $i$.
Consider the rescaling limit:
\[
   (Y/s_i, y_i)\to (T_x(Y),y_\infty),
\]
where we may assume that $d_X^{-1}([0,s_i])$
converges to a convex set $Z_\infty$ of $T_x(Y)$.
If $\xi_{y_\infty}^+$ denotes the direction perpendicular to $Z_\infty$, from the convexity of $Z_\infty$, we have 
$\angle(\uparrow_{y_\infty}^{o_x}, \xi_{y_\infty}^+)\ge \pi/2$. From the  lower semicontinuity of angle,
this yields a contradiction to \eqref{eq:angle(dX)}.
\end{proof}  

\pmed\n
{\bf Gluing maps.}\, 
Next we describe the properties of the gluing map 
$\eta_0:C_0\to X_0$ proved in 
\cite{YZ:part1}.
\pmed

\begin{lem}[{\cite[Lemma 4.5]{YZ:part1}}] \label{lem:eta'}
The restrictions of $\eta_0$  to 
${\rm int}\,C_0^k$\,$(k=1,2)$ 
have the following properties:
 \begin{enumerate}
 \item $\eta_0 :{\rm int}\,C_0^1\to {\rm int}\,X_0^1$ is a bijective local isometry  with 
respect to the intrinsic metrics$;$
 \item $\eta_0:{\rm int}\,C_0^2\to {\rm int}\,X_0^2$ is a locally isometric double covering
with respect to the intrinsic metrics.
 \end{enumerate}
\end{lem} 

\pmed

Recall that  $X_0^k\setminus \mathcal S^k$ is open in $X_0$.
Moreover we have the following. 

\begin{lem}[{\cite[Lemma 3.15]{YZ:part1}}] \label{lem:X02}
 ${\rm int}\, X_0^2=X_0^2\setminus \mathcal S^2$ is open in $X$. In particular, we have 
 \begin{enumerate}
 \item for every $x\in X_0^2\setminus \mathcal S^2$, there is an $\epsilon>0$ such that
 $\dim_H X \cap B(x, \epsilon) = m-1;$
 \item  $\mathcal S$ is empty if and only if $X_0=X_0^1$.
\end{enumerate}
\end{lem} 

\psmall

\begin{prop}[{\cite[Proposition 4.9]{YZ:part1}}] \label{prop:eta-2cover}
The restriction $\eta_0:C_0^2\to X_0^2$ is a locally 
almost isometric double covering. Namely, for every $x\in X_0^2$, there
is a neighborhood $V$ of $x$ in $X_0^2$ such that 
$\eta_0^{-1}(V)$ consists of disjoint open subsets
$W_1$ and $W_2$ of $C_0^2$ and each restriction $\eta_0:W_k\to V$\,\,$(k=1,2)$\,
 is almost isometric in the sense that for all $\tilde y, \tilde z\in W_k$ we have  
\[
    1-\tau_x(r)< \frac{|\eta_0(\tilde y), \eta_0(\tilde z)|_Y}{|\tilde y,\tilde z|_{C_0}} \le 1,
\]
where $r=\diam(V)$.
\end{prop}  

From now, we denote by  $B^{X_0}(x,r)$ the $r$-ball around $x$ with respect to the 
extrinsic metric induced from $X$.

\begin{lem}[{\cite[Sublemma 4.13]{YZ:part1}}] \label{lem:almost-2cover}
For any $p\in C_0^2$, if $r>0$ is small enough, then $\eta_0:B^{C_0}(p,r)\to X_0$ is injective, and 
for all $\tilde y,\tilde z\in B^{C_0}(p,r)$ we have 
\[
    1-\tau_p(r)\le 
 \frac{|\eta_0(\tilde y), \eta(\tilde z)|_Y}{|\tilde y, \tilde z|_{C_0}}\le 1.
\]
\end{lem}  
\pmed


\setcounter{equation}{0}
\section{Infinitesimal structure} \label{sec:infinitesimal}

In this section, we exhibit some 
results on the infinitesimal structure 
of the limit spaces, which will be 
needed in later sections.
 Most of them are obtained in  
\cite{YZ:part1} except the results on the directional derivative of $\pi$.

First of all, we recall that the limit space 
$N$ of $M_i$ is infinitesimally Alexandrov.
This implies that  the space of directions 
$\Sigma_x(N)$ at any $x\in N$
is  defined and 
becomes an Alexandrov space with curvature $\ge 1$. 
Moreover the boundary $N_0$ of $N$ is infinitesimally sub-Alexandrov, which shows that the space of directions 
$\Sigma_x(N_0)$ at any $x\in N_0$
is  defined  as a closed subset of 
$\Sigma_x(N)$ and the intrinsic metric $\Sigma_x(N_0)^{\rm int}$ of 
$\Sigma_x(N_0)$ induced from $\Sigma_x(N)$ is an Alexandrov space with curvature $\ge 1$. 
 In addition, 
$\Sigma_x(X)$  and $\Sigma_x(X_0)$
are naturally identified with 
$\Sigma_x(N)$  and $\Sigma_x(N_0)$
respectively.
For the detail, we refer the reader to \cite[Section 5]{YZ:part1}.

\begin{prop}[{\cite[Proposition 3.24]{YZ:part1}}] \label{prop:tang-cone}
For every $x\in X_0$, under the convergence  
\[
    \lim_{\e\to 0}\biggl(\frac{1}{\e} Y, x\biggr) =(T_x(Y),o_x),
\]
$(X_0, x)$ and $(X, x)$ converge to the Euclidean cones $K(\Sigma_x(X_0)),o_x)$ and $K(\Sigma_x(X)),o_x)$ 
over $\Sigma_x(X_0)$ and $\Sigma_x(X)$
respectively as $\e\to 0$.
\end{prop}

We set
\[
       T_x(X_0) := K(\Sigma_x(X_0)), \quad T_x(X) := K(\Sigma_x(X))
\]
and call them the {\it tangent cones}  of $X_0$ and $X$ at $x$ respectively, as usual.

The infinitesimal structure at points of 
$X_0^2$ is rather simple and
described 
in  Lemma \ref{lem:double-sus}.

\pmed
Now we consider the {\it differential} of the map $\eta:C\to Y$, which is defined as a rescaling limit as follows. 
Fix $p\in C_0$ and $x=\eta_0(p)\in X_0$, and
let $t_i$ be an arbitrary sequence of positive numbers with $\lim_{i\to\infty} t_i=0$.
Passing to a subsequence, we may assume that
\[
     \eta_i:=\eta:  \biggl(\frac{1}{t_i}C, p\biggr) \to \left(\frac{1}{t_i}Y, x\right)
\]
converges to a $1$-Lipschitz map
\[
   d\eta_p:  (T_p(C),o_ p) \to (T_x(Y), o_x)
\]
between  the tangent cones of the Alexandrov spaces.

It is shown in \cite{YZ:part1} that 
the limit 
$d\eta_p:  (T_p(C),o_ p) \to (T_x(Y), o_x)$ is uniquely determined independent of 
the choice of $\{ t_i\}$.
We  call it the {\it differential} of $\eta$ at $p$.
Note that $d\eta_p$ induces a $1$-Lipschitz map
\[
            (d\eta_0)_p:T_p(C_0) \to T_x(X_0),
\]
and  an injective local isometry
\[
  (d\eta)_p:T_p(C)\setminus T_p(C_0) \to T_p(Y)\setminus T_x(X).
\]
In what follows, we simply write as $d\eta$ and $d\eta_0$ for $d\eta_p$ and $(d\eta_0)_p$ respectively, when 
$p$ is fixed.

For the perpendicular $\gamma_x^+(t):=\eta(p,t)$,
let $\xi_x^+:=   \dot\gamma_x^+(0)$, and set
\[
    \Sigma_x(Y)^+:=\{ \xi\in\Sigma_x(Y)\,|\,\angle(\xi,\xi_x^+)\le\pi/2\}.
\]

\begin{prop}[{\cite[Proposition 3.27]{YZ:part1}}] \label{prop:length'}
For every $\tilde v\in T_p(C)$, we have  
\[
                |d\eta(\tilde v)| = |\tilde v|.
\]
In particular,  $\eta$ and $\eta_0$ induce 
$1$-Lipschitz maps 
$$
d\eta:\Sigma_p(C) \to \Sigma_x(Y),\quad
d\eta_0:\Sigma_p(C_0) \to \Sigma_x(X_0)
$$
such that 
\begin{itemize}
\item both $d\eta:\Sigma_p(C) \to \Sigma_x(Y)^+$ and
$d\eta_0:\Sigma_p(C_0) \to \Sigma_x(X_0)$
are surjective$\,;$
\item $d\eta:\Sigma_p(C)\setminus\Sigma_p(C_0) \to \mathring{\Sigma}_x(Y)^+$
is a bijective local isometry.
\end{itemize}
\end{prop}
\psmall

For any $x\in X_0^1$ and $p\in C_0$ with 
$\eta_0(p)=x$, consider the map
$d\eta_0:\Sigma_p(C_0)\to\Sigma_x(X_0)$.

\begin{lem}[{\cite[Lemma 6.1]{YZ:part1}}]\label{lem:sharp-deta} We have 
$1\le \# d\eta_0^{-1}(v) \le 2$ for all $v\in\Sigma_x(X_0)$.
\end{lem}

Now we define the involution $f_*:\Sigma_p(C_0)\to\Sigma_p(C_0)$ 
as 
\begin{align} \label{eq:defn-df}
 \text{$f_*(\tilde v):= \tilde w$ \, if\, $\{\tilde v, \tilde w\}=d\eta_0^{-1}(d\eta_0(\tilde v))$}, 
\end{align}

The infinitesimal structure at points of 
$X_0^1$ is given by the following
 
\begin{thm}[{\cite[Theorem 6.4]{YZ:part1}}]\label{thm:X1-f}
For each $x\in X_0^1$,
take $p\in C_0$ with $\eta_0(p)=x$. Then 
$f_*:\Sigma_p(C_0)\to\Sigma_p(C_0)$ is an isometry satisfying  the following:
\benu
 \item $\Sigma_x(X_0)^{\rm int}$ is isometric to the quotient space $\Sigma_p(C_0)/f_*\,:$
 \item If ${\rm rad}(\xi_x^+)=\pi/2$, then  $\Sigma_x(Y)$ and $\Sigma_x(X)$ are  isometric to the quotient geodesic  spaces 
 $\Sigma_p(C)/f_*$ and 
 $\Sigma_p(C_0)/f_*$ respectively$\,;$
\item If ${\rm rad}(\xi_x^+)>\pi/2$, then $f_*$ is the identity and 
$\Sigma_x(Y)$ is isometric to the gluing $\Sigma_p(C)\cup_{d\eta_0} \Sigma_x(X)$, where the identification is made by the isometry $d\eta_0:\Sigma_p(C_0) \to \Sigma_x(X_0)^{\rm int}.$ 
\eenu
\end{thm} 

The following is a key result in \cite{YZ:part1}.

\begin{thm}[{\cite[Theorem 1.2]{YZ:part1}}] \label{thm:non-trivial}
For any $p\in\tilde{\ca S}^1$,
$f_*:\Sigma_p(C_0)\to\Sigma_p(C_0)$ is not
the identity.
\end{thm}

Finally we exhibit a result on the Hausdorff dimensions of the
metric singular set and the boundary singular sets
of $N_0$.

Theorems \ref{thm:X1-f} and \ref{thm:non-trivial}
imply $\ca S^1\subset X_0^{\rm sing}$.

\begin{defn} \label{defn:cusp}
A point $x$ of ${\rm int} X_0^1$ is called a 
{\it cusp} if $f_*:\Sigma_p(C_0)\to \Sigma_p(C_0)$
is not the identity. We denote by 
${\ca C}$ the set of all cusps.
Obviously we have $\ca C\subset X_0^{\rm sing}$.
\end{defn}

\begin{thm}[{\cite[Theorem 1.4]{YZ:part1}}] \label{thm:dim(metric-sing)}
Let a sequence $M_i$ in $\ca M(n,\kappa,\nu,\lambda, d)$ non-inradius collapses/converges  to a geodesic space $N$.
If $\dim N=m$, then  the following holds$:$
\begin{enumerate}
\item $\dim_H N_0^{\rm sing} \le m-2$. In particular,
$\dim_H (\ca S^1\cup\ca C)\le m-2\,;$
\item $\dim_H (N_0^{\rm sing} \cap {\rm int}\, N_0) \le m-3$.
\end{enumerate}
\end{thm}
\pmed\n

\pmed\n 
{\bf Blow up limits.}\,
Let $(Y, X, X_0)$ (resp. $(C,C_0)$) be the limit spaces
of $(\tilde M_i,M_i,\pa M_i)$ (resp. $(C_{M_i}, \{ 0\}\times(\pa M_i)^{\rm int})$.
For any sequence $y_i\in X_0$, $\tilde y_i\in C_0$
with $y_i=\eta_0(\tilde y_i)$ and $\e_i\to 0$,
let us consider the rescaling limits 
\begin{align} \label{eq:inradius-conv}
     \left(\frac{1}{\e_i} Y, y_i\right) \to (Y_\infty, y_\infty), \quad
\left(\frac{1}{\e_i} C, \tilde y_i\right) \to (C_\infty, \tilde y_\infty),
\end{align}

\par\n
{\bf Notations.}\,
From here on, let us denote by  $X_\infty$, $(X_0)_\infty$, $Y_\infty$, 
$(\pa Y)_\infty$  and $C_\infty$, $(C_0)_\infty$
 the limits of $X$, $X_0$, $Y$, $\pa Y$
and $C$, $C_0$  respectively,
with respect to  blow-up rescaling limits
like \eqref{eq:inradius-conv}
under consideration.
Obviously, $Y_\infty$ is a complete noncompact nonnegatively curved Alexandrov space with boundary $(\pa Y)_\infty$.
$1$-Lipschitz maps $\eta_\infty:C_\infty\to Y_\infty$ and $(\eta_0)_\infty:(C_0)_\infty\to (X_0)_\infty$ are defined as the limits of 
$\eta$ and $\eta_0$ respectively.
A perpendicular $\gamma_{y_\infty}^+$ and 
a perpendicular direction $\xi_{y_\infty}^+$ at 
$y_\infty$ are defined similarly.
For any $p_\infty\in (C_0)_\infty$, let
$x_\infty:=\eta_\infty(p_\infty)\in (C_0)_\infty$.
The differential 
of $\eta_\infty$ at $p_\infty$ 
\[
    d\eta_\infty:T_{p_\infty}(C_\infty)\to T_{x_\infty}(Y_\infty)
\]
is uniquely determined as the blow-up limit of 
$\eta_\infty$ as before. In the same way as Proposition \ref{prop:length'}, we can verify  
\beq \label{eq:d(eta-infty)}
      |d\eta_\infty(\tilde v)|=|\tilde v|
\eeq
for all $\tilde v\in T_{p_\infty}(C_\infty)$, and hence
it induces $1$-Lipschitz maps
$d(\eta_\infty):\Sigma_{p_\infty}(C_\infty)\to \Sigma_{x_\infty}(Y_\infty)$ and 
$d\eta_\infty:\Sigma_{p_\infty}((C_0)_\infty)\to \Sigma_{x_\infty}((X_0)_\infty)$.  
These maps satisfies the properties corresponding to Proposition \ref{prop:length'}.

\begin{lem} \label{lem:inrad-collapse}
Under the rescaling limits \eqref{eq:inradius-conv},
it holds that
\begin{enumerate}
\item $X_\infty$ is convex in $Y_\infty$. 
\item $C_\infty=(C_0)_\infty\times\R_+$.
\end{enumerate}
\end{lem} 
\begin{proof}
(1) follows from \cite[Lemma 7.2]{YZ:part1}. 
(2) is clear since $C=[0,t_0]\times_\phi C_0$.
\end{proof}

\pmed\n
{\bf The directional derivative of $\pi$.}\,
We consider the projection $\pi:Y\to X$
and study its infinitesimal properties.
For any $y\in Y\setminus{\rm int}X$,
let $z:=\pi(y)$ and $t(y):=|y,X|=|y,z|$.
Note that 
\beq \label{eq:pi-expression}
     \pi=\eta_0\circ\tilde\pi\circ\eta^{-1}\quad
    \text{on $Y\setminus X$}.
\eeq

First assume $t(y)>0$. Then we have the orthogonal
decomposition 
\beq\label{eq:decomp-Ty(Y)}
    T_y(Y)=T_y(C_{t(y)})\oplus \R\xi^+(y),
\eeq
where $\xi^+(y)$ denotes the direction of the
perpendicular through $y$. 
We call $T_y(C_{t(y)}^Y)$ and $\R\xi^+(y)$ the \emph{horizontal} and the \emph{vertical} subspaces of $T_y(Y)$ respectively.
For any 
$v\in T_y(Y)$, let $v=v_H+v_V$ denote the orthogonal decomposition of $v$.

\begin{lem}\label{lem:derivative-pi}
There exists the directional derivative 
$d\pi_y: T_y(Y)\to T_z(X_0)$ having the property
\beq \label{eq:|dpi|=phi}
      |d\pi_y(v)|=\phi(t(y))^{-1}|v_H|.
\eeq
\end{lem}
\begin{proof}
Let $\tilde y:=\eta^{-1}(y)$. Note that $d\eta^{-1}$
preserves the decompostions \eqref{eq:decomp-Ty(Y)} and the corresponding decomposition
\[
      T_{\tilde y}(C)=T_{\tilde y}(C_{t(y)})\oplus \R\xi^+(\tilde y)
\]
of $T_{\tilde y}(C)$. From the warped product structure of $C$, 
$d\tilde\pi_{\tilde y}$ has the property corresponding to \eqref{eq:|dpi|=phi}.
For any $v\in T_y(Y)$, let $c$ be a curve parametrized on $[0,\e)$ such that $\dot c(0)=v$.
Then we have 
\[
|(\tilde\pi\circ\eta^{-1}\circ c)'(0)|=\phi(t(y))^{-1}|v_H|.
\]
The conclusion follows from \eqref{eq:pi-expression}
and 
Proposition \ref{prop:length'} immediately. 
\end{proof}

When $t(y)=0$, we have the following.

\begin{lem}\label{lem:derivative-pi-0}
If $y\in X_0$, the directional derivative 
$d\pi_y: T_y(Y)\to T_y(X)$ is characterizes as follows: For any $\xi\in T_y(Y)\setminus T_y(X)$,
choose $v\in\Sigma_y(X_0)$ as in Proposition
\ref{prop:perp+horiz}(1). Then we have 
\beq \label{eq:dpi=sin}
   d\pi_y(\xi)=\cos \angle(\xi,v)\, v.
\eeq
\end{lem}
\begin{proof}
We can apply \eqref{eq:pi-expression} partially.
Take $\tilde\xi\in T_{\tilde y}(C)\setminus T_{\tilde y}(C_0)$ such that $d\eta(\tilde\xi)=\xi$.
From the splitting $T_{\tilde y}(C)=T_{\tilde y}(C_0)\times\R_+$, we have the formula
for $d\tilde\pi_{\tilde y}(\tilde\xi)$ corresponding to
\eqref{eq:dpi=sin}. The conclusion follows from
Proposition \ref{prop:length'}.
\end{proof}

Next for any fixed $x\in X_0$, we investigate the behaviour of $d\pi_y$ when $y$ is close to $x$
and $|y,\pi(y)|/|x,y|$ is small enough.

\begin{lem} \label{eq:dpi-x}
For any fixed $x\in X_0$, if $0<|y,\pi(y)|/|x,y|<\tau_x(|x,y|)$, then we have 
\[
     \angle(d\pi_y(\uparrow_y^x), \uparrow_{\pi(y)}^x)<\tau_x(|x,y|).
\]
\end{lem}
\begin{proof}
The proof is by contradiction.
Suppose the lemma does not hold. Then we have sequence $y_i\to x$ such that 
$0<|y_i,\pi(y_i)|/|x,y_i|\to 0$ and 
\[
   \angle(d\pi_{y_i}(\uparrow_{y_i}^x), 
\uparrow_{\pi(y_i)}^x)\ge \omega>0
\]
for a constant $\omega$ independent of $i$.
Let $z_i:=\pi(y_i)$ and $s_i:=|y_i,z_i|$.

Consider the point $y_i':=\gamma_{y_i,x}(s_i)$, and 
set $\gamma_i:=\gamma_{y_i,y_i'}$.
Let $\tilde y_i$, $\tilde z_i$ and $\tilde y_i'$ be 
the respective lifts of  $y_i$, $z_i$ and $y_i'$ to $C$.
Let $\tilde\gamma_i:=\gamma_{\tilde y_i,\tilde y_i'}$
be the lift of $\gamma_i$.
It should be emphasized that 
we do not know if a geodesic 
$\gamma_{z_i, y_i'}$ has a lift to $C$.
This forces us to take a  geodesic 
$\gamma_{\tilde z_i, \tilde y_i'}$ in $C$ and consider 
its $\eta$-image, as described below. 

\begin{slem} \label{slem:geode-warped}
One can choose a minimal 
geodesic  $\gamma_{\tilde z_i,\tilde y_i'}$ such that 
$\tilde\pi(\tilde\gamma_i)=\tilde\pi(\gamma_{\tilde z_i,\tilde y_i'})$ up to monotone parametrization.
\end{slem}
\begin{proof}
First choose any $\gamma_{\tilde z_i,\tilde y_i'}$ and set 
$\tilde \sigma_i:=\gamma_{\tilde z_i,\tilde y_i'}$. 
From the warped product structure of $C$, 
both the images $\tilde\pi(\tilde\gamma_i)$ and
$\tilde\pi(\tilde\sigma_i)$ provide shortest
curves joining $\tilde z_i$ and $\tilde\pi(\tilde y_i')$. 
Let $\tilde\pi(\tilde\gamma_i)(s)$ and
$\tilde\pi(\tilde\sigma_i)(s)$\,$(0\le s\le |\tilde z_i, \tilde\pi(\tilde y_i')|_{C_0})$ be the arclength
parameter of $\tilde\pi(\tilde\gamma_i)$ and
$\tilde\pi(\tilde\sigma_i)$.
For each $s$, take $t(s)\in [0, t_0]$  such that 
$(t(s), \tilde\pi(\tilde\sigma_i)(s))\in\tilde\sigma_i$.
It follows from the warped product structure of $C$
that the curve 
\[
\tilde\rho_i(s):=(t(s), \tilde\pi(\tilde\gamma_i)(s))
\]
is shortest joining $\tilde z_i$ and $\tilde y_i'$ since $L(\tilde\rho_i)=L(\tilde\sigma_i)$.
Since $\tilde\pi(\tilde\rho_i)=\tilde\pi(\tilde\gamma_i)$,  $\tilde\rho_i$ is a required one.
\end{proof}

Let $\tilde\sigma_i$ denote the geodesic
joining $\tilde z_i$ to $\tilde y_i'$ provided in 
Sublemma \ref{slem:geode-warped}
satisfying 
$\tilde\pi(\tilde\gamma_i)=\tilde\pi(\tilde\sigma_i)$.
Consider $\sigma_i:=\eta\circ\tilde\sigma_i$.
As mentioned above,  we do not know if $\sigma_i$ is $Y$-minimal.
However  $\sigma_i$ has the definite direction
$\dot\sigma_i(0)=d\eta(\dot{\tilde\sigma}_i(0))$
at $z_i$. Note also that  
\beq\label{eq:pi(sigma)=pi(gamma)}
       \pi(\sigma_i)=\pi(\gamma_i).
\eeq

Set
\[
     \xi_i^+:=\dot\gamma_{z_i,y_i}(0),\,\,
u_i:=\dot\sigma_i(0),\,\,
\nu_i:=\dot\gamma_{z_i,x}(0).\,\,
\]
Let $v_i\in\Sigma_{z_i}(X_0)$ be the direction 
defined by $\pi(\sigma_i)=\pi(\gamma_i)$.
It follows from 
Lemma \ref{lem:not-perp}(3) that 
\beq \label{eq:angle(3)=pi/2-Sec4}
\angle(\xi_i^+,v_i)=\angle(\xi_i^+,u_i)+\angle(u_i,v_i)=\pi/2.
\eeq
\begin{figure}
\begin{center}
\begin{tikzpicture}
[scale = 0.5]

\draw  [thick] (-1,0)  to [out=-10, in=190] (1,0);
\draw  [thick] (-1,0)  to [out=90, in=225] (0,4);
\draw  [thick] (1,0)  to [out=90, in=-45] (0,4);
\fill (-0.95,2) circle (2pt) node [left] {\small{$u_i$}};
\draw  [thick] (-0.95,2)  to [out=-60, in=150] (1,0);
\fill (0,4) circle (2pt) node [above] {\small{$\xi_i^+$}};
\fill (-1,0) circle (2pt) node [left] {\small{$v_i$}};
\fill (1,0) circle (2pt) node [right] {\small{$\nu_i$}};

\end{tikzpicture}
\end{center}  
\vspace{-0.3cm}
\caption{ }
\end{figure}
\pmed


Here we use the abbreviation $a\sim b$ if and only if $|a-b|<o_i$.
Since $\angle(\xi_i^+,\nu_i)\sim \pi/2$, 
\eqref{eq:angle(3)=pi/2-Sec4} implies  
\[
\angle\nu_i\xi_i^+u_i=\angle\nu_i\xi_i^+v_i
\ge \wangle\nu_i\xi_i^+v_i \sim |\nu_i,v_i|
 \ge\omega/2,
\]
for large enough $i$.

We show 
\beq \label{eq:angle(xi,u)=a+b-4}
\angle(\xi_i^+,u_i)=\angle(\tilde\xi_i^+,\tilde u_i)\sim \pi/4,
\eeq
where $\tilde\xi_i^+:=d\eta^{-1}(\xi_i^+)$ and $\tilde u_i:=d\eta^{-1}(u_i)$.
 The first equality follows from  Proposition \ref{prop:length'}. 
Letting $s_i:=|\tilde y_i,\tilde z_i|$, we consider the convergence
\beq \label{eq:split-Y/s-4}
      (C/s_i,\tilde y_i)  \to (C_\infty, \tilde y_\infty).
\eeq
By the splitting theorem, 
$C_\infty$ is isometric to a product
$\R\times D_\infty$.
Let $\tilde z_i':=\pi(\tilde y_i')$.
From $\Sigma_{\tilde z_i}(C)=\tilde\xi_i^+*\Sigma_{\tilde z_i}(C_0)$,
we have
\[
 \angle(\tilde\xi_i^+, d\tilde\pi(\tilde u_i))=
   \angle(\tilde\xi_i^+, \tilde u_i)+
  \angle(\tilde u_i, d\tilde\pi(\tilde u_i))
     =\pi/2.
\]  
Note that Lemma \ref{lem:deviation}
can be applied to $C$-minimal geodesic 
joining two points of $C_0$, which implies 
\[
   \angle(d\tilde\pi(\tilde u_i), \uparrow_{\tilde z_i}^{\tilde z_i'})<O(|\tilde z_i,\tilde z_i'|)=o_i.
\]
Thus we have 
\[
   \angle \tilde y_i\tilde z_i \tilde y_i'+
\angle \tilde y_i'\tilde z_i \tilde z_i' \sim \pi/2.
\]
From the lower semicontinuity of angle,  we have
\[
  \pi/2\ge \angle \tilde y_\infty \tilde z_\infty \tilde y_\infty'+
\angle \tilde y_\infty'\tilde z_\infty \tilde z_\infty' 
\ge \angle\tilde y_\infty \tilde z_\infty\tilde z_\infty' 
=\pi/2.
\]
This  implies $\lim_{i\to\infty}\angle \tilde y_i\tilde z_i \tilde y_i'=\angle \tilde y_\infty \tilde z_\infty \tilde y_\infty'$, which is equal to $\pi/4$
from the splitting 
$C_\infty=\R\times D_\infty$,
yielding \eqref{eq:angle(xi,u)=a+b-4}.

Next we show 
\beq \label{eq:angle(uzeta)=pi/4-Sec4}
\angle(u_i,\nu_i)\sim \pi/4.
\eeq

Choose small $t_i>0$ such that
$\angle(\dot\sigma_i(0),\uparrow_{w_i}^{\sigma_i(t_i)})<1/i$, and set 
$d_i:=|\sigma_i(t_i),X|$.
Put  $a_i:=\gamma_{w_i}^+(d_i)$ and
$a_i':=\sigma_i(t_i)$.
To get \eqref{eq:angle(uzeta)=pi/4-Sec4},  
it suffices to show  
\beq \label{eq:stretch(xwx)-4}
\angle xz_i a_i'
\sim  \pi/4.\quad
\eeq

Let us 
consider the convergence
\beq \label{eq:split-Y/s-4}
      (Y/d_i,z_i)  \to (Y_\infty, z_\infty).
\eeq
By the splitting theorem, 
$Y_\infty$ is isometric to a product
$\R\times Z_\infty$.
Let $a_\infty'$ be the respective limits
of $a_i'$ under \eqref{eq:split-Y/s-4}.
Choosing a $(1,o_i)$-strainer $(x,x_i)$ at $z_i$ such 
that $t_i/\min\{ |x,z_i|, |x_i,z_i|\}<o_i$, and 
applying Lemma \ref{lem:comparison},
we have 
\beq \label{eq:strainer-ineq-4}
\begin{cases}
\begin{aligned}
&\wangle xz_i a_i'+ \wangle x_{i}z_i a_i'\sim\pi,
\quad \\
& \angle xz_i a_i'\sim \wangle xz_i a_i',
 \quad
 \wangle x_{i} z_i a_i'\sim\angle x_{i} z_i a_i.
\end{aligned}
\end{cases}
\eeq
In view of the splitting $Y_\infty=\R\times Z_\infty$
under \eqref{eq:split-Y/s-4}, we have 
$\angle \hat x z_\infty a_\infty'=\pi/4$, where 
$\hat x$ is the element of $Y_\infty(\infty)$ 
defined as the limit of the geodesic 
$\gamma^Y_{z_i,x}$ under \eqref{eq:split-Y/s-4}. This together with \eqref{eq:strainer-ineq-4} 
yields  \eqref{eq:stretch(xwx)-4}.

Combining \eqref{eq:angle(xi,u)=a+b-4},   \eqref{eq:angle(uzeta)=pi/4-Sec4} and 
$\angle(\xi_i^+,\nu_i)\sim \pi/2$, we obtain  
$\wangle\nu_i u_i \xi_i^+\sim \pi$,
which implies $\angle\nu_iu_i v_i\to 0$, and hence
$|v_i,\nu_i|\to 0$.
This is a contradiction and 
completes the proof of Lemma \ref{eq:dpi-x}.
\end{proof}

 From now on, we abbreviate $d\pi_y$ to
$d\pi$ as long as there is no confusion.

\setcounter{equation}{0}
\section{Almost isometries and local stability}  \label{sec:almost-stability}
\psmall
In this section, we are concerned with the local structures around points of $X_0$ in both $X_0$ and $X$.

First we describe some basic properties of the
thick points and the thin points of $X_0$.
For $x\in X_0$, let ${\bm \xi}_x^\perp$ denote the set of all perpendicular directions at $x$, and 
consider the {\it radius
of $\Sigma_x(Y)$ viewed from ${\bm \xi}_x^\perp$} defined as 
\[
{\rm rad}({\bm \xi}_x^\perp):=\max \{ \angle({\bm \xi}_x^\perp, \xi)\,|\,\xi\in\Sigma_x(Y)\}.
\]
The definition of thick/thin points of $X_0$ 
given below is equivalent to the one given in Definition \ref{defn:thin-thick-brpoint}.

\begin{defn}\label{defn:thick}
We say that a point $x\in X_0$  is {\it thick} (resp. {\it thin}) in $X$
if ${\rm rad}({\bm \xi}_x^\perp)>\pi/2$ (resp.
${\rm rad}({\bm \xi}_x^\perp)=\pi/2$).
Let ${\ca Tk}$ (resp. ${\ca Tn}$)
denote the set of all thick points
(resp. all thin points), and let ${\ca Tk}^{\rm reg}$
be the set of all boundary regular points of $X$
(see Definition \ref{defn:thin-thick-brpoint}).
We set ${\ca Tn}^{\rm reg}:={\ca Tn}\cap X_0^{\rm reg}$.
\end{defn}

 Note that 
\begin{itemize}
\item $X_0^2\subset {\ca Tn}\,;$
\item ${\ca Tk}\subset {\rm int} X_0^1\setminus \ca C$, \,${\ca Tn}^{\rm reg}\subset ({\rm int} X_0^1\setminus \ca C) \cup X_0^2\,;$
\item  ${\ca Tk}^{\rm reg}\subset {\ca Tk} \cap X_0^{\rm reg}$. However, the equality does not necessarily hold
(see Example \ref{ex:convexbody})$\,;$
\item ${\ca Tk}^{\rm reg}\cup{\ca Tn}^{\rm reg}$ coincides with the set of points $x\in  X_0$ such that the extrinsic metric
$\Sigma_x(X_0)$ is isometric to $\mathbb S^{m-2}$.
\end{itemize}

\begin{lem}\label{lem:dimTKNc}
\[
   \dim_H(X_0\setminus({\ca Tk}^{\rm reg}\cup{\ca Tn}^{\rm reg}))\le m-2.
\]
\end{lem}
\begin{proof}  Since
$X_0^2\setminus {\ca Tn}^{\rm reg}\subset
X_0^{\rm sing}$,
from  \cite[Theorem 1.4]{YZ:part1} and  Theorem \ref{thm:dim-sing},
it suffices to show that 
$$
X_0^1\setminus({\ca Tk}^{\rm reg}\cup{\ca Tn}^{\rm reg})\subset ({\rm int} Y)^{\rm sing}\cup
 D(Y)^{\rm sing}.
$$
Let $x$ be an arbitrary point of 
$X_0^1\setminus({\ca Tk}^{\rm reg}\cup{\ca Tn}^{\rm reg})$.
If $x\in {\rm int} Y$, then we have
$x\in ({\rm int} Y)^{\rm sing}$.
Actually, if $\Sigma_x(Y)=\mathbb S^{m-1}$, then 
Theorem \ref{thm:X1-f} implies that 
$\Sigma_x(X)=\mathbb S^{m-1}_+$, which contradicts
$x\notin  {\ca Tk}^{\rm reg}$.
If $x\in \pa Y$, then we have
$x\in D(Y)^{\rm sing}$.
Actually, if $\Sigma_x(D(Y))=\mathbb S^{m-1}$, then 
we have $\Sigma_x(Y)=\mathbb S^{m-1}_+$, 
and hence $x\in X_0^{\rm reg}\cap{\ca Tn}$. This contradicts
$x\notin  {\ca Tn}^{\rm reg}$.
\end{proof}

\begin{ex}\label{ex:convexbody}
Let $X$ be a compact convex body in $\mathbb R^m=\R^m\times\{ 0\}\subset\R^{m+1}$,
and let $M_\e$ be the intersection of the
boundary of the  
$\e$-neighborhood of $X$ in $\R^{m+1}$ with  
the half space $x_{m+1}\le 0$.
Then $M_\e$ is an $m$-dimensional nonnegatively curved manifold with totally geodesic boundary,
which converges to $X$ as $\e\to 0$.
Note that $X_0=X_0^1$.

As a special case, let $X$ be a $3$-dimensional cube
$I^3$. Then we have 
\begin{itemize}
\item any point of $X_0=\pa I^3$ is thick in $X\,;$
\item$\ca Tk^{\rm reg}$ is the union of open maximal 
faces of the polyhedron $\pa I^3\,;$
\item for any point of the edges of $\pa I^3$ except the 
vertices, we have $\Sigma_x(X_0)^{\rm int}=\mathbb S^1$, and hence $x\in X_0^{\rm reg}$, but $x\notin \ca Tk^{\rm reg}$. 
\end{itemize}
\end{ex}

\begin{rem}\label{rem:thin-large}
 Using \cite[Example 9.7]{YZ:part1},
one obtains  an example of non-inradius collapse such that
$\ca H^{m-1}(X_0^2)>0$ ($m=\dim X$).
Similarly, 
replacing the function $g$ in  Example \cite[Example 6.18(1)]{YZ:part1}  by a sequence
of functions $g_n$ similar to the one in
\cite[Example 9.7]{YZ:part1},
one can construct an example of non-inradius collapse such that 
\begin{itemize}
\item $\ca H^{m-1}(\ca Tn\cap X_0^1)>0\,;$
\item no local inradius collpase occurs
in the sense that $\dim_H(U)=m$
for any  nonempty open subset $U$ of $X$.
\end{itemize} 
\end{rem}

\pmed\n
\subsection{Local structure of $N_0$}
\psmall  
In this subsection, we prove Theorem \ref{thm:regular-ball}.
Recall that $B^{X_0}(x,r)$ (resp. $B^{X_0^{\rm int}}(x,r))$ denotes the metric $r$-ball around $x\in X_0$ with respect to 
the extrinsic metric (resp. the intrinsic metric) induced from $X$.

\pmed
\begin{proof}[Proof of Theorem \ref{thm:regular-ball}]  
(1)\,  Let $x\in X_0^{\rm reg}\cap X_0^1$.
By Theorems \ref{thm:X1-f} and \ref{thm:non-trivial}, we have $x\in {\rm int} X_0^1$ and  $\eta_0^{-1}(x)\in C_0^{\rm reg}$. The conclusion follows immediately from 
Lemma \ref{lem:eta'} and Theorem \ref{thm:almost-isometry} 
for $C_0$.
The conclusion of (3) for 
$\mathring{B}^{N_0^{\rm int}}(x,\e)$ 
also follows in a similar way, and hence omitted.

In what follows, let us consider the extrinsic metric of $X_0$. 

(2)\,
 If  $x\in {\ca Tk}^{\rm reg}$ (resp.
 $x\in {\ca Tn}^{\rm reg}  \cap X_0^1$),
Theorem \ref{thm:X1-f}
shows that $x\in Y^{\rm reg}$ (resp. $x\in D(Y)^{\rm reg}$).
Let us consider the case  $x\in {\ca Tk}^{\rm reg}$. 
By replacing $Y$ by $D(Y)$, we can discuss 
the case $x\in {\ca Tn}^{\rm reg} \cap X_0^1$ in the same way, and hence we omit the proof for this case.

For any $\delta>0$, we can find an $(m,\delta)$-strainer $\{(a_i,b_i)\}_{i=1}^m$ such that 
$a_i,b_i\in X_0$\,$(1\le i\le m-1)$, and  $a_m:=\gamma_x^+(t)$ for small $t>0$.
Letting $\varphi_i(y):=|a_i,y|_Y$, we have a
$\tau_x(\delta,\e)$-almost isometry
$\varphi=(\varphi_1,\ldots,\varphi_m):B^Y(x,\e)\to\R^m$ onto an open subset
of $\R^m$ (Theorem \ref{thm:almost-isometry}).
We show that $\psi:=(\varphi_1,\ldots,\varphi_{m-1})$ provides a $\tau_x(\delta,\e)$-almost isometry 
$\psi:B^{X_0}(x,\e)\to\R^{m-1}$ onto an open subset
of $\R^{m-1}$.
\cite[Theorem 5.4]{BGP} shows that 
$\psi$ is an open map.
From Lemma \ref{lem:deviation}, we have
for all $y,z\in B^{X_0}(x,\e)$
\beq\label{eq:quot|yz|XY}
   1\le \frac{|y,z|_X}{|y,z|_Y}\le 1+O(\e^2).
\eeq
We show that 
\beq \label{eq:angle(ayz)}
 \tilde\angle^Y a_m yz >\pi/2-\tau_x(\delta,\e),\quad
\tilde\angle^Y a_m zy >\pi/2-\tau_x(\delta,\e).
\eeq
Suppose this does not hold. Then we may assume that there are sequences $y_i,z_i\in X_0$ converging to $x$ satisfying
$\tilde\angle^Y a_m y_iz_i\le \pi/2-c$ for 
a constant $c>0$ independent of $i$ and $t$.
Set $\e_i:=|y_i,z_i|_Y$ and consider the rescaling limit
\[
    \biggl( \frac{1}{\e_i}Y,y_i\biggr)\to (Y_\infty,y_\infty).
\]
It turns out that $\angle(\gamma_{y_\infty}^+, \uparrow_{y_\infty}^{z_\infty})\le\pi/2-c$.
This implies that  
$$
|\gamma_{y_i}^+(\e_iR),z_i|<|\gamma_{y_i}^+(\e_iR),y_i|
$$ for large $R>0$ and $i$. This is a contradiction.

\eqref{eq:angle(ayz)} yields that 
$||a_m,y|-|a_m,z|| < \tau_x(\delta,\e)|y,z|_Y$.
Together with \eqref{eq:quot|yz|XY}, this yields
that 
\beqq
   \biggl|\frac{|\psi(y),\psi(z)|_{\R^{m-1}}}{|y,z|_X}
      -1\biggr| <\tau_x(\delta,\e).
\eeqq
This yields the conclusion for $\mathring{B}^X(x,\e)$.
The argument for $\mathring{B}^{X_0^{\rm int}}(x,\e)$
is almost the same. It only suffices to show the following in place of \eqref{eq:quot|yz|XY}:
\beq\label{eq:quot|yz|XintY}
   1\le \frac{|y,z|_{X_0^{\rm int}}}{|y,z|_Y}\le 1+\tau_x(\delta,\e).
\eeq
To show \eqref{eq:quot|yz|XintY},
we replace the above $\varphi_m$ by
$\tilde\varphi_m(y):=|C_t^Y,y|_Y$ for small
$t>0$.
By Theorem \ref{thm:ALMOST},
$\tilde\varphi:=(\varphi_1,\ldots,\varphi_{m-1},\tilde\varphi_m)$ provides a $\tau_x(\delta,\e)$-almost isometry 
$\tilde\varphi:B^{Y}(x,\e)\to\R^{m}$ onto an open subset
of $\R^{m}$.
Now the point is 
$\tilde\varphi(B^Y(x,\e)\cap X_0)=
\tilde\varphi(B^Y(x,\e))\cap \R^{m-1}\times\{ 0\}$.
For any $y,z\in B^{X_0^{\rm int}}(x,\e)$,
if $\alpha$ denotes the minimal geodesic joining
$\tilde\varphi(y)$ and $\tilde\varphi(z)$,
then $\tilde\varphi^{-1}\circ\alpha$ is a curve
joining $y$ and $z$ in $X_0$.
This implies that 
\begin{align*}
  |y,z|_{X_0^{\rm int}} &\le L(\tilde\varphi^{-1}\circ\alpha)\le (1+\tau_x(\delta,\e))|\tilde\varphi(y),\tilde\varphi(z)| \\
&\le (1+\tau_x(\delta,\e))^2|y,z|_Y,
\end{align*}
showing \eqref{eq:quot|yz|XintY}.

(3)\,
Next suppose $x\in X_0^2$.
Let $\{ p_1,p_2\}:=\eta_0^{-1}(x)$, 
and $U_k:=\eta_0(\mathring{B}(p_k,\e))$\,
$(k=1,2)$.
In this case, we find the $(m,\delta)$-strainer $\{(a_i,b_i)\}_{i=1}^m$ such that 
 $b_m\in\gamma_x^-$ in addition.
In a way similar to the above argument in the case of $x\in {\ca Tk}^{\rm reg}$, 
we can conclude that 
$U_k$ is $\tau_x(\delta,\e)$-almost isometric to open 
subsets of $\R^{m-1}$ with respect to 
the extrinsic metric of $X_0$.
This completes the proof.
\end{proof}

\pmed\n
\subsection{Local structure of $N$}
In this subsection, we investigate the local structure around any thick point in $X$,  and prove
Theorems \ref{thm:regular-ballX(intro)} and
\ref{thm:almostisometry(intro)}.

The following argument works for an arbitrary  finite dimensional Alexandrov space $Y$ with curvature bounded below. We let $m=\dim Y<\infty$ as usual.
%
%

\begin{defn} \label{defn:set-strainer}
For  $\delta>0$ and $1\le k\le m$, a collection  of $k$ pairs  $\{(A_i,B_i)\}_{i=1}^k$ of closed subsets of $Y$ is called a
$(k,\delta)$-\emph{strainer} at $x\in Y$ if it satisfies
\beq
\begin{aligned}\label{defn:AxB}
   \tilde\angle A_ixB_i > \pi - \delta, \quad &
                \tilde\angle A_ixA_j > \pi/2 - \delta, \\
      \tilde\angle B_ixB_j > \pi/2 - \delta,\quad &
                \tilde\angle A_ixB_j > \pi/2 - \delta,
\end{aligned}
\eeq
for all $1\le i\neq j\le k$. 
 The constant 
$\ell:=\min_{1\le i\le k} \{ |A_i,x|, |B_i,x|\}$ is called the {\it length} of the strainer.
\end{defn}

\begin{thm}$($cf.\cite{BGP}$)$\label{thm:ALMOST}
 Let $\{(A_i,B_i)\}_{i=1}^m$ be an
$(m,\delta)$-strainer at $x\in Y$ with length $\ge\ell$.
Then 
the distance map 
$\Phi:B(x,\e)\to \R^m$ defined by 
$\Phi(y)=(|A_1,y|,\ldots,|A_m,y|)$  is a
$\tau(\delta,\e/\ell)$-almost isometry
onto an open subset of $\R^m$
\end{thm}

In the case when $A_i,B_i$ are all points,
Theorem \ref{thm:ALMOST} is proved in \cite[Theorem 9.4]{BGP} (Theorem \ref{thm:almost-isometry}).
In the general case, we need to modify the 
argument there, though the essential idea 
is still the same as \cite[Theorem 9.4]{BGP}.
Since the authors do not know a reference for 
Theorem \ref{thm:ALMOST}, we give the 
proof below for reader's convenience.

\begin{proof}[Proof of Theorem \ref{thm:ALMOST}]
It is easy to verify that $\{ (\Uparrow_y^{A_i},\Uparrow_y^{B_i})\}_{i=1}^{m}$ is an $(m,\delta)$-strainer
of $\Sigma_y(Y)$ in the sense of 
\cite[Definition 9.1]{BGP} for all $y\in B^Y(x,\e)$
if $\e$ is small enough.  From \cite[Lemma 9.3]{BGP}, we have 
\beq \label{eq:Sumcos2=1}
  \biggl|\sum_{i=1}^m\cos^2\angle(\Uparrow_y^{A_i},\xi)-1\biggr| <\tau(\delta)
\eeq
for any $\xi\in\Sigma_y(Y)$.
We set $\Phi=(\Phi_1,\ldots,\Phi_m)$, and
$$
\ell:=\min _{1\le i\le m} \min\{|A_i,x|,|B_i,x|\}.
$$

\begin{slem}\label{slem:set-comparison}
For each $1\le i\le m$ and arbitrary $y,z\in B^Y(x,\e)$, if $\e$ is small enough, we have 
\[
   \biggl|\frac{|\Phi_i(z)-\Phi_i(y)|}{|z,y|} + \cos\angle(\Uparrow_y^{A_i}, \uparrow_y^z)\biggr|
     <\tau(\delta,\e/\ell).
\]
\end{slem}
\begin{proof}
Choose  nearest points $y',z'$ of $A_i$ from $y,z$
respectively.
Since $|\angle y'yz-\wangle y'yz|<\tau(\delta,\e/\ell)$ from Lemma \ref{lem:comparison}, 
we have 
\beq
\begin{aligned} \label{eq:bdabove}
\Phi_i(z)-\Phi_i(y) &\le |z,y'|-|y,y'|\\
 & \le |y,z|(-\cos \angle y'yz+\tau(\delta,\e/\ell)).
\end{aligned}
\eeq
This shows the half of Sublemma \ref{slem:set-comparison}.

\begin{clm}\label{clm:angle(y'yz')}
If
$\e$ is small enough, we have 
$\angle y'yz' <4\delta$.
\end{clm}
\begin{proof}
We proceed by contradiction.
Suppose this does not hold. Then we have 
sequences $y_j$ and $z_j$ converging to $x$
such that 
$\angle y_j'y_jz_j' \ge 4\delta$,
where $y_j',z_j'$ are nearest points of $A_i$ from
$y_j,z_j$ respectively.
Take a nearest point $w_j$ of $B_i$ from $y_j$.
Let $x_1',x_2'\in A_i$ and $w\in B_i$ be limits
of $y_j',z_j'$ and $w_j$ as $j\to\infty$.
Since $x_1',x_2'$ are nearest point of $A_i$ from $x$, \eqref{defn:AxB} implies  $\wangle x_k' x w >\pi-\delta$ 
$(k=1,2)$. Hence we have for large $j$
\[
   \angle y_j' y_j w_j \ge\wangle y_j' y_j w_j >\pi-2\delta, \quad
\angle z_j' y_j w_j\ge 
 \wangle z_j' y_j w_j >\pi-2\delta,
\]
yielding $\angle y_j'y_jz_j'<4\delta$,
a contradiction.
\end{proof}

Since $|\angle z'yz-\wangle z'yz|<\tau(\delta,\e/\ell)$ from  Lemma \ref{lem:comparison}, 
using Claim \ref{clm:angle(y'yz')}, similarly we have
\beq
\begin{aligned} \label{eq:bddbelow}
\Phi_i(z)-\Phi_i(y) &\ge |z,z'|-|y,z'|\\
& \ge |y,z|(-\cos \angle z'yz-\tau(\delta,\e/\ell))\\
 & \ge |y,z|(-\cos \angle y'yz-\tau(\delta,\e/\ell))
\end{aligned}
\eeq
Combining \eqref{eq:bdabove}
and \eqref{eq:bddbelow},
we complete the 
proof of Sublemma \ref{slem:set-comparison}.
\end{proof}

It follows from
\eqref{eq:Sumcos2=1} and Sublemma \ref{slem:set-comparison}  that 
\[
 \biggl|\frac{|\Phi(z)-\Phi(y)|^2}{|z,y|^2} -1\biggr|
 <\tau(\delta,\e/\ell).
\]
Thus $\Phi$ is  
$\tau(\delta,\e/\ell)$-almost isometric.
Finally by the invariance of domain theorem 
in $\R^m$, $\Phi(\mathring{B}(x,\e))$ is open.
This completes the proof of Theorem \ref{thm:ALMOST}.
\end{proof}

\begin{proof}[Proof of Theorem \ref{thm:regular-ballX(intro)}]
 (1)\,
Consider the convergence
\beq \label{eq:rescaling-stability}
    \biggl(\frac{1}{r}Y,x\biggr) \to (T_x(Y), o_x), \quad r\to 0.
\eeq
For $0<r<t_0$, set $C^+_r:=\bigcup_{r\le t\le t_0} C^Y_t$ (see \eqref{eq:CtY}), and let $C^+_*\subset T_x(Y)$ be the limit 
of $C^+_r$ under \eqref{eq:rescaling-stability}.
Note that 
$X=\{ d_{C^+_r}\ge r\}$ and 
$T_xX=\{ d_{C^+_*}\ge 1\}$.

Let us consider the distance functions
$(d^Y_x, d^Y_{C^+_r})$ on $Y$, which converge to
$(d_{o_x}, d_{C^+_*})$ on $T_x(Y)$ under \eqref{eq:rescaling-stability}.
For any $0<\e<1$, consider the metric annuli
\begin{align*}
A^X(x,\e r,r)&:=(d_x, d_{C^+_r})^{-1}([\e r,r]\times[r,\infty))\cap X,\\
A^{T_xX}(o_x,\e,1)&:=(d^Y_{o_x}, d^Y_{C^+_*})^{-1}([\e,1]\times[1,\infty)\cap T_xX.
\end{align*}
Note that $d_{o_x}$ is $(1-\delta)$-regular 
on $\{ 0<d_{o_x}\le 1\}$ for any $\delta>0$.
We show that  
$(d_{o_x}, d_{C^+_*})$ is $(c_0,\delta)$-regular on 
$\{ 0<d_{o_x}\le 1, d_{C^+_*}=1\}=B^{T_xX_0}(o_x,1)\setminus \{ o_x\}$,
where $c_0$ is a constant depending only on 
$x$.
For any $v\in B^{T_xX_0}(o_x,1)\setminus \{ o_x\}$,
$\xi_v^+\in\Sigma_v(T_xY)$ denotes the perpendicular
direction at $v$ to $T_x X$.
Let $\nabla d_{o_x}$ denote the direction at 
$v$ with $\angle(\nabla d_{o_x},\uparrow_v^{o_x})=\pi$.
Note that 
$\Sigma_v(T_xY)$ is the spherical suspension 
$\{ \nabla d_{o_x},\uparrow_v^{o_x}\}*\Sigma_\infty$, where 
$\Sigma_\infty$ is an Alexandrov space
with curvature $\ge 1$ containing $\xi_v^+$.
From the lower semicontinuity of angles together
with the assumption ${\rm rad}(\xi_x^+)>\pi/2$,
we have 
\[
  {\rm rad}(\xi_v^+):=
\sup \{ \angle(\xi_v^+, \xi)\,|\, \xi\in\Sigma_\infty\}>\pi/2.
\]
Therefore from the suspension structure of 
$\Sigma_v(T_xY)$, it is possible to choose
$w\in\Sigma_v(T_xY)$ satisfying
\beq \label{eq:regular-angle(ovw)}
\angle(\xi_v^+,w)>\pi/2+c, \quad  
\angle(\uparrow_v^{o_x},w)>\pi/2+c,
\eeq
where $c>0$ is a constant depending only on 
$x$.
Since $\angle(\xi_v^+,\uparrow_v^{o_x})=\pi/2$,
\eqref{eq:regular-angle(ovw)} shows that 
$(d_{o_x}, d_{C^+_*})$ is $(c_0,\delta)$-regular at $v$ for any $\delta$.

By Theorem \ref{thm:stability-respectful},
we have a homeomorphism
\[
\varphi_1:A^{X^{\rm ext}}(x,\e r,r)\to
A^{T_xX}(o_x,\e,1),
\]
where we set 
$$
A^{X^{\rm ext}}(x,a,b):=A^Y(x,a,b)\cap X.
$$
By iteration, we also have 
homeomorphisms
\[
\varphi_i:A^{X^{\rm ext}}(x,\e^i r,\e^{i-1}r)\to
A^{T_xX}(o_x,\e^i,\e^{i-1})
\]
for all $i\ge 1$.
 Here since $\varphi_i$ are constructed based on
gluing of locally defined homeomorphisms, we may assume that $\varphi_i$  extends 
$\varphi_{i-1}$ (see \cite[Theorem 4.3]{Pr:alexII}). 
Combining all $\varphi_i$\, $(i\ge 1)$, we get 
a homeomorphism
$\varphi:B^X(x,r)\to B^{T_xX}(o_x,1)$.
The required neighborhood is given by 
$U(x,r):=\mathring{B}_{\rm ex}^X(x,r)$.

(2)\,  By Theorem \ref{thm:X1-f}, $x$ is a 
regular point of $Y$. From the assumption on $x$, for any $\delta>0$, choose
an $(m,\delta)$-strainer $\{ (a_i,b_i)\}_{i=1}^{m}$ of $Y$ at $x$ such that 
$\{ (a_i,b_i)\}_{i=1}^{m-1}\subset X_0$ and 
$a_m=C_t^Y$ for some $t>0$.
Consider the distance map $\Phi:\mathring{B}^Y(x,\e) \to \R^m$ defined by 
\[
  \Phi(y):= (|a_1,y|,\ldots, |a_{m-1},y|, |a_m,y|-t),
\]
By Theorem \ref{thm:ALMOST}, 
$\Phi$ is a
$\tau(\delta,\e/\ell)$-almost isometry,
and so is 
the restriction 
$\Phi:\mathring{B}^Y(x,\e)\cap X\to \R^m_+$. 
This completes the proof of Theorem \ref{thm:regular-ballX(intro)}(2).
\end{proof}
\psmall

Now we give an application of Theorem \ref{thm:regular-ballX(intro)}.
The {\it radius} of a metric space $\Sigma$
is defined as 
$$
{\rm rad}(\Sigma):=\inf_{x\in\Sigma} \sup_{y\in\Sigma} |x,y|.
$$

\begin{thm}$($\cite{GP}, \cite{PtPt:extremal}$)$
\label{thm:RadiusSphere}
If $\Sigma$ is a compact Alexandrov space with
curvature $\ge 1$ satisfying ${\rm rad}(\Sigma)>\pi/2$, then it is homeomorphic to a sphere.
\end{thm}

\begin{cor} \label{cor:stability}
 If $x\in  {\ca Tk}$ satisfies  
${\rm rad}(\Sigma_x(X_0)^{\rm int})>\pi/2$, 
then 
$\mathring{B}^{X^{\rm ext}}(x,\e)$ is homeomorphic to 
$\R_+^{m}$ for small enough $\e>0$.
\end{cor}

\begin{proof} By Theorem \ref{thm:RadiusSphere},
$\Sigma_x(X_0)$ is homeomorphic to $\mathbb S^{m-2}$.
Let $\zeta$ be the farthest point of 
$\Sigma_x(X)$ from $\Sigma_x(X_0)$
(and hence from $\xi_x^+$).
Since $d_{\xi_x^+}$ is regular on
$\Sigma_x(Y)\setminus \{ \xi_x^+,\zeta\}$,
Perelman's fibration theorem
 (\cite{Pr:alexII}, \cite{Per}) yields that 
$\Sigma_x(X)$ is homeomorphic to the unit cone over
$\Sigma_x(X_0)$, which is homeomorphic to a disk.
The conclusion  immediately follows  from 
Theorem \ref{thm:regular-ballX(intro)}(1).
\end{proof}  

We now investigate the local structure around any regular thin point.
 
\begin{proof}[Proof of Theorem \ref{thm:almostisometry(intro)}]
First we consider
\pmed\n
{\bf Case 1)\, $x\in {\ca Tn}^{\rm reg}\cap X_0^1$}.\,
\pmed
\begin{figure}
\begin{center}
\begin{tikzpicture}
[scale = 0.5]
\draw  [-, thick] (-4,0) -- (4,0);
\draw [-, thick] (-4,1.3) to [out=-35, in=175] (0,0);
\draw [-, thick] (4,1.3) to [out=215, in=05] (0,0);
\draw [dotted, thick] (-4,1.3) -- (-3.2, 2.2);
\draw [dotted, thick] (4,1.3) -- (3.2, 2.2);
\draw [dotted, thick] (-3.2, 2.2) to [out=-30, in=210]     (3.2,2.2);

\fill (0,0) circle (2pt) node [below] {\small{$x$}};

\fill (-3.5,-0.2) circle (0pt) node [above] {\small{$X$}};
\fill (0,0.2) circle (0pt) node [above] {\small{$Y$}};

\fill (4,1.3) circle (0pt) node [right] {\small{$X_0$}};
\fill (4,0) circle (0pt) node [right] {\small{$\pa Y$}};
\filldraw [fill=lightgray, opacity=.1] 
(-4,0) to (4,0)
 to  [out=90, in=-90]  (4, 1.3)
to [out=215, in=5]  (0,0)
to  [out=175, in=-35] (-4,1.3)
to  [out=90, in=90]  (-4, 0);

\end{tikzpicture}
\end{center} 
\vspace{-0.5cm}
\caption{Thin point}   \label{fig:special-thin}
\end{figure}

%
%

Note that $x\in\pa Y$ and 
$\Sigma_x(X_0)=\Sigma_x(\pa Y)$ (see Fig.\ref{fig:special-thin}).
For any $\delta>0$, choose
an $(m,\delta)$-strainer $\{ (a_i,b_i)\}_{i=1}^{m}$ of $D(Y)$ at $x$ such that 
$\{ (a_i,b_i)\}_{i=1}^{m-1}\subset X_0$,
$a_m=C^Y_{t}$ and $b_m=\tilde C^Y_{t}$ for some 
$t>0$, where 
$\tilde C^Y_{t}$ is another copy of $C^Y_{t}$ in
$D(Y)$.
Set $s_i:=|a_i,x|$ for $1\le i\le m-1$, and
consider the distance maps $\psi:B^{D(Y)}(x,2\e)\to \R^{m-1}$, $\Phi:B^{D(Y)}(x,2\e)\to \R^m$ defined by 
\begin{align*}
&\psi(y):= (|a_1,y|-s_1,\ldots, |a_{m-1},y|-s_{m-1}), \quad \\
& \Phi(y):= (\psi(y), |a_m,y|-t_0).
\end{align*}

By Theorem \ref{thm:ALMOST},
$\Phi$ is a $\tau(\delta,\e/\ell)$-almost isometry
onto an open subset of $\R^m$.

\begin{slem} \label{slem:psi-almostisom}
$\psi$ provides a $\tau(\delta,\e/\ell)$-almost isometric map from $B^{X_0}(x,2\e)$ onto an open subset of $\R^{m-1}$. 
\end{slem}
\begin{proof}
For any $y,z\in B^Y(x,2\e)\cap X_0$,
let $\xi:=\dot\gamma^Y_{y,z}(0)\in\Sigma_y(Y)$.
By Lemma \ref{lem:deviation}, we have 
a direction $v\in\Sigma_y(X_0)$ such that 
$\angle(\xi,v)<O(\e)$.
Since $\angle(\Uparrow_y^{a_m},v)=\pi/2$,
it follows from \eqref{eq:Sumcos2=1} that 
\beq
 \biggl|\sum_{i=1}^{m-1}\cos^2\angle(\Uparrow_y^{a_i},\xi)-1\biggr| <\tau(\delta,\e).
\eeq
 \cite[Theorem 9.4]{BGP} (see also Sublemma \ref{slem:set-comparison}) then yields that 
$\psi:B^Y(x,2\e)\cap X_0\to\R^{m-1}$ is 
$\tau_x(\delta,\e)$-almost isometric.
By Lemma \ref{lem:deviation}, this provides a
$\tau_x(\delta,\e)$-almost isometry
from $B^{X_0}(x,\e)$ onto an open subset of 
$\R^{m-1}$.
\end{proof}

\begin{slem}\label{slem:unique-intersetY}
For any $z\in B^{\pa Y}(x,2\e)$, we have
\[
  \max_{\xi\in\Sigma_z(\pa Y)}   |\angle(\Uparrow_z^{a_m},\xi)-\pi/2|<\tau_x(\e).
\]
\end{slem}
\begin{proof}
Suppose that the sublemma does not hold.
Then there exists a sequence $y_i\in\pa Y$ converging to 
$x$ and $\xi_i\in\Sigma_{y_i}(\pa Y)$ such that 
$|\angle(\Uparrow_{y_i}^{a_m},\xi_i)-\pi/2|\ge \theta_0$ for some
$\theta_0>0$ independent of $i$.
 By \cite[Lemma 1.8]{Ya:conv}, we may assume 
\begin{itemize}
\item
$\angle(\Uparrow_{y_i}^{a_m},\xi_i)<\pi/2-\theta_0\,;$ 
\item there is a sequence $z_i\in \pa Y$ satisfying 
\[
     \angle(\uparrow_{y_i}^{z_i},\xi_i)<o_i,\quad
     |y_i,z_i|=|y_i,x|.
\]
\end{itemize} 
Consider the rescaling limit
\beq \label{eq:limitT(DY)}
   \bigg( \frac{1}{|x,y_i|}D(Y),y_i\biggr) \to (T_x(D(Y)), y_\infty).
\eeq
Let $\xi_\infty:=\uparrow_{y_\infty}^{z_\infty}\in\Sigma_{y_\infty}(T_x(D(Y)))$. Since 
$z_\infty\in(\pa Y)_\infty=T_x(\pa Y)$ and 
$T_x(\pa Y)=\R^{m-1}$,
we have $\xi_\infty\in\Sigma_{y_\infty}(T_x(\pa Y))$. 
From the assumption and the 
lower semicontinuity of angles, we have 
\beq\label{eq:acute-angle}
\angle(\dot\gamma_\infty(0), \xi_\infty)\le\pi/2-\theta_0,
\eeq
where $\gamma_\infty$ is the limit of 
a minimal geodesic $\gamma_i$ 
from $y_i$ to $a_m$. 

Let $y_i'\in X_0$ be the intersection of 
$\gamma_i$  with $X_0$.
Note that $\gamma_{y_i'}^+$ is a 
subarc of $\gamma_i$. Since 
$y_i'$ also converges to $y_\infty$ under the convergence
\eqref{eq:limitT(DY)},  $\gamma_\infty$ is the perpendicular 
$\gamma_{y_\infty}^+$ to 
 $T_x(\pa Y)$ at $y_\infty$.
This is a contradiction to \eqref{eq:acute-angle}.
\end{proof}

Note that the directional derivative 
$\Phi_*$ of $\Phi$ at any $z\in B(x,2\e)$ is $\tau_x(\e)$-almost
preserving the lengths of vectors and the angles
in the sense that
\[
     ||\Phi_*(v)|-1|<\tau_x(\e),\quad 
  |\angle(\Phi_*(v),\Phi_*(w))-\angle(v,w)|<\tau_x(\e)
\]
for all $v,w\in\Sigma_z(D(Y))$  (see \cite[Lemma 9.3,
Theorem 9.5]{BGP}).
\begin{slem} \label{slem:unique-int-paY}
For any $y\in B^{X_0}(x,2\e)$,
there is a unique $f(y)\ge 0$ such that 
the curve $t\mapsto\Phi^{-1}(\psi(y), t)$ \, $(t\ge 0)$ 
meets $\pa Y$ at $t=f(y)$.
\end{slem}
\begin{proof}
Let $(t_1,\ldots,t_m)$ be the canonical coordinates of 
$\R^m$.
Note that 
$|\cos\angle(\Uparrow_z^{a_i},\Uparrow_z^{a_m})|<\tau_x(\delta,\e)$
for all $z\in B^X(x,\e)$. 
Since 
\[
d\Phi_z(\Uparrow_z^{a_m})=
     (-\cos\angle(\Uparrow_z^{a_1},\Uparrow_z^{a_m}),\ldots,
-\cos\angle(\Uparrow_z^{a_{m-1}},\Uparrow_z^{a_m}),-1),
\]
it follows that 
\beq \label{eq:slope-g}
\angle(d\Phi^{-1}_{\Phi(z)}(-\pa/\pa t_m),\Uparrow_z^{a_m}) <\tau_x(\delta,\e).
\eeq
Sublemma \ref{slem:unique-intersetY} yields the conclusion
immediately.
\end{proof}

Let $g:=f\circ\psi^{-1}:B^{\R^{m-1}}(0,\e)\to\R_+$,
and set
\begin{align*}
G_g&:=\{(z,t)\,|\, 0\le t\le g(z),z\in B^{\R^{m-1}}(0,\e)\}, \\
W &:=\Phi^{-1}(G_g). 
\end{align*}
It is easily verified that the canonical  map
$\Phi:W\to G_f$ 
is $\tau_x(\e)$-almost isometry with respect to both
the intrinsic  and the extrinsic metrics.

Finally we show 

\begin{slem}\label{slem:fistau-Lip}
$g$ is $\tau_x(\e)$-Lipschitz.
\end{slem}

\begin{proof}
For any $y\in B^{X_0}(x,\e)$, let $y'\in B^{X_0}(x,\e)$ be a point close to $y$.
Set $u:=(\psi(y),f(y))$, $u':=(\psi(y'),f(y'))$
and 
$z:=\Phi^{-1}(u)$, $z':=\Phi^{-1}(u')$.
By Sublemma \ref{slem:unique-intersetY}, we get
\[
|\angle(\uparrow_z^{a_m},\uparrow_z^{z'})-
\pi/2|
<\tau_x(\e)
\]
if $u'$ is sufficiently close to $u$.
It follows from \eqref{eq:slope-g} that 
\[
\biggl| 
\angle\biggl(-\biggl(\frac{\pa}{\pa t_m}\biggr)_u,\uparrow_u^{u'}\biggr) -\frac{\pi}{2}\biggr|<\tau_x(\e).
\]
This implies that $g$ is $\tau_x(\e)$-Lipschitz.
\end{proof}

\pmed\n
{\bf Case 2)\, $x\in X_0^2\cap X_0^{\rm reg}$}.\,
\pmed

The basic idea is the same as Case 1).
Let $\{ p,q\}:=\eta_0^{-1}(x)$, and 
$\e\ll |p,q|$.
Set $U(x,2\e):=\eta_0(B^{C_0}(p,2\e))$ and 
$V(x,2\e):=\eta_0(B^{C_0}(q,2\e))$.
Let $A_m$ and $B_m$ be  small balls around  $\gamma_x^+(t_0)$ and $\gamma_x^-(t_0)$ in 
$C_{t_0}^Y$ of a fixed radius, respectively.
For any $\delta>0$, choose
an $(m,\delta)$-strainer $\{ (a_i,b_i)\}_{i=1}^{m}$ of $Y$ at $x$
such that $a_m:=A_m$ and $b_m:=B_m$.
Consider the distance maps $\psi:B^{Y}(x,2\e)\to \R^{m-1}$, $\Phi:B^{Y}(x,2\e)\to \R^m$ defined by 
\[
\psi(y):= (|a_1,y|,\ldots, |a_{m-1},y|), \quad
\Phi(y):= (\psi(y), |a_m,y|-t_0).
\]
As in Case 1), $\psi|_{U(x,2\e)}$ and $\Phi$ are  $\tau_x(\e)$-almost isometric.

Replacing  $\pa Y$ by $V(x,\e)$, we have the following 
sublemmas in a similar way:

\begin{slem}\label{slem:unique-intersetY2}
For any $z\in B^Y(x,2\e)\cap V(x,2\e)$, we have
\[
  \max_{\xi\in\Sigma_z(V(x,2\e))}   |\angle(\Uparrow_z^{a_m},\xi)-\pi/2|<\tau_x(\e).
\]
\end{slem}

\begin{slem} \label{slem:unique-int-paY2}
For any $y\in U(x,3\e/2)$,
there is a unique $f(y)\ge 0$ such that 
the curve $t\to\Phi^{-1}(\psi(y), t)$ \, $(t\ge 0)$ 
meets $V(x,2\e)$ at $t=f(y)$.
\end{slem}

Defining  $g:=f\circ\psi^{-1}:B^{\R^{m-1}}(0,\e)\to\R_+$, $G_g$ and $W$ as in Case 1), 
one can verify that the canonical map
$\Psi:W\to G_g$
is  a $\tau_x(\e)$-almost isometry in the same way.

Finally as in  Sublemma \ref{slem:fistau-Lip},
one can show that $g$ is $\tau_x(\e)$-Lipschitz.
This completes the proof of
Theorem \ref{thm:almostisometry(intro)}.
\end{proof}

\begin{conj} \label{conj:stability}
The homeomorphism version of Theorem \ref{thm:almostisometry(intro)}
would also hold for any thin point $x\in X_0^1$ with 
$f_*={\rm id}$
and for any $x\in X_0^2$
if the ball $B^{\R^{m-1}}(0,\e)$
is replaced by $B^{T_xN_0}(o_x,\e)$.
\end{conj}

\pmed
\setcounter{equation}{0}

\section{Local contractibility} \label{sec:loc=cont}\,
In this section, we prove the local Lipschitz contractibility of $X$.
Using Perelman's local stability theorem
for $Y$ together with the projection
$\pi:Y\to X$, which is a deformation retraction,
we can easily verify that for any $x\in X_0$, $B^X(x,\e)$ is 
contractible in  a slightly larger metric ball
for small enough $\e>0$.
To show the existence of a contractible
neighborhood, we make use of 
the gradient flows for semi-concave functions
(see \cite{Pet}).

In \cite{Per}, Perelman constructed a strictly concave function on a small neighborhood of any  point of an Alexandrov space with 
curvature bounded below.
Kapovitch modified Perelman's construction
in \cite{Kap}.
First we roughly recall  their construction
in our Alexandrov space $Y$.
\psmall
Let $x\in Y$.
For positive numbers $\theta$ and $\delta$ with 
$\delta\ll \theta\le1/100$,
take a small enough positive number $r=r(x,\theta,\delta)$.
Let $\{ x_{\alpha}\}_{\alpha\in A}$ be a maximal $\theta r$-discrete  system in 
$S(x,r)$, and take a maximal $\delta r$-discrete 
 system $\{ x_{\alpha\beta}\}_{\beta\in A_\alpha}$
in $B(x_{\alpha},\theta r)\cap S(x,r)$.
Taking small enough $r$,
we may assume that 
\begin{itemize}
\item $\{\uparrow_x^{x_\alpha}\}_{\alpha\in A}$ is 
 $\theta/2$-discrete and $2\theta$-dense  in $\Sigma_x(Y)\,;$
\item $\{\uparrow_x^{x_{\alpha\beta}}\}_{\beta\in A_\alpha}$ is $\delta/2$-discrete and
$2\delta$-dense in $B^{\Sigma_x(Y)}(\uparrow_x^{x_{\alpha}}, \theta)$.
\end{itemize}

For  $0<\e \ll r$ and $0<\zeta\ll 1$, let 
$\phi:[0,2r]\to\R_+$ be the $C^1$-function 
 with $\phi(0)=0$ 
satisfying  
\beq \label{eq:Sec6-phi}
\begin{aligned}
\begin{cases}
\text{$\phi' = 1$ \hspace{1cm} on $[0,r-2\e]$}\\
\text{$\phi' = 1-\zeta$  \hspace{0.3cm} on $[r+2\e,2r]$}\\
\text{$\phi'' = -\zeta/4\e$ \hspace{0.08cm} on $[r-2\e, r+2\e]$}. 
\end{cases}
\end{aligned}
\eeq

We consider the following $1$-Lipschitz functions $g_{\alpha}$ and $g$ on $M$:
\beq \label{eq:def-concavef}
   g_{\alpha} := \frac{1}{\# A_\alpha} \sum_{\beta\in A_{\alpha}} 
                                \phi(|x_{\alpha\beta}, \,\cdot\,|), \qquad
  g:= \min_{\alpha\in A} g_{\alpha},
\eeq

Here we need the following result.

\begin{thm}$($\cite{Per},\cite{Kap}$)$\label{thm:convex-nbd}
There exists $\e_{r,\theta,\delta,\zeta}>0$ such that for any $0<\e\le\e_{r,\theta,\delta,\zeta}$,
the function $g$ defined above is  $(-1)$-concave on $B(x,2\e)$
having a unique maximum at $x$ in $B(x, 2\e)$.
\end{thm}

Let $s_0:=\min_{B(x,\e)}\,g$. We define a convex neighborhood  $W$ of $x$ as 
\beq\label{eq:def=W}
       W:= \{ g\ge s_0\}\cap B(x,2\e).
\eeq
We show that
\beq \label{eq:W-inclusion}
B(x,\e)\subset W\subset B(x,(1+\tau(\theta)+\tau(\e/r))\e).
\eeq
In fact, the first inclusion is obvious.
For the second inclusion, 
for any $y\in W\setminus \{ x\}$, choose $\alpha\in A$ with $\angle(\uparrow_x^y,\uparrow_x^{x_\alpha})<2\theta$.
Since 
$\angle(\uparrow_x^y,\uparrow_x^{x_{\alpha\beta}})<4\theta$,
by the law of cosine, we have for all $\beta\in A_\alpha$
\beq \label{eq:zxab-cosine}
||y,x_{\alpha\beta}| - (r-|x,y|\cos\wangle yxx_{\alpha\beta})|\le O((|x,y|/r)^2).
\eeq
Choose $\beta\in A_\alpha$ such that 
$|y,x_{\alpha\beta}|$ is maximal among the elements of $A_\alpha$. From  \eqref{eq:zxab-cosine},
we then have 
\beq
\begin{aligned}\label{eq:aboveg(z)}
\phi(r-|x,y|&\cos\wangle yxx_{\alpha\beta}+O((|x,y|/r)^2) \\
   &\ge g_\alpha(y)\ge g(y) \ge s\ge\phi(r-\e).
\end{aligned}
\eeq
Since $\phi$ is monotone, this implies 
\begin{align*}
       |x,y| &\le  (\cos\wangle yxx_{\alpha\beta})^{-1}(\e+O((|x,y|/r)^2)  \\
   &\le (1+\tau(\theta)+ \tau(\e/r))\e
\end{align*}
as required.
\psmall
\begin{prop}\label{prop:1-Lip-contractibleAlex}
$W$ is $(1+\zeta, \e)$-Lipschitz contractible
to $x$ for small enough $\theta$, $\zeta$  and  $\e/r$.
\end{prop}
\psmall
\begin{rem} \label{rem:extMY}
Proposition \ref{prop:1-Lip-contractibleAlex}
is related with  \cite[Theorem 1.2]{MY:SLC}.
\end{rem}

\begin{proof}[Proof of Proposition \ref{prop:1-Lip-contractibleAlex}]
(1)\, First we show that
\beq \label{eq:grad-dg}
dg(\uparrow_y^x)\ge (1-\zeta)(1-\tau(\theta)-\tau(\e/r))
\eeq
for all $y\in W\setminus \{ x\}$.

Take $\alpha'\in A$ and $\beta'\in A_{\alpha'
}$ such that $g(y)=g_{\alpha'}(y)$
and $|y,x_{\alpha'\beta'}|$ is minimal 
among the elements of $A_{\alpha'}$. 
From $g_\alpha(y)\ge g_{\alpha'}(y)$, we have 
$|y,x_{\alpha\beta}|\ge|y,x_{\alpha'\beta'}|$.
This implies 
\begin{align*}
\phi(r-|x,y| & \cos\wangle yxx_{\alpha\beta}+O((|x,y|/r)^2)  \\
  &\ge 
\phi(r-|x,y|\cos\wangle yxx_{\alpha'\beta'}-O((|x,y|/r)^2),
\end{align*}
which yields 
$\cos\wangle yxx_{\alpha'\beta'}\ge 
\cos 4\theta -\tau(|x,y|/r)$. Thus we have
\beq\label{eq:wangle(yxxab}
\wangle yxx_{\alpha'\gamma'}\le \tau(\theta)+
\tau(|x,y|/r)
\eeq
for all $\gamma'\in A_{\alpha'}$.
Since $\wangle yx_{\alpha'\gamma'}x\le \tau(\e/r)$, we obtain
\beq\label{eq:angle-xyac}
\wangle xyx_{\alpha'\gamma'}\ge \pi-\tau(\theta)-\tau(\e/r),
\eeq
which yields \eqref{eq:grad-dg}.

\pmed
\n
(2) \, By \eqref{eq:grad-dg}, the gradient flow $\Phi(y,t)$ of $g$
starting at any point $y\in W$ reaches $x$ at a finite time $T>0$.
We show that $T\le (1+2\zeta)\e$.

By \eqref{eq:grad-dg}, we obtain 
\begin{align*}
    \phi(r)-s_0& \ge g(\Phi(y,T))-g(\Phi(y,0))\\
      &\ge (1-\zeta)(1-\tau(\theta)-\tau(\e/r))T.
\end{align*}
 Take $y_0\in B(x,\e)$ with $g(y_0)=s_0$.
Applying \eqref{eq:zxab-cosine} and \eqref{eq:wangle(yxxab} for $y_0$ in place of $y$,
we get
\begin{align*}
    \phi(r)-s_0& \le \phi(r)-\phi(r-\e\cos
     (\tau(\theta)+\tau(\e/r)) +O((\e/r)^2))\\
     & \le \e(1-\tau(\theta)+\tau(\e/r)).
\end{align*}
Combining the above two inequalities, 
we get 
\beq\label{eq:boundT}
T\le \frac{\e(1-\tau(\theta)+\tau(\e/r))}{(1-\zeta)(1-\tau(\theta)-\tau(\e/r))} \le \e(1+2\zeta)
\eeq 
for 
small enough $\theta$ and $\e/r$. 

Since $g$ is concave and $1$-Lipschitz on $B(x,2\e)$,   Proposition \ref{prop:lip} shows that 
$\Phi:W\times [0,T]\to W$ is $1$-Lipschitz.
By \eqref{eq:boundT}, it 
pvovides a $(1, (1+2\zeta)\e)$-Lipschitz
strongly deformation retraction of $W$ to $x$. 
This completes the proof.
\end{proof}
%
\pmed  

 We denote by $\dot \Phi(y,t)=\nabla_{\Phi(y,t)}g$ the right derivative of 
$\Phi(y,t)$ with respect to $t$.
By \eqref{eq:gradient}, \eqref{eq:grad-dg} and the first
variation formula, 
we have the following immediately.

\begin{cor} \label{cor:grad-g} For any $y\in W\setminus \{ x\}$, we have 
\beq \label{eq:grad-normangle-dg}
\begin{cases}
\begin{aligned}
&|\dot \Phi(y,t)|\ge (1-\zeta)(1-\tau(\theta)-\tau(\e/r)), \quad \\
&  \angle(\dot \Phi(y,t),\uparrow_{\dot \Phi(y,t)}^x)<
\tau(\theta,\zeta)+\tau(\e/r),\\
& \frac{d}{dt}d_p^Y(\Phi(y,t)) < -1+\tau(\theta,\zeta)+\tau(\e/r).
\end{aligned}
\end{cases}
\eeq
\end{cor}  

\pmed

\begin{proof}[Proof of Theorem  \ref{thm:contractible}]
For $x\in X\setminus X_0$,
Proposition \ref{prop:1-Lip-contractibleAlex}
immediately implies the conclusion.

For any $x\in X_0$ and $\zeta>0$, choose 
$\e=\e_{1,r,\theta,\delta,\zeta}$,  the $(-1)$-concave function  $g$  and the convex neighborhood $W$ of $x$ as  above  togetehr with the 
$(1,(1+\zeta)\e)$-Lipschitz strongly deformation 
retraction $\Phi:W\times [0,(1+\zeta)\e]\to W$
of $W$ to $x$ defined by the gradient flow of $g$.
%
We consider $W$ as an Alexandrov space, and
set 
\[
      W_X:=W\cap X.
\]

\begin{lem} \label{lem:G1=pi}
For any $(y,t)\in W_X\times [0,\e]$, we have 
$$
\pi\circ \Phi(y,t)\in W_X.
$$
\end{lem}
\begin{proof} 
For any $(y,t)\in W_X\times [0,\e]$, 
if $\Phi(y,t)\in (W\setminus X)\cap B(x,\e/2)$,
the triangle inequality implies $|x,\pi(\Phi(y,t))|<\e$
and hence $\pi(\Phi(y,t))\in W_X$. Thus
we may always assume $\Phi(y,t)\in W\setminus B(x,\e/2)$.

Consider the distance function 
$d_X$ from $X$ in $Y$.
Lemma \ref{lem:dX-est}  together with \eqref{eq:grad-normangle-dg} implies
\beq\label{eq:angle(F,TC)}
\angle(\dot \Phi(y,t),  T_{\Phi(y,t)}(d_X^{-1}([0,s])))<\tau(\theta,\zeta)+\tau(\e/r)+\tau_x(\e).
\eeq
Choose a fixed function $\tau_{x}(\theta,\zeta,\e/r)$
such that 
$$
\tau(\theta,\zeta)+\tau(\e/r)+\tau_x(\e)
<\tau_{x}(\theta,\zeta, \e/r),
$$
where the LHS of the above inequality represents
the RHS of \eqref{eq:angle(F,TC)}.
We show 
\beq\label{eq:|z,X|<small}
    |\Phi(y,t),X|\le \tau_{x}(\theta,\zeta, \e/r)t
\eeq
for all $t\ge 0$ with $\Phi(y,t)\in W\setminus B(x,\e/100)$.
Let $t_0\ge 0$ be the supremum of such $u\ge 0$
satisfying \eqref {eq:|z,X|<small}  and $\Phi(y,t)\in  W\setminus B(x,\e/100)$ for all $t\in [0,u]$.
We may assume  
$|\Phi(y,t_0),X|= \tau_{x}(\theta,\zeta, \e/r)t_0$
and  $\Phi(y,t)\in  W\setminus B(x,\e/100)$.
Set $s_0:=|\Phi(y,t_0),X|$.
From \eqref{eq:angle(F,TC)} for $t=t_0$, we obtain
for any small enough $\delta>0$, 
\begin{align*}
   |\Phi(y,t_0+\delta), X|&=s_0+ |\Phi(y,t_0+\delta), d_X^{-1}([0,s_0])| \\
& \le s_0+(\tau(\theta)+\tau(\e/r)+\tau_x(\e))\delta\\
&  < \tau_{x}(\theta,\zeta, \e/r)(t_0+\delta).
\end{align*}
This contradicts the choice of $t_0$, and proves
\eqref{eq:|z,X|<small}.

Since $\frac{d}{dt}g(\Phi(y,t))=|\nabla g|^2\ge
(1-\zeta)^2(1-\tau(\theta)-\tau(\e/r))^2$, we have
\[
    g(\Phi(y,t))\ge g(y)+ t(1-\zeta)^2(1-\tau(\theta)-\tau(\e/r))^2.
\]
On the  other hand, since $g$ is $1$-Lipschitz, we have
\begin{align*}
   g(\pi\circ \Phi(y,t)) &\ge g(\Phi(y,t))-|\Phi(y,t), \pi\circ\Phi(y,t)|  \\
& \ge g(y) + t(1-\zeta)^2(1-\tau(\theta)-\tau(\e/r))^2-\tau_{x}(\theta,\zeta, \e/r)t \\
&\ge g(y) +   t\{ (1-\zeta)^2-2\tau_{x}(\theta,\zeta, \e/r)\}
\ge g(y)
\end{align*}
if $\theta,\zeta, \e/r\ll 1-\zeta$, 
yielding $\pi\circ \Phi(y,t)\in W$ 
\end{proof}

From here on, we use  a general symbol 
\[
    \tau_{x}(\theta,\zeta,\e/r)=
\tau(\theta,\zeta)+\tau(\e/r)+\tau_x(\e).
\]
Since 
$\pi: C^Y_{[0,\tau_{x}(\theta,\zeta, \e/r)\e]}\to X_0$
is $\phi(\tau_{x}(\theta,\zeta, \e/r)\e)^{-1}$-Lipschitz, by Lemma \ref{lem:G1=pi},
\beq\label{eq:retracionH}
H:=\pi\circ \Phi|_{W_X\times [0,(1+\zeta)\e]}
:W_X\times [0,\e]\to W_X
\eeq
provides a
$(1+\tau_{x}(\theta,\zeta, \e/r)\e, (1+\zeta)\e)$-Lipschitz strong deformation retraction 
of $W_X$  to $x$.
If $\theta$, $\zeta$ and $\e/r$ are chosen small enough,
$W_X$ is  $(1+\zeta,\e)$-Lipschitz strongly contractible to $x$.
This completes the proof of Theorem  \ref{thm:contractible}
\end{proof}  

\begin{figure}
\begin{center}
\begin{tikzpicture}
[scale = 0.5]

\draw (0,0) circle [x radius=5,y radius=1.7];
\draw  [-] (-5,0)  to [out=70, in=110] (5,0);

\fill (-0.5,0) circle (2pt) node [below] {\small{$x$}};
\fill (4,0) circle (2pt) node [below] {\small{$y$}};
\draw  [-,thin] (4,0)  to [out=150, in=30] (-0.5,0);

\draw  [-,thin] (-0.5,0)  to [out=-10, in=190] (1,0);
\draw  [-,thin] (1,0)  to [out=10, in=170] (2,0);
\draw  [-,thin] (2,0)  to [out=-10, in=190] (3,0);
\draw  [-,thin] (3,0)  to [out=10, in=170] (4,0);

\fill (2,0.65) circle (2pt);
\fill (2,0.50) circle (0pt) node [above] {\tiny{$\Phi(y,t)$}};
\fill (2,0) circle (2pt) node [below] {\tiny{$H(y,t)$}};
\draw[thin] (2,0.65)--(2,0);
\draw[thin] (1,0.55)--(1,0);
\draw[thin] (3,0.45)--(3,0);
\draw[thin] (0.25,0.38)--(0.25,-0.1);
\draw[->,thin] (0.65,0.5)--(0.35, 0.4);
\draw[->,thin] (0.65,-0.06)--(0.35, -0.06);
\fill (-3,0) circle (0pt) node [below] {\small{$W_X$}};
\fill (-3.2,2) circle (0pt) node [above] {\small{$W$}};
\end{tikzpicture}
\end{center}  
\caption{Homotopy $H$ }
\end{figure}
\pmed

The existence of the right derivative  $\dot H(y,t)$  is guaranteed by  Lemmas \ref{lem:derivative-pi} and \ref{lem:derivative-pi-0}.
In what follows, we present some geometric properties of the homotopy $H$ corresponding to ones 
for $\Phi$ given in Corollary \ref{cor:grad-g}.
These will be useful in 
Sections \ref{sec:LHS} and \ref{sec:LLHS}.

\begin{lem}\label{lem:gradH-6}
For any $(y,t)\in W_X\times\R_+$ with $H(y,t)\in W\setminus B(x,\e/100)$, we have 
\[
 \angle(\dot\gamma^Y_{H(y,t),x}(0), \dot H(y,t)) <  \tau_{x}(\theta,\zeta,\e/r)
\]
\end{lem}
\begin{proof}
Suppose this does not hold. Then we have sequences
$\theta_i\to 0$, $\zeta_i\to 0$, $\e_i/r_i\to 0$, $(y_i,t_i)\in W_i\times\R_+$ with 
$H_i(y_i, t_i)\in W_i\setminus B(x, \e_i/100)$ such that 
\beq\label{eq:angle(H,gammaH)}
     \angle(\dot H_i(y_i,t_i),\dot\gamma_{H_i(y_i,t_i),x}^Y)\ge \omega>0,
\eeq
where $W_i$ is a convex neighborhood of $x$ defined 
for $r_i$, $\theta_i\gg\delta_i$, $\zeta_i$, $\e_i$ as in \eqref{eq:def=W} by a strictly concave function $g_i$,
$H_i$ is 
defined as in \eqref{eq:retracionH} via  
the gradient flow $\Phi_i$ for $g_i$
and $\omega$ is a constant independent of $i$.
Set  $z_i:=\Phi_i(y_i,t_i)$, $w_i:=\pi(z_i)=H_i(y_i,t_i)$ and
$s_i:=|z_i,w_i|$.  If $s_i=0$, then   \eqref{eq:grad-normangle-dg} and Lemma \ref{lem:deviation} imply 
\begin{align*}
\angle(\dot \Phi_i(z_i,0),&\dot H_i(y_i,t_i))=\angle(\dot \Phi_i(z_i,0),T_{z_i}(X)) \\
&\le\angle(\dot \Phi_i(z_i,0),\dot\gamma^Y_{z_i,x}(0))
  +\angle(\dot\gamma^Y_{z_i,x}(0),(\pi\circ\gamma^Y_{z_i,x})'(0)) \\
&<\tau(\theta_i,\zeta_i)+\tau(\e_i/r_i)+O(\e_i).
\end{align*}  
Again by \eqref{eq:grad-normangle-dg},
we have a 
contradiction to \eqref{eq:angle(H,gammaH)}.

From now,  we assume $s_i\neq 0$, and 
set $\nu_i:=\dot\gamma_{w_i,x}(0)$.
Let $v_i\in\Sigma_{w_i}(Y)$ be the direction 
defined by $\pi(\gamma_{z_i,x})$.
By \eqref{eq:grad-normangle-dg}, we have
$\angle(\dot H(y_i,t_i), v_i)\to 0$.
Since $\angle(\nu_i,v_i)\to 0$ by Lemma \ref{eq:dpi-x}, we have a contradiction to \eqref{eq:angle(H,gammaH)}.
\end{proof}

\begin{slem}\label{slem:|H|=1}
For any $(y,t)\in W_X\times\R_+$ with $H(y,t)\in W\setminus B(x,\e/100)$, we have 
\[
 |\dot H(y,t)| >1- \tau_{x}(\theta,\zeta,\e/r).
\]
\end{slem}
\begin{proof} 
First consider the case  $\Phi(y,t)\in X$.
Applying Corollary \ref{cor:grad-g} and 
Lemma \ref{lem:deviation}, 
we obtain 
\beq
\begin{aligned}\label{eq:alpha<tau}
\alpha &:=\angle(\dot\Phi(y,t),\Sigma_{\Phi(y,t)}(X))\\
  &\le \angle(\dot\Phi(y,t),\dot\gamma^Y_{\Phi(y,t),x}(0))+ \angle(\dot\gamma^Y_{\Phi(y,t),x}(0),
   (\pi\circ\gamma^Y_{\Phi(y,t),x})'(0))\\
&\le \tau(\theta,\zeta)+\tau(\e/r).
\end{aligned}
\eeq  
The conclusion follows from Lemma \ref{lem:derivative-pi-0}.

%
%

Next suppose  $\Phi(y,t)\in Y \setminus X$.
By \eqref{eq:|z,X|<small}, we have 
\beq\label{eq:s<taut}
     s:=|\Phi(y,t),X|\le \tau_{x}(\theta,\zeta,\e/r)t< \tau_{x}(\theta,\zeta,\e/r)\e.
\eeq
Set $z:=\Phi(y,t)$ and $w:=\pi(z)$.
Choose $u\in Y$ with $|u,w|=\e$  such that 
$\wangle xwu>\pi-\tau_x(\e)$.
From a limit argument with the lower semicontinuity 
of angle, we have 
$\angle zwz \ge \pi/2-\tau_x(\e)$. It follows from 
Proposition \ref{lem:comparison} applied to the 
$(1,\tau_x(\e))$-strainer $(x,u)$ together with 
\eqref{eq:s<taut} that
$\angle xzw\ge \pi/2-\tau_{x}(\theta,\zeta,\e/r)$.
Namely we have 
\[
\beta:=\angle(\dot\Phi(y,t),\Sigma_{\Phi(y,t)}(C^Y_s))<\tau_{x}(\theta,\zeta,\e/r).
\]
By Lemma \ref{lem:derivative-pi}, we then have
\begin{align*}
|\dot H(y,t)| &= \phi(s)^{-1}|\dot\Phi(y,t)|\cos\beta
\ge 1-\tau_{x}(\theta,\zeta,\e/r).
\end{align*}
\end{proof}

Lemma \ref{lem:gradH-6} and Sublemma \ref{slem:|H|=1} imply the following immediately.

\begin{cor}\label{cor:dp(dotH)=-1}
For any $(y,t)\in W_X\times\R_+$ with $H(y,t)\in W\setminus B(x,\e/100)$, we have 
\[
 \frac{d}{dt}d_p^Y(H(y,t)) <-(1- \tau_{x}(\theta,\zeta,\e/r)).
\]
\end{cor}

\begin{rem} \label{rem:not-convex} 
The neighborhood $V=W_X$ obtained in 
Theorem \ref{thm:contractible}
 is not necessarily convex in $X$. 
We left the construction of such an 
example to the readers.
\end{rem} 
\pmed

\setcounter{equation}{0}
\section{Lipschitz homotopy stability}  \label{sec:LHS}

Our methods developed so far can be applied to 
the general collapse/convergence 
in  $\ca M=\ca M(n,\kappa,\nu,\lambda,d)$.

For $1\le m\le n$, let 
$\ol{\ca M}(m)$ denote the set of all
$m$-dimensional spaces 
$X$ in $\ol{\ca M}$, and for $v>0$
let
$\ol{\mathcal M}(m,v)$ denote the set of all 
$X\in \ol{\mathcal M}(m)$ with 
$\ca H^m(X)\ge v$.
We will see in Lemma \ref{lem:vol-conv}
that $\ol{\mathcal M}(m,v)$  is compact with 
respect to the Gromov-Hausdorff distance.
In this section, we establish the Lipschitz 
homotopy convergence in 
$\ol{\mathcal M}(m,v)$.
\psmall

\subsection{Preliminary discussion}
We begin with  preliminary discussion on the convergence in $\ol{\mathcal M}(m,v)$.

\pmed
\n


\begin{defn}\label{defn:extension}
We call an Alexandrov space $Y$ is an 
	{\it extension} of $N \in\ol{\mathcal M}$
if $N$ is the limit of a sequence $M_i\in \ca M$ and $Y$ is the limit of the extension $\tilde
M_i$ of $M_i$.
\end{defn}

The following example shows
that an extension of $N \in\ol{\mathcal M}$
is not uniquely determined.

\begin{ex}  \label{ex:Ext-nonGH}
\,(1)\,
Let $D_i$ be a sequence of nonnegatively curved 
two-disks with totally geodesic boundary 
converging to the interval $I=[0,1]$,
where $\pa D_i\to \{ 0\}$.
Let $M_i:=D_i\times \mathbb S^1$
and $M_i':=I\times \mathbb S^1\times S^1_{1/i}$.
Both $M_i$ and $M_i'$ converge to $X:=I\times\mathbb S^1$.
Let $Y$ and $Y'$ be the extensions of  $X$ determined by the limits of the extensions $\tilde M_i$  and $\tilde M_i'$ respectively.
Note that $Y$ is the gluing of 
$[0,1]\times \mathbb S^1$ and $[0,t_0]\times_\phi 
\mathbb S^1$ along $\{ 0\}\times \mathbb S^1$
and $Y'$ is the gluing of 
$I\times \mathbb S^1$ and $[0,t_0]\times_\phi 
(\pa I\times \mathbb S^1)$ along 
$\pa I\times \mathbb S^1=\{ 0\}\times\pa I\times \mathbb S^1$.


(2)\, 
Let $D_i$ be as in (1).
Here we assume that $D_i$ is rotationally symmetric with center $p_i$ and $D_i$ contains a small  rotationally symmetric disk
$D_i'$ with center $p_i$ and totally geodesic boundary. We assume that
there is an isometry $\varphi_i:D_i\setminus \mathring{D}_i'\to \mathbb S^1_{1/i}\times [0,t_0]$
sending $\pa D_i$ to  $\mathbb S^1_{1/i}\times 0$ and $\delta_i:=|p_i,\pa D_i'|\to 0$ as $i\to\infty$.
Let $(\theta,r)$ be the coordinates of $D_i$
with $\pa D_i=\{ r=0\}$ providing 
the product on $\mathbb S^1_{1/i}\times [0,t_0]$.
Let $\phi_i:[0,t_0+\delta_i]\to \R_+$ the extension of 
$\phi$ such that $\phi_i=\e_0$ on $[t_0, t_0+\delta_i]$.
We consider the warped product 
\[
     M_i:=D_i\times_{\phi_i(t)} \mathbb S^1,\quad
g_{M_i}=g_{D_i}+\phi_i(t)^2 g_{\mathbb S^1}.
\]
We also consider the product $M_i':=\mathbb S^1\times [0,1/i]$. 
Both the extensions $\tilde M_i$ and $\tilde M_i'$ 
converge to 
$Y:=[-t_0,t_0]\times_{\tilde\phi}\mathbb S^1$,
where $\tilde\phi(t)=\phi(|t|)$.
Note that $M_i$ and $M_i'$ converge to 
$X:=[0,t_0]\times_\phi \mathbb S^1$
and $X':=\mathbb S^1$ respectively. 
Note that the convergence $M_i\to X$ is a non-inradius 
collapse and $M_i'\to X'$ is an inradius collapse.
\end{ex}

\begin{rem}\label{rem:XnocontrolY}
Let $X_i\in \ol{\ca M}(m,v)$ converge to 
$X\in \ol{\ca M}(m,v)$, and let $Y_i$ and $Y$ be extensions of 
$X_i$ and $X$ respectively.
As shown in Example \ref{ex:Ext-nonGH},
the $GH$-convergence $X_i\to X$ does not necessarily 
imply the $GH$-convergence $Y_i\to Y$.
Similarly, the $GH$-convergence $Y_i\to Y$ does not necessarily 
imply the $GH$-convergence $X_i\to X$.
\end{rem}

Concerning the above remark, we have the following easy lemma.
 
\begin{lem} \label{lem:cY-conv}
Let $Y$ and $Y_i$ be extensions of 
$X\in \ol{\ca M}$ and $X_i\in \ol{\ca M}$
respectively.
Suppose  the $GH$-convergence
\beq \label{eq:GHconvYY}
     Y_i \to Y.
\eeq
If $C_{t_0}^{Y_i}$ converges to $C_{t_0}^Y$
under \eqref{eq:GHconvYY},  
then  $X_i$ converges to $X$
under \eqref{eq:GHconvYY}.
\end{lem}
\begin{proof}
From the assumption, $Y_i\setminus{\rm int} X_i=B(C^{Y_i}_{t_0},t_0)$ converges to
$Y\setminus{\rm int} X=B(C^{Y}_{t_0},t_0)$ under 
\eqref{eq:GHconvYY}.
This yields the conclusion
$X_i\to X$  immediately.
\end{proof}

\begin{prop} \label{prop:homotopy=}
For any $X\in \ol{\ca M}(m,v)$, there exists $\e=\e_X>0$
such that if $X'\in\ol{\ca M}(m,v)$ satisfies 
$d_{GH}(X,X')<\e$, then it has the same homotopy type
as $X$.
\end{prop}

\begin{proof}
Suppose the conclusion does not hold.
Then we have a sequence$X_i$ in $\ol{\ca M}(m,v)$
converging to $X\in\ol{\ca M}(m,v)$ such that $X_i$ does not have the 
same homotopy type as $X$ for any $i$.
Let $Y_i$ be any extension of  $X_i$.
We may assume that $(Y_i, X_i)$ converges to 
$(Y,W)$, where $W$ is a closed subset of 
an Alexandrov space $Y$.

\begin{lem}\label{lem:YextX}
$Y$ is an extension of $X$.
\end{lem}
\begin{proof}
For each $i$, take a sequence $M_{ij}$ in $\ca M$
converging to $X_i$ 
such that $\tilde M_{ij}$ converges to $Y_i$ 
as $j\to\infty$.
If we take large enough $j=j(i)$, setting $M_i:=M_{ij(i)}$, we may assume
that 
\begin{itemize}
\item $M_i$ converges to $X\,;$
\item $(\tilde M_i, M_i)$ converges to $(Y,W)$.
\end{itemize}
By Proposition \ref{prop:intrinsic},
the intrinsic metric $W^{\rm int}$ is isometric to 
$X$. This completes the proof.
\end{proof}  

Next we show  $\dim Y_i=\dim Y$.
Suppose this does not hold. Then the only possibility is that  
$\dim Y_i=m+1$,  and $\dim Y=m$.
Let $M_{ij}$ and $M_i$ be as in the proof of Lemma \ref{lem:YextX}. It follows from 
$\dim Y=\dim X$ that ${\rm inrad}(M_i)>c>0$
for some constant $c$ in dependent of $i$.
Moreover we may also assume 
${\rm inrad}(M_{ij})>c>0$ for any large enough $i$
and $j$. However this yields
$\dim X_i=\dim Y_i=m+1$, a contradiction.

Thus by Theorem \ref{thm:stability},
$Y_i$ is homeomorphic to $Y$ for large $i$.
Since $X_i$ and $X$ have the same Lipschitz homotopy types as $Y_i$ and $Y$ respectively (see Proposition \ref{prop:retractionYX}),
$X_i$ has the same homotopy type as $X$ for large $i$,
which is a contradiction.
\end{proof}

\begin{lem}\label{lem:volu-XY}
For any $X\in\ol{\ca M}(m,v)$, any extension $Y$
of $X$ satisfies $\ca H^{\dim Y}(Y)\ge v_0>0$,
where $v_0=v_0(m,\lambda,t_0,\e_0,v)$.
\end{lem}

 Note that $\dim Y=m$ if $X$ is the limit
of an non-inradius collapse/convergence
and $\dim Y=m+1$ if $X$ is the limit
of an inradius collapse.

\begin{proof}[Proof of Lemma \ref{lem:volu-XY}]
The lemma is obvious when ${\rm int}X$ is nonempty.
Suppose $X=X_0$. In this case we have $\ca H^m(C_0)\ge H^m(X)\ge v$. From the warped product 
structure on $C$, we obtain 
\[
 \ca H^{m+1}(Y)=\ca H^{m+1}(C)=\ca H^m(C_0)\int_0^{t_0}
\phi(t)^m(t)\,dt=cv,
\]
where $c>0$ is a uniform constant.
\end{proof} 
\psmall
\subsection{Local $C$-Lipschitz contractibility}
From the proof of Proposition \ref{prop:homotopy=},
we cannot conclude that $X_i$ has the same Lipschitz
homotopy type as $X$.
In what follows, more strongly, we provide a  quatitative Lipschitz homotopy convergence $X_i\to X$,
and prove Theorem \ref{thm:Lip-homotopy}.

A metric space is \textit{locally $C$-doubling} if any $\epsilon$-ball is covered by at most $C$ metric balls of radius $\epsilon/2$, where $0<\epsilon<C^{-1}$.

A metric space $X$ is called {\it locally $C$-Lipschitz contractible} if  for any $x \in X$ and $0<\epsilon<C^{-1}$, there exists a $(C,\epsilon)$-Lipschitz  contractible domain $V$ in $X$ containing  $B(x,\epsilon)$ (see \cite{FMY}).

For the proof of Theorem \ref{thm:Lip-homotopy},
we use the following general result
established in \cite{FMY},
which is a Lipschitz version of \cite{Peter}.

\begin{thm}[\cite{FMY}]\label{thm:Peter}
For any $C>0$ there exists $\tilde C>0$ satisfying the following:
Let $X$ and $X'$ be metric spaces that are locally $C$-Lipschitz contractible and locally $C$-doubling.
If their Gromov-Hausdorff distance is less than $\tilde C^{-1}$, then we have
\[d_{\tilde C\mathchar`-\mathrm{LH}}(X,X')<\tilde Cd_\mathrm{GH}(X,X').\]
\end{thm}

We begin with 

\begin{lem}\label{lem:C-doubling}
There exists $C>0$ such that 
any $N\in\ol{\ca M}$ is $C$-doubling.
\end{lem}
\begin{proof}
Let $Y$ be an Alexandrov space extending $N$. 
Since $Y$ is $C$-doubling for some $C$ 
by the Bishop-Gromov comparison theorem,
Lemma \ref{lem:XbiLipX} implies the conclusion.
\end{proof}

To show that any $N\in\ol{\ca M}(m,v)$ is locally $C$-Lipschitz contractible
(Corollary \ref{cor:N-Ccontractible}), we need 
the following. 

\begin{thm}[cf. \cite{PtPt:extremal},\cite{FMY}]\label{thm:conv-cover}
 For any $0<a<1$, there exists a positive constant $C=C(m,\kappa,\nu,\lambda,d,v,a)$ satisfying the following$:$ For any $X\in \ol{\ca M}(m,v)$ and its extension $Y$,
any $p \in Y$ and $0<\epsilon<C^{-1}$, there exists a $(C,C\epsilon)$-Lipschitz deformation retraction
$F:W\times [0,C\e]\to W$ that is $1$-Lipschitz on $F^{-1}(W\setminus B(p, \e/100))$
satisfying the following:
\begin{enumerate}
\item There exists a $1$-Lipschitz and 
$-(\mu/\e)$-concave function $h$ defined on a neighborhood of $p$ such that 
$W$ is a superlevel set of $h$,
where $\mu >0$ is a uniform constant$\,;$
\item Let $q$ be the unique maximum point of $h$. Then 
$B(p,\epsilon)\cup B(q, \epsilon)\subset W\,;$
\item  
For any $(x,t)\in W\times [0,C\e]$ 
with $F(x,t)\in W\setminus B(p,\e/100)$, we have 
\begin{enumerate}
\item $h(\gamma_{x,p}(t))-h(x)\ge (1-a)t\,;$ 
\item $\angle(\dot F(x,t),\uparrow^p_{F(x,t)})<
a$, \,\, $|\dot F(x,t)|> (1-a)\,;$
\item $|F(x,t), \pa W|\ge (1-a)t+|x,\pa W|$.
\end{enumerate} 
\item Suppose $p\in X$ and set $W_X:=W\cap X$. Then $H:=\pi\circ F|_{W_X\times [0,C\e]}$ defines a
$(C, C\epsilon)$-Lipschitz  deformation 
retraction of $W_X$ to a point 
that is $1$-Lipschitz on $H^{-1}(W_X\setminus B(p, \e/100))$

such that  
for any $H(y,t)\in W_X\setminus 
B^{X^{\rm int}}(p, \e/100)$, we have 
\begin{enumerate}
\item $\angle(\dot H(x,t),\uparrow^p_{H(x,t)})<
a$, \,\, $|\dot H(x,t)|> (1-a)\,;$
%
\item $|H(y,t),\pa W_X|_Y\ge (1-a) t+|y,\pa W_X|_Y$,
where $\pa W_X$ denotes the topological boundary of $W_X$ in $X$. 
\end{enumerate}
\end{enumerate}
\end{thm}

\begin{rem}\label{rem:ConvexHull} (1)\, Let 
$\ca A(m,\kappa,d,v)$ denote the set of all $m$-dimensional compact 
Alexandrov space $Y$ with curvature 
$\ge\kappa$,  diameter $\le d$ and $\ca H^m(Y)\ge v$.  
 The statements (1), (2) and (3) in the above theorem actually hold for any $Y\in\mathcal A(m,\kappa,d,v)$ and $p \in Y$.
\par
\n  
(2)\, Changing the time parameter in an ovbious way, 
we see that $W$ is $(C,\e)$-Lipschitz deformation 
retraction to a point$\,;$  
\par
\n
(3)\, The existence of such a convex domain  $W$ 
containing $B(p,\e)$ was first
established in \cite{PtPt:extremal}.
In \cite{FMY}, $(C,\e)$-Lipschitz contractibility of $W$ was proved.
\par\n
(4)\, For each fixed $x\in Y$, the defining concave 
function $h$ of a convex neighborhood, say 
$W_x$, in Theorem \ref{thm:convex-nbd}
is constructed via averaged distance functions from points on a tiny  metric sphere  around $x$. 
Note that the geometric size of $W_x$  is not 
comparable when $x$ moves over $Y$. This suggests that we need a global construction
to get convex neighborhoods whose geometric size are comparable.
To construct the  concave function $h$ defining $W$ in  Theorem \ref{thm:conv-cover}, 
following \cite{PtPt:extremal},
we first construct averaged  
Busemann functions in a blow-up limit of Y, then 
we lift it to $Y$. 
\end{rem}
As a direct consequence of Theorem \ref{thm:conv-cover}, we have the following corollary.

\begin{cor} \label{cor:N-Ccontractible}
There is  $C=C(m,\kappa,\nu,\lambda, d,v)>0$ such that
any $N \in\ol{\mathcal M}(m,v)$
is locally $C$-Lipschitz  contractible.
\end{cor}

\begin{proof}[Proof of Theorem \ref{thm:Lip-homotopy}]
This follows from Theorem  \ref{thm:Peter},
Lemma \ref{lem:C-doubling}  and Corollary \ref{cor:N-Ccontractible}. 
\end{proof}

\begin{proof}[Proof of Theorem \ref{thm:conv-cover}]
The basic approach in the proof below is the same as in 
 \cite{PtPt:extremal},\cite{FMY}.
However,
we need a more refined 
construction of the strictly concave function 
$h$ to obtain (3).
We proceed by contradiction.  
Suppose that  the conclusion does not hold.
Then there are $a>0$ and sequences $C_i\to\infty$, $0<\epsilon_i<C_i^{-1}$, $X_i\in\ol{\mathcal M}(m,v)$ and their extensions $Y_i$ and  $p_i\in Y_i$ such that 
no strictly concave function $h_i$ satisfies the conclusion for $a$, $C_i$, $p_i\in Y_i$ and $\e_i$.
Let us consider the rescaling $\tilde Y_i:=\epsilon_i^{-1}Y_i$, and  write metric balls and spheres in $\tilde Y_i$ as $\tilde B(\,,\,)$ and $\tilde S(\, ,\,)$, respectively.
Consider the convergence
\beq \label{eq:blow-upY}
(\tilde Y_i,p_i) \to (Y_\infty,p).
\eeq
For the statement (4), we assume $p_i\in X$. In this case,  under \eqref{eq:blow-upY},  $X_i$ converges to 
a closed convex subset $X_\infty$ of $Y_\infty$ (Lemma \ref{lem:inrad-collapse}).
The  uniform lower volume bound for $Y_i$ implies that $\dim Y_\infty=m$ and 
the ideal boundary $Y_\infty(\infty)$ of $Y_\infty$ has dimension $m-1$.

We first prove the  statements in $(Y_\infty,X_\infty)$ corresponding to $(1)\sim (4)$,  and then show that the construction can be lifted to $\tilde Y_i$.
Along the same line as \cite[4.3]{PtPt:extremal},
\cite{Kap}, 
we first construct a strictly concave function $h$ defined on an arbitrarily large neighborhood of $p$, which can be lifted to a function $\tilde h_i$ on $\tilde Y_i$ with the same concavity.
We show that the convex domain $W$ defined as a superlevel set of $h$ satisfies the desired properties
and the same holds for $\tilde h_i$.
\pmed
\n
{\bf Step 1. Construction in the blow-up limit $Y_\infty$.}\,\,
Let $0<\delta\ll\theta\ll 1$. 
Take a maximal $\theta$-discrete set $\{ l_\alpha, \}_{\alpha\in A}$ in the ideal boundary $Y_\infty(\infty)$ of $Y_\infty$.
For each $\alpha$, take a maximal $\delta$-discrete set 
$\{l_{\alpha\beta}\}_{\beta\in A_\alpha}$ in the ball  $B^{Y_\infty(\infty)}(l_\alpha,\theta)$.
Note that 
\begin{equation}\label{eq:a}
\# A_\alpha\ge c(\theta/\delta)^{m-1}
\end{equation}
for some uniform constant $c>0$ depending only on 
$m,\kappa,d,v$.
This  follows from the Bishop-Gromov inequality.

The parameters  $\theta, \delta$ will be determined later on.
For each $\alpha\in A$ and $\beta\in A_\alpha$,
 there exist rays $r_\alpha$ and $r_{\alpha\beta}$ emanating from $p$
in the directions to $l_{\alpha}$ and 
$l_{\alpha\beta}$ respectively.

In what follows, since 
\beq\label{eq:GH-idealbdy}
   d_{GH}(R^{-1}S(p,R),Y_\infty(\infty))<\tau_p(R^{-1}),
\eeq
we choose  large enough $R\gg1$, and use the symbol  $\mu_p(R)$ 
instead of $\tau_p(R^{-1})$.
For any fixed $x\in S(p,R)$,
choose an $\alpha\in A$ such that 
\beq \label{eq:|x,r|<mu}
|x,r_\alpha(R)|<(\mu_p(R)+\theta)R.
\eeq
We denote by $\frak a(x)$ the set of all such $a\in A$. Put
$q_{\alpha\beta}:=r_{\alpha\beta}(\hat R)$
for all $\beta\in A_\alpha$.
Then 
we have 
\beqq 
|x,r_{\alpha\beta}(R)|<(\mu_p(R)+2\theta)R,
\eeqq 
and hence
\beq \label{eq:|xqab|}  
    |x,q_{\alpha\beta}|\le  (\hat R-R)+(\mu_p(R)+2\theta)R
\eeq
by triangle inequality.
In what follows, we assume  $\hat R\gg R$.
Then we have
\beq \label{eq:wang(xpq)}
\begin{cases}
\begin{aligned}
\wangle xpq_{\alpha\beta}&\le\wangle xpr_{\alpha\beta}(R)\le \mu_p(R)+\tau(\theta),\\
\wangle pq_{\alpha\beta}x&\le \wangle  r_{\alpha\beta}(R)q_{\alpha\beta}x \le
(\mu_p(R)+\tau(\theta))R/(\hat R-R)\\
&\le  \mu_p(R)+\tau(\theta).
\end{aligned}
\end{cases}
\eeq
which implies 
\beq
\begin{aligned} \label{eq:wang(pxq)}
\wangle pxq_{\alpha\beta} & \ge \pi-\wangle pq_{\alpha\beta}x -\wangle xpq_{\alpha\beta} \\
 & \ge  \pi-(\mu_p(R)+\tau(\theta)).
\end{aligned}
\eeq
\psmall
For $0<\zeta<1$, let 
$\phi:\R_+\to\R_+$ be the $C^1$-function 
 with $\phi(0)=0$ determined by  
\beq \label{eq:defn=phi}
\begin{cases}
\text{$\phi'= 1$ \hspace{1.2cm} on $[0,\hat R-2R]$} \\
\text{$\phi' = 1-\zeta$ \hspace{0.45cm}  on $[\hat R+2R,\infty)$}\\
\text{$\phi'' = -\zeta/4R$ \hspace{0.02cm} on $[\hat R-2R,\hat R+2R]$}.
\end{cases}
\eeq
For any $\alpha\in A$, we define  the functions $h_\alpha$ on $B(p,2R)$ by
\[
h_\alpha:=\frac1{\# A_\alpha}\sum_{\beta\in A_\alpha}\phi(|q_{\alpha\beta},\cdot\,|).\qquad
\]
Clearly $h_\alpha$ is $1$-Lipschitz, and $h_\alpha(p)=\phi(\hat R)$ for all $\alpha\in A$.

The following sublemma essentially follows from 
{\cite[3.6]{Per}} (cf. \cite{Kap}).
For the basic idea of the proof, see also \cite[Claim 3.4]{FMY}.

\begin{slem}[{\cite[3.6]{Per}},cf. \cite{Kap}]\label{slem:conc}
$h_\alpha$ is $(-\mu)$-concave on $B(p,2R)$ for some $\mu>0$ depending only on $c$ in \eqref{eq:a}, $\zeta$
 and $R$, provided $\delta/\theta$ is small enough
and $\hat R$ is large enough.
\end{slem}

Let $h$ be the minimum of $h_\alpha$ over all $\alpha\in A$, which is also $1$-Lipschitz and $(-\mu)$-concave on $B(p,2R)$.

For any $x\in S(p,R)$, let $\alpha\in \frak{a}(x)$ be as in \eqref{eq:|x,r|<mu}.
Choose $\alpha_*\in A$ such that 
$h(x)=h_{\alpha_*}(x)$.
We denote by $\frak{a}_*(x)$ the set of all
such $\alpha_*\in A$.
Take $\beta\in A_{\alpha}$ and $\beta_*\in A_{\alpha_*}$
such that $|x,q_{\alpha\beta}|\ge |x,q_{\alpha_*\beta_*}|$.
It follows that 
\eqref{eq:wang(xpq)}  and \eqref{eq:wang(pxq)} hold for $q_{\alpha_*\gamma_*}$
for all $\gamma_*\in A_{\alpha_*}$  in place of $q_{\alpha\beta}$.  
Note also that \eqref{eq:wang(xpq)}  and \eqref{eq:wang(pxq)}
hold for all $x\in A(p,R/1000,2R)$.

Set
\[W:=h^{-1}([b,\infty))\cap \mathring{B}(p,2R),\quad  b:= \min_{x\in S(p,R)} h(x).\]

In Sublemma \ref{slem:inrad(W)}, we show 
that $W$ is a compact convex domain of $Y_\infty$.

In what follows, we set 
\[
\rho:=\mu_p(R)+\tau(\theta). 
\]

Let $\Phi$ denote the gradient flow for $h$. We have the following gradient estimates.

\begin{slem} \label{slem:grad-f-Phi}
For any $(y,t)\in W\times\R_+$ with $\Phi(y,t)\in W\setminus B(p,R/100)$, we have the following for small enough $\zeta$ and $\rho$$:$
\begin{enumerate}
\item $\angle(\dot\Phi(y,t),\uparrow_{\Phi(y,t)}^p)<\tau(\zeta,\rho)$,\quad $|\dot\Phi(y,t)|\ge 1-\zeta-\rho\,;$
\item $h(\gamma_{y,p}(t))\ge h(y) +(1-\zeta-\rho)t$.
\end{enumerate}
\end{slem}
\begin{proof} (1)\,
 First remark that for any $z\in A(p,R/1000,2R)$, 
from  \eqref{eq:wang(pxq)}, it holds that  
 \begin{align} \label{eq:wang(pzq)}
\wangle pzq_{\alpha\beta} & \ge \pi-\rho
\end{align}
for all $\alpha\in\frak{a}_*(z)$ and $\beta\in A_\alpha$.
This implies that 
\beq \label{eq:dh(gammayp)}
 dh(\dot\gamma_{z,p}(0))\ge (1-\zeta)(1-\rho^2)\ge 1-\zeta-\rho,
\eeq
Applying \eqref{eq:dh(gammayp)} to $z:=\Phi(y,t)$,
we obtain the conclusion (1) by \eqref{eq:gradient}.

(2)\,
Consider the function $f(s)=|p,\Phi(y,s)|$.
(1) implies that
\begin{align*} 
   f'(s)&=d_p'(\dot\Phi(y,s))=
-\langle\uparrow_{\Phi(y,t)}^p,\dot\Phi(y,s)\rangle    \\
&\le -dh(\uparrow_{\Phi(y,s)}^p)\le 
-(1-\zeta-\rho).
\end{align*}
It follows  
that 
\begin{align*}
R/100 &\le f(t)\le f(0)-t(1-\zeta-\rho), 
\end{align*}
and hence $t\le (f(0)-R/100)/(1-\zeta-\rho)$. 
This implies 
\[
 |p,\gamma_{y,p}(t)|=f(0)-t
  \ge f(0)\biggl( 1-\frac{1}{1-\zeta-\rho}\biggr)
\ge R/1000.
\]
Therefore we can apply \eqref{eq:dh(gammayp)}
for  $\gamma_{y,p}(s)$\,$(s\in [0,t])$ to get the conclusion (2).
\end{proof}

\begin{slem}\label{slem:inrad(W)}
We have the inclusions
\beq \label{eq:inclusionBW}
B(p, R)\subset W
\subset B(p,(1+\rho)R).
\eeq
In particular, $W$ is a compact convex domain of $Y_\infty$.
\end{slem}
\begin{proof} 
The first inclusion is obvious from the definition of 
$W$. By the triangle inequality and \eqref{eq:|xqab|}, for any $x\in S(p,R)$, we have  
$\hat R-R\le |x,q_{\alpha\beta}| \le \hat R-R + R\rho$
for any $\alpha\in\frak{a}_*(x)$ and $\beta\in A_\alpha$.
Using \eqref{eq:defn=phi}, we obtain
\begin{align*}
0\le\phi(|q_{\alpha\beta},x|)-\phi(\hat R-R)
&\le \int_{\hat R-R}^{\hat R-R + \rho R}
      \phi'(t)\,dt \le \rho R,
\end{align*}
which implies 
\beq\label{eq:h(x)-b}
    h(x)-b \le \rho R.
\eeq
For any $y\in\pa W$, let 
$z:=\Phi(y,t_1)\in S(p,R)$ for some $t_1\ge 0$.
By Sublemma \ref{slem:grad-f-Phi}, we have 
\[
 h(z)=h(\Phi(y,t_1))\ge (1-\zeta-\rho)t_1+h(y),
\]
and hence 
\beq\label{eq:t<h-h}
t_1\le (h(z)-h(y))(1-\zeta-\rho)^{-1}.
\eeq
Since $d_p'(\dot\Phi(y,s))\ge -1$, we have
$R-|p,y|\ge -t_1$.
It follows from \eqref{eq:h(x)-b} and \eqref{eq:t<h-h} that 
\begin{align*}
  R-|p,y|&\ge -(h(z)-h(y))(1-\zeta-\rho)^{-1} 
 \\
&\ge -\rho R(1-\zeta-\rho)^{-1} \ge -2\rho R, 
\end{align*}
yielding $|p,y|\le (1+\rho) R$.
\end{proof}

\begin{slem}\label{slem:Hdist(WS)}
For $y\in W\setminus B(p,R/100)$, 
set $T:=d_p(y)$ and $W_y:=h^{-1}([h(y),\infty))$. Then we have 
\beq \label{eq:angle(xi,ST)}
    d_H(\Sigma_y(\pa W_y), \Sigma_y(S(p,T))) <\tau(\rho),
\eeq
where $d_H$ denotes the Hausdorff  distance in $\Sigma_y(Y_\infty)$.
\end{slem}
\begin{proof} 
First we show that $\Sigma_y(\pa W_y)$ is contained in a $\tau(\rho)$-neighborhood of 
$\Sigma_y(S(p,T))$.
If this does not hold, we have $\xi\in\Sigma_y(\pa W_y)$ such that 
$\angle(\xi, \Sigma_y(S(p,T)))\ge \omega>0$
for some uniform constant $\omega$ independent of  $\rho$.
We assume $d_p'(\xi)>0$, since the other case is similarly discussed.
Take a sequence $z_i\in\pa W_y$ converging to $y$ such that 
$\angle(\uparrow_y^{z_i},\xi)\to 0$.
\eqref{eq:wang(pzq)} then shows 
\[
     \wangle z_iyq_{\alpha\beta}<  \pi/2- \omega+\rho <\pi/2- \omega/2,
\] 
for all
$\alpha\in\frak{a}_*(y)$, $\beta\in A_\alpha$
and large enough $i$ if $\rho$ is small enough.
Take any $\alpha_i\in\frak{a}_*(z_i)$.
Passing to a subsequence, we may assume
that $\alpha_i$ does not depend on $i$
and therefore all $\alpha_i$ coincide with some $\alpha\in\frak{a}_*(y)$.
This yields the contradiction
$h(z_i) <h(y)$. 

Next we show the reverse inclusion.
Take any $\xi\in\Sigma_y(S(p,T))$.
It suffices to show that if  $\rho$ is  
small enough, then for any fixed $a>0$,
there is $u_0\in\Sigma_y(\pa W_y)$ in the $a$-neighborhood of $\xi$.
Set $v_+:=\uparrow_y^p$, $v_-:=\uparrow_y^{q_{\alpha}}$ with $\alpha\in\frak{a}_*(y)$.
Consider geodesic segments $\gamma_+$ and 
$\gamma_-$ in $T_y(Y_\infty)$ joining $\xi$ to $v_+$ and $v_-$ 
respectively, and let $\gamma$ be the union of 
$\gamma_+$ and $\gamma_-$.
Let $\tilde \gamma=\tilde\gamma_1\cup\tilde\gamma_2$ be the broken geodesic
 in $\Sigma_y(Y_\infty)$ defined by the normalization of $\gamma$.
Let $U$ denote the open cone in $T_y(Y_\infty)$ making 
angle $<a$ with $\xi$.
Take any $u\in\tilde\gamma\setminus U$.
If $u\in \tilde\gamma_+$,
from \eqref{eq:wang(pzq)}, we have
\[
    \angle(u,\uparrow_y^{q_{\alpha\beta}}) 
> \pi/2+a/2-\rho>\pi/2+a/3,
\]
for small enough $\rho$. This yields  $dh(u)>0$.
If $u\in \tilde\gamma_-$,
we have  $dh(u)<0$ in a similar way.
Thus we can find $u_0\in\tilde\gamma\cap U$
such that $dh(u)=0$, that is, $u\in\Sigma_y(\pa W_y)$.
\end{proof}
\psmall

\begin{slem} \label{slem:grad-f-Phi2}
For any $(y,t)\in W\times\R_+$ with $\Phi(y,t)\in W\setminus B(p,R/100)$, we have 
\[
  |\Phi(y,t), \pa W|\ge (1-\tau(\rho))t+|y,\pa W|.
\]
\end{slem}

\begin{proof} 
Let $(y,t)$ be as in the sublemma.
For any fixed $s\in [0,t]$, set $z:=\Phi(y,s)$, and 
choose a nearest point $u$ of $\pa W$ from 
$w$ such that if $\alpha_{\min}(s)$ denotes the
minimum of angles between $\dot\Phi(y,s)$
and all the minimal geodesics joining $z$ to $\pa W$,
then $\alpha_{\min}(s)=\angle(\dot\Phi(y,s),\dot\gamma_{z,u}(0))$.
Then the first variation formula shows
\[
   \frac{d}{ds}|\Phi(y,s), \pa W|=-|\dot\Phi(y,s)|\cos\alpha_{\min}(s).
\]
To prove the sublemma, it suffices to show 
\beq \label{eq:wangle(pwu)=pi}
  \wangle pzu>\pi-\tau(\rho).
\eeq

Let $\xi_*:=\uparrow_u^z$, $\xi:=\uparrow_u^p$.
Since $\angle(\xi_*,\Sigma_u(\pa W))\ge \pi/2$,
considering the double $D(\Sigma_u(W))$, we see 
that $\Sigma_u(W)$ is the half suspension 
$\xi_**\Sigma_u(\pa W)$.
It follows from Sublemma \ref{slem:Hdist(WS)}
that $\angle(\xi_*,\xi)<\tau(\rho)$.
Thus we obtain 
\beq\label{eq:angle(puw)}
\angle puz<\tau(\rho).
\eeq
To prove \eqref{eq:wangle(pwu)=pi}, it suffices to show  $\wangle pzu>\pi-a$  for any fixed $a>0$ and for small enough $\rho$.
Suppose this does not hold. Then there are
$a>0$ and sequences
$\hat R_i > R_i\to\infty$ with $R_i/\hat R_i\to 0$, $\theta_i\to 0$, $y_i\in W_i$ and
$z_i=\Phi_i(y_i,s_i)\in W_i\setminus B(p,R_i/100)$ such that 
\[
\wangle pz_iu_i \le \pi-a, 
\] 
where $W_i$ is a convex set, $\Phi_i$
and  $u_i\in \pa W_i$ are  constructed and taken as before for $R=R_i$, $\hat R=\hat R_i$, $\theta=\theta_i$ and $u=u_i$.
If $|z_i,u_i|/R_i \to 0$, then 
$\wangle u_ipw_i\to 0$. Together with \eqref{eq:angle(puw)}, this implies 
\[
\wangle pz_iu_i\ge \pi-\wangle u_ipz_i-\wangle pu_iz_i \ge \pi - o_i,
\]
which is a contradiction.
Next suppose $|z_i,u_i|/R_i \ge c>0$ for a constant 
independent of $i$, and consider the 
convergence
$(Y_\infty/R_i,p)\to K(Y_\infty(\infty)),p_\infty)$. 
It follows from \eqref{eq:angle(puw)} that 
$z_\infty\in\gamma_{p_\infty u_\infty}$ and 
therefore $\wangle pz_iu_i\to\pi$, a contradiction.
This proves \eqref{eq:wangle(pwu)=pi},
 completes the proof of of  Sublemma \ref{slem:grad-f-Phi2}.  
\end{proof}

\begin{slem}\label{slem:Phi-gamma}
For any $(y,t)\in W\times\R_+$ with $\Phi(y,t)\in W\setminus B(p,R/100)$, we have 
\[
  \angle(\dot\Phi(y,0),\dot\gamma_{y,\Phi(y,t)}(0))<\tau(\zeta,\rho).
\]
\end{slem}
\begin{proof}
Set $z:=\Phi(y,t)$ and $\beta:=\angle(\dot\Phi(y,0),\dot\gamma_{y,z}(0))$. Observe that 
\[
    dh(\dot\gamma_{y,z}(0))\le\langle\dot\Phi(y,0),\dot\gamma_{y,z}(0)\rangle\le\cos\beta.
\]
It follows from the concavity of $h$ that  
\[
h(z)\le h(y)+|y,z|dh(\dot\gamma_{y,z}(0))
\le h(y)+|y,z|\cos\beta,
\]
and hence
$h(z)-h(y)\le |y,z|\cos\beta\le t$.
On the other hand, from Sublemma \ref{slem:grad-f-Phi}, we get
\[
     h(z)-h(y)=\int_0^t  \frac{d}{ds} h(\Phi(y,s) ds
     =\int_0^t |\dot\Phi(y,s)|^2 ds\ge  (1-\tau(\zeta,\rho))t.
\]
Combining these two inequalities, we have
$\cos\beta\ge 1-\tau(\zeta,\rho)$, and 
the conclusion.
\end{proof}

\psmall

By Sublemma \ref{slem:grad-f-Phi},
the maximum point $q$ of $h$ 
must be contained in $B(p,R/1000)$
if $\zeta$ and $\rho$ are small enough.
It follows that 
\beq \label{eq:B(q)subsetW}
    B(q,0.99R)\subset B(p,0.991R)\subset \mathring{W}.
\eeq

\pmed
\n
{\bf Step 2. Projection to $X_\infty$.}\,\,
\pmed  
 In the blow-up limit \eqref{eq:blow-upY},
we assume $p_i\in X_i\subset Y_i$.
Let $\eta_i:C_i\to Y_i$ 
and $(\eta_i)_0:(C_i)_0\to (X_i)_0$ be as in 
Section \ref{sec:non-inradius}.

We may assume that under  \eqref{eq:blow-upY}
\begin{itemize}
\item $X_i$ converges to a closed convex 
subset $X_\infty$ of $Y_\infty\,;$
\item $\eta_i$ and $(\eta_i)_0$ converge to $1$-Lipschitz maps
$\eta_\infty:C_\infty\to Y_\infty$ and 
$(\eta_\infty)_0:(C_\infty)_0\to (X_\infty)_0$
respectively$\,;$
\item $\pi_i:Y_i\to X_i$ and 
$\tilde\pi_i: C_i\to (C_i)_0$ converge to  $1$-Lipschitz
maps $\pi_\infty:Y_\infty\to X_\infty$ and 
 $\tilde\pi_\infty:C_\infty\to (C_\infty)_0$
which are the projections along the perpendiculars to $X_\infty$ and $(C_\infty)_0$ respectively.
\end{itemize}
Note that 
$X_\infty(\infty)\subset Y_\infty(\infty)$
and  
$\pi_\infty=\eta_\infty\circ\tilde\pi_\infty\circ\eta_\infty^{-1}$ on $Y_\infty\setminus X_\infty$.

\pbig
Since 
\beq\label{eq:GH-idealbdy}
   d_{GH}(R^{-1}S^{X_\infty}(p,R),X_\infty(\infty))<\mu_p(R),
\eeq
for any $x\in A^{X_\infty}(p, R/100,R)$,
we can choose a geodesic ray $l:[0,\infty)\to X_\infty$
such that 
$|x,l(R)|<\mu_p(R)R$, and hence 
\beq\label{eq:wangle(pxl)}
     \wangle px l(\hat R) >\pi - \rho.
\eeq
\pbig

  For the point $p\in X_\infty$, let $W$ be as 
in Step 1.
Set 
\[
      W_{X_\infty}:=W\cap X_\infty.
\] 
In a way similar to Lemma \ref{lem:G1=pi},
we have the following.

\begin{lem} \label{lem:G2=pi}
For any $(y,t)\in W_{X_\infty}\times \R_+$, we have 
$$
\pi_\infty\circ \Phi(y,t)\in W_{X_\infty}.
$$
\end{lem}
\psmall
Consider the $1$-Lipschitz map  
$$
H:=\pi_\infty\circ \Phi:W_{X_\infty}\times \R_+\to
W_{X_\infty}.
$$   
 
In what follows we describe some 
properties of the homotopy $H$
similar to those of $\Phi$  in Sublemmas \ref{slem:grad-f-Phi} and \ref{slem:grad-f-Phi2},  to prove Theorem \ref{thm:conv-cover}(4).
 Roughly speaking, we show that $H(y,t)$ is almost parallel to $\Phi(y,t)$.

\begin{lem} \label{lem:almost-grad-phi}
For any $(y,t)\in W_{X_\infty}\times\R_+$ with 
$H(y,t)\in W_{X_\infty}\setminus B^{X_\infty}(p, R/100)$, we have 
\begin{enumerate}
 \item $\angle(\dot H(y,t),\dot\gamma_{H(y,t),p}(0))<\tau(\zeta,\rho), \quad 
 |\dot H(y,t)| > 1-\tau(\zeta,\rho)$,  
 \item $|H(y,t),\pa W_{X_\infty}|_{Y_\infty}\ge (1-\tau(\zeta,\rho))t+|y,\pa W_{X_\infty}|_{Y_\infty}$. 
\end{enumerate}
\end{lem}
\begin{proof} 
(1)\,
The first inequality of  (1)  follows in a way 
similar to Lemma \ref{lem:gradH-6},
and and the second inequality follows in a way 
similar to Sublemma \ref{slem:|H|=1}.
Therefore we omit the details.

 (2)\,  The basic idea of the proof of (2) is the same as that of Sublemma \ref{slem:grad-f-Phi2}.
In  a way similar to Sublemma \ref{slem:Hdist(WS)},
we obtain the following, and hence omit the proof.

\begin{slem}\label{slem:Hdist(WSX)}
For $y\in W_{X_\infty}\setminus B^{X_\infty}(p,R/100)$, 
set $T:=d_p(y)$ and $W_{X_\infty y}:=h^{-1}([h(y),\infty))\cap X_\infty$. Then we have 
\beq \label{eq:angle(xi,ST)}
    d_H(\Sigma_y(\pa W_{X_\infty y}), \Sigma_y(S^{X_\infty}(p,T))) <\tau(\rho),
\eeq
where $d_H$ denotes the Hausdorff  distance in $\Sigma_y(Y_\infty)$.
\end{slem}

\psmall

Let $(y,t)$ be as in the lemma.
For any fixed $s\in [0,t]$, set $z:=H(y,s)$, and 
choose a nearest point $u$ of $\pa W_{X_\infty}$ from 
$w$ such that if $\alpha_{\min}(s)$ denotes the
minimum of angles between $\dot H(y,s)$
and all the minimal geodesics joining $z$ to $\pa W_{X_\infty}$,
then $\alpha_{\min}(s)=\angle(\dot H(y,s),\dot\gamma_{z,u}(0))$.
Then the first variation formula shows
\[
   \frac{d}{ds}|H(y,s), \pa W_{X_\infty}|=-|\dot H(y,s)|\cos\alpha_{\min}(s).
\]

Since $W_{X_\infty}$ is convex in $Y_\infty$,
we can show that $\wangle pzu>\pi-\tau(\rho)$
in the same way as \eqref{eq:wangle(pwu)=pi}.
This completes the proof of (2),
and hence  the proof of Lemma  \ref{lem:almost-grad-phi}.
\end{proof}

\pmed\n
{\bf Step 3. Lifting.}\,\,
Finally we lift the above situation to $\tilde Y_i$.
Let $\tilde h_i$ and $\tilde W_i$ denote the lifts of $h$ and $W$ to $\tilde Y_i$ respectively.
Namely, letting $\tilde q_\alpha^i$, $\tilde q_{\alpha\beta}^i$ be the images of 
$q_\alpha$, $q_{\alpha\beta}$
under an $o_i$-approximation
$\varphi_i:(Y_\infty,p)\to (\tilde Y_i,p_i)$,
we define
\begin{align*}
\tilde  h_{i,\alpha}&:=\frac1{\# A_\alpha}\sum_{\beta\in A_\alpha}\phi(|\tilde q_{\alpha\beta}^i, \cdot\,|),\quad
\tilde h_i:=\min_{\alpha\in A} \tilde h_{i,\alpha},\\
\tilde W_i&:=\{ \tilde h_i\ge a_i\},\quad
 b_i:=\min_{x\in \tilde S(p_i,R)}\tilde h_i(x).
\end{align*}
Then $\tilde h_i$ is $1$-Lipschitz and $(-\lambda)$-concave on $\tilde B(p_i,2R)$.
This follows from the same argument using volume comparison as in Sublemma \ref{slem:conc}.
Therefore $\tilde W_i$ is a convex domain containing $\tilde B(p_i,R)$. 
Let $\tilde\Phi_i$ denote the gradient flow of $\tilde h_i$.
Note that the above argument based on comparison
angles is stable under the pointed Gromov-Hausdorff convergence. Thus from  Sublemmas \ref{slem:grad-f-Phi} and \ref{slem:grad-f-Phi2}, we get 
for any $(y,t)\in \tilde W_i\times \R_+$ with 
$\tilde\Phi_i(y, t)\in \tilde W_i\setminus \tilde B(p_i,R/100)$ 
\beq\label{eq:tildeh(gamma)>t}
\begin{cases}
\begin{aligned}
    & \tilde h_i(\gamma^{\tilde Y_i}_{y,p_i}(t))-\tilde h_i(y)\ge (1-\tau(\zeta,\rho))t, \\
& \angle(\dot{\tilde\Phi}_i(y,t),\uparrow_{\tilde\Phi_i(y,t)}^{p_i})<\tau(\zeta,\rho),\,\,\,\,
|\dot{\tilde\Phi}_i(y,t)|>1-\tau(\zeta,\rho), \\
&|\tilde\Phi_i(y,t), \pa \tilde W_i|\ge (1-\tau(\zeta,\rho))t+|y,\pa \tilde W_i|. 
\end{aligned}
\end{cases}
\eeq
From Lemma \ref{lem:almost-grad-phi}, we get 
for any $(y,t)\in (\tilde W_i)_{\tilde X_i}\times\R_+$ with 
$\tilde H_i(y,t)\in (\tilde W_i)_{\tilde X_i}\setminus B^{\tilde X_i}(p, R/100)$,

\beq\label{eq:tildeH(gamma)>t}
\begin{cases}
\begin{aligned}
    &  \angle(\dot{\tilde H}_i(y,t),\dot\gamma_{\tilde H_i(y,t),p_i}(0))<\tau(\zeta,\rho), \,\,
 |\dot{\tilde H}_i(y,t)|> 1-\tau(\zeta,\rho),\\
&  |\tilde H_i(y,t),\pa (\tilde W_i)_{\tilde X_i}|_{\tilde Y_i}\ge (1-\tau(\zeta,\rho))t+|y,\pa (W_i)_{\hat X_i}|_{\tilde Y_i},
\end{aligned}
\end{cases}
\eeq
 
To confirm the situation in the original metric $Y$, let 
\[
\text{
    $ h_i:=\e_i \tilde h_i$, \quad  
$W_i:=\{ h_i\ge\e_i b_i\}=\tilde W_i$}.
\] 
Note that $h_i$ is $1$-Lipschitz and 
$-(\lambda/\e_i)$-concave on $W_i$.
Let $\Phi_i$ denote the gradient flow of $h_i$:
\[
         \Phi(y_i,t)=\tilde\Phi(y_i,t/\e_i).
\] 
Recall that $B(p_i, \e_i R)\subset W_i$.
From \eqref{eq:tildeh(gamma)>t},
we have 
\beq \label{eq:Phi-final}
\begin{cases}
&h_i(\gamma^{Y_i}_{y,p_i}(t)) -h_i(y) \ge (1-\tau(\zeta,\rho))t, 
\quad\\
& \angle(\dot{\Phi}_i(y,t),\dot\gamma_{\Phi_i(y,t),p_i}(0))<\tau(\zeta,\rho), \,\,
 |\dot{\Phi}_i(y,t)| > 1-\tau(\zeta,\rho),\\
& |\Phi_i(y,t), \pa W_i|\ge (1-\tau(\zeta,\rho))t+|y, \pa W_i|.
\end{cases}
\eeq
for  any $(y,t)\in W_i\times \R_+$ with 
$\Phi_i(y,t)\in W_i\setminus B(p_i,\e_i R/100)$. 

Consider 
\[
    H_i:=\pi\circ \Phi_i.
\]
From \eqref{eq:tildeH(gamma)>t}
we have
\beq \label{eq:H-final}
\begin{cases}
&\angle(\dot{H}_i(y,t),\dot\gamma_{H_i(y,t),p_i}(0))<\tau(\zeta,\rho), \,\,
 |\dot{H}_i(y,t)|>1-\tau(\zeta,\rho),\\
& |H_i(y,t), \pa (W_i)_{X_i}|\ge (1-\tau(\zeta,\rho))t+|y, \pa (W_i)_{X_i}|,
\end{cases}
\eeq
for  any $(y,t)\in (W_i)_{X_i}\times \R_+$ with 
$H_i(y,t)\in (W_i)_{X_i}\setminus B(p_i,\e_iR/100)$.

\pmed\n
{\bf Step 4. Lipschitz contractibility.}\,\,
\pmed

Finally we have to verify that the convex set $W_i$ 
constructed in Step 3 
is $(C, C\epsilon_i)$-Lipschitz strongly 
contractible to 
a point of $W_i$.
However this is established in 
Steps 2 and 3 in the proof of  \cite[Theorem 3.1]{FMY}.
Therefore we only give an outline of the argument.

First we need to return to Step 1.
Recall that $q\in W$ is the unique maximum point  of the $-\mu$ concave function $h$. 
Fix a small enough $r>0$.
By using a maximal $\theta r$-discrete set $\{ x_{\alpha}\}_{\alpha\in \hat A}$ in $S(q,r)$ and 
 a maximal $\delta r$-discrete 
 system $\{ x_{\alpha\beta}\}_{\beta\in \hat A_\alpha}$
in $B(x_{\alpha},\theta r)\cap S(q,r)$,
we construct a $(-1)$-concave function $g$ on 
$B(q,\rho)$ as in Theorem \ref{thm:convex-nbd}.
By the gradient flow $\Phi$ of $h$, we can deform $W$ to 
$B(q,\rho)$.
By Lemma \ref{lem:grad-df}, we have
\[
    dh(\uparrow_x^q)\ge \mu |p,x|/2
\]
for any $x\in W\setminus B(q,\rho)$.
Therefore $\Phi$ gives a $(1, C_1)$-Lipschitz
strong deformation retraction of $W$ to
$B(q, \rho)$, where $C_1:=(\mu \rho/2)^{-1}$. From the construction, 
for any $x\in B(q,\rho)\setminus \{ q\}$, there is $y\in
S(q,2\rho)$ such that 
$$
\wangle qxy>\pi/2+\nu
$$
for a uniform constant $\nu>0$.
If we consider the gradient frow $\Psi$ of $d_{S(q,2\rho)}$, it provides a $(C_2,C_2\rho)$-Lipschitz strong
deformation retraction of $B(q,\rho)$ to $q$.
\psmall
We return  to Step 2.
Let 
$\tilde g_i$ be a $(-1)$-concave function defiend as a
lift of $g$ to $\tilde Y_i$.
Let $\hat q_i$ be a maximum point of $g_i$ converging to $q$.
We have to show that the gradient flow $\tilde\Phi_i$
for $\tilde h_i$ meets $\tilde S(\hat q_i,\rho)$ for large enough $i$
if $\theta$ is small enough compared with $\rho$.
Suppose this does not hold. Then we have a 
point  $x_i\in \tilde S(\hat q_i,\rho)$
such that $d\tilde h_i(\uparrow_{x_i}^{\hat q_i})\to 0$.
We may assume $x_i$ converges to a point $x\in
S(q,\rho)$. We may also assume $\alpha_i\in \frak{a}_*(x_i)$ (this is the notation for $\tilde g_i$)  does not depend 
on $i$. Then $\alpha=\alpha_i\in \frak{a}_*(x)$.
If $\theta$ is small enough (e.g., $\theta \hat R\le \rho/10$), then we have 
\[
     \lim_{i\to\infty}\wangle \hat q_i x_i q_{\alpha\beta}^i
         =\wangle  \hat q x q_{\alpha\beta}\ge \pi/2+\nu,
\]
which implies $d\tilde h_i(\uparrow_{x_i}^{\hat q_i})> c(\nu)>0$, a contradiction. 
Thus $\tilde\Phi_i$ provides a $(1,C_3)$-Lipschitz
strong deformation retraction of $\tilde W_i$ to 
$\tilde B(\hat q_i,\rho)$.

A key is to show that for any $x_i\in B(\hat q_i,\rho)\setminus \{ \hat q_i \}$, there is $y_i\in
S(\hat q_i,2\rho)$ such that 
\beq \label{eq:wangle(qxy>pi/2}
\wangle \hat q_ix_iy_i>\pi/2+\tilde \nu
\eeq
for a constant $\tilde\nu>0$ independent of $i$.
This was established in \cite[Theorem 3.1]{FMY}.
\eqref{eq:wangle(qxy>pi/2}  implies that the gradient flow $\tilde\Psi_i$ of $d_{S(\hat q_i, 2\rho)}$ provides a $(C_4,C_4\rho)$-Lipschitz
strongly deformation retraction of $\tilde B(\hat q_i, \rho)$
to $\hat q_i$.

Combining the two Lipschitz deformation retractions
$\tilde\Phi_i$ and $\tilde\Psi_i$, we define 
a $(C_4,C_3+C_4)$-Lipschitz deformation retraction $\tilde F_i$ 
of $\tilde W_i$ to  $\hat q_i$
as 
\beq
 \tilde F_i(x,t)=
\begin{cases}
   \tilde\Phi_i(x,t)   &  0\le t\le C_3,  \\
   \Psi_i(\Phi_i(x,\e_i), t-C_3)  & C_3\le t\le C_3+C_4.
\end{cases}
\eeq

In the original metric of $Y_i$, letting $C:=C_3+C_4$, we finally get
a $(C, C\e_i)$-Lipschitz deformation retraction $F_i$ 
of $W_i$ to  $\hat q_i$
satisfying the conclusions (1),(2) and (3)
for small enough $\zeta$ and $\rho$
by \eqref{eq:Phi-final}.
Note that it is $1$-Lipschitz on 
$F_i^{-1}(W_i\setminus B(p_i,\e_i/100))$.

Let $H_i:=\pi\circ  F_i$. This defines the 
$(C, C\e)$-Lipschitz strong deformation retraction
of $(W_i)_{X_i}$ to $\pi(\hat q_i)$ satisfying the conclusion (4)
by \eqref{eq:H-final}.
Note that it is $1$-Lipschitz on 
$H_i^{-1}((W_i)_{X_i}\setminus B(p_i,\e_i/100))$.
Thus we have a contradiction.
This completes the proof of Theorem \ref{thm:conv-cover}.   
 \end{proof}

\begin{proof}[Proof of Corollary \ref{cor:finiteness}]
Let $C>1$ be the constant given in Theorem \ref{thm:Lip-homotopy}. From the compactness of
$\ol{\ca M}(m,v)$, for any $0<\e<C^{-1}$,
there are finitely many $X_1,\ldots, X_k$
and $GH$-neighborhoods $\ca U_i$ of $X_i$ in 
$\ol{\ca M}(m,v)$ for each $1\le i\le k$
such that  
\begin{itemize}
\item $\{ \ca U_i\}_{i=1}^k$ covers $\ol{\ca M}(m,v)\,;$
\item arbitrary $X$ and $X'$ in $\ca U_i$ have
the Gromov-Hausdorff distance 
$d_{GH}(X,X')<C^{-1}\e$.
\end{itemize}
Since $d_{\text{$C$-$LH$}}(X,X')<\e$
for $X,X'\in \ca U_i$ by Theorem \ref{thm:Lip-homotopy}, $X$ and $X'$ have the same $(C,\e)$-Lipschitz homotopy types.
This completes the proof.
\end{proof}

\pmed


\setcounter{equation}{0}
\section{Local Lipschitz homotopy stability}  \label{sec:LLHS}

In this section, we prove Theorems 
\ref{thm:metric-sphere} and \ref{thm:punctured-ball}.

\pmed\n
{\bf Lipschitz homotopy stability of small metric spheres}.
\pmed

Let $T^r_p(Y):=\{ v\in T_p(Y)\,|\,|v|\le r\}$.

To prove Theorem \ref{thm:metric-sphere},
we first verify the following.

\begin{prop} \label{prop:int-metric-sphereGH} 
For any $X\in\ol{\ca M}$ and $p\in X$, 
$S^{X^{\rm ext}}(p,r)/r$ converges to  $\Sigma_p(X)$
as $r\to 0$ with respect to the Gromov-Hausdorff 
distance.
\end{prop}
\begin{proof}
Since the proposition is well known in case $X$ is an Alexandrov space, 
we assume $p\in X_0$. 
First suppose $p\in {\rm int} Y$, and 
consider the convergence 
\beq \label{conv(BYtoT(Y)}
B^Y(p,2r)/r  \xrightarrow{GH} T^2_p(Y) \quad \text{as $r\to 0$.}
\eeq
Let $\varphi: B^Y(p,2r)/r \to T^2_p(Y)$ and 
$\psi:T^2_p(Y)\to B^Y(p,2r)$ be $\tau_p(r)$-approximations such that
\beq\label{eq:varphi-psi}
\begin{aligned}
&|\psi\circ\varphi(y),y|<\tau_p(r), \quad  
|\varphi\circ\psi(\tilde y),\tilde y|<\tau_p(r),  \\
& 
|\varphi(x), T^2_p(X)|<\tau_p(r), \quad  
|\psi(\tilde x),B^{X^{\rm ext}}(p,2r)|<\tau_p(r), 
\end{aligned}
\eeq
for all $y\in B^Y(p,2r)$ and $\tilde y\in T^2_p(Y)$
and for all
$x\in B^{X^{\rm ext}}(p,2r)$ and 
$\tilde x\in T^2_p(X)$.

In what follows, we use the identifications
\[
   \Sigma_p(X)=\Sigma_p(X)\times\{ 1\}\subset T_p(X), \quad
\Sigma_p(Y)=\Sigma_p(Y)\times\{ 1\}\subset T_p(Y).
\]
Note that $\Sigma_p(X)$ is convex in $\Sigma_p(Y)$
 (\cite[Lemma 5.12]{YZ:part1})
and 
$\Sigma_p(X)^{\rm reg}\subset  \Sigma_p(Y)^{\rm reg}$ (\cite{BGP}).
For any $\delta$, take a finite set $\tilde V:=\{ \tilde v_i\}_{i=1}^I\subset \Sigma_p(X)$
such that 
\begin{itemize}
\item it is $\delta/2$-discrete and $\delta$-dense in $\Sigma_p(X)\,;$
\item $\tilde V\subset\Sigma_p(X)^{\rm reg}
       \subset\Sigma_p(Y)^{\rm reg}\subset T_p(Y)^{\rm reg}$.
\end{itemize}
 Let $\tilde\Gamma$ be the union of all 
minimal geodesics
$\tilde \gamma_{i,j}$ joning $\tilde v_i$
and $\tilde v_j$ in $\Sigma_p(X)$.
Note that $\tilde\Gamma$
is contained in $\Sigma_p(X)^{\rm reg}$
 (see \cite{Pet:Parallel}). 
For any small enough  $0<\e_0\ll 1/\dim Y$, choose $\sigma_0>0$ such that 
the $\sigma_0$-neighborhood 
\[
   \tilde A^{T_p(Y)}:=B^{T_p(Y)}(\tilde\Gamma,\sigma_0)
\]
of $\tilde\Gamma$ in $T_p(Y)$ is contained in 
$T_p(Y)_{\e_0}^{\rm reg}$, the $\e_0$-regular set of 
$T_p(Y)$ (see Section \ref{sec:prelim}).
Under the convergence \eqref{conv(BYtoT(Y)},  for small enough $r$, we have a closed domain $A_r^Y\subset B^Y(p, 2r)/r$ contained in $Y_{2\e_0}^{\rm reg}$,  and a 
$\tau_p(r,\e_0)$-almost isometry
$f^Y_r: A_r^Y \to \tilde A^{T_p(Y)}$  such that
\beq \label{eq:|f-varphi|}
|f^Y_r(x), \varphi(x)|<\tau_p(r)
\eeq
for all
$x\in A_r^Y$
(see \cite{BGP}, \cite{Ya:conv}).
Let $g_r^Y:=(f_r^Y)^{-1}$. 
For any $y=f_r^Y(x)\in \tilde A^{T_p(Y)}$, we then have
\[
|g_r^Y(y),\psi(y)|\le |x,\psi\circ\varphi(x)|+|\psi\circ\varphi(x),\psi\circ f_r^Y(x)|<\tau_p(r).
\]  
Set 
$$
A_r:=A_r^Y\cap X, \quad  
\tilde A:= \tilde A^{T_p(Y)}\cap T_p(X).
$$
Let $\pi_\infty:T_p(Y)\to T_p(X)$ be the projection along the perpendiculars to $T_p(X)$.
Although one can show $\tilde A=\pi_\infty(\tilde A^{T_p(Y)})$, we do not give the proof since we do not use it. 
%
Consider the maps
\begin{align*}
&f_r:=\pi_\infty\circ f^Y_r: A_r\to  \pi_\infty(\tilde A^{T_p(Y)}), \\  
&g_r:=\pi\circ g_r^Y:\tilde A\to \pi(A_r^Y),
\end{align*}
where 
we may assume 
$\pi_\infty(\tilde A^{T_p(Y)})\subset T_p(Y)_{\e_0}^{\rm reg}$ and 
$\pi(A_r^Y)\subset Y_{2\e_0}^{\rm reg}$.

For each $1\le i\le I$, let $v_i^r$ be a nearest point  of $S^{X^{\rm ext}}(p,r)$ from $g_r(\tilde v_i)$.
\eqref{eq:varphi-psi} implies  
$|v_i^r, g_r(\tilde v_i)|<\tau_p(r)$.


\begin{slem} \label{slem:delta-dense}
$V:=\{ v_i^r\}$ is $7\delta$-dense in $S^{X^{\rm ext}}(p,r)/r$. 
\end{slem}
\begin{proof}
 For any $v\in S^{X^{\rm ext}}(p,r)/r$, choose
$i$ such that $|\varphi(v),\tilde v_i|<\delta+\tau_p(r)$.
Then we have
\begin{align*}
|v,v_i^r| &\le|\varphi(v),\varphi(v_i^r)|+\tau_p(r)\\
&\le |\varphi(v),\tilde v_i| +
|\tilde v_i, \varphi\circ\psi(\tilde v_i)|  \\
&+
|\varphi\circ\psi(\tilde v_i),\varphi\circ g_r(\tilde v_i)|+
|\varphi\circ g_r(\tilde v_i),\varphi(v_i^r)|+
\tau_p(r)\\
&\le \delta+\tau_p(r).
\end{align*}
Take points $u, u_i\in S^{X^{\rm ext}}(p,(1+2\delta)r)/r$ such that
\[
    \wangle pvu>\pi-\tau_p(r),\quad
\wangle pv_i^ru_i>\pi-\tau_p(r).
\]
It is easy from the curvature condition that
$|u,u_i|<2\delta$ for small enough $r$.
Let $\alpha$ be a broken $Y$-geodesic 
from $v$ to $v_i^r$ in $A^{Y/r}(p,1, 1+2\delta)$  consisting of minimal geodesics 
$vu$, $uu_i$ and $u_iv_i^r$.
Note that $\pi\circ\alpha$ is a curve in 
$A^{X^{\rm ext}/r}(p,1, 1+2\delta)$ of length
$L(\pi\circ\alpha)\le (1+\tau_p(r))L(\alpha) <7\delta$.
Let $H:A^{X^{\rm ext}/r}(p,1, 1+2\delta)\times [0,3\delta]\to A^{X^{\rm ext}/r}(p,1, 1+2\delta)$
be the $(1+\tau_p(\theta,r)r,3\delta)$-Lipschitz strong deformation 
retraction of  $A^{X^{\rm ext}/r}(p,1, 1+2\delta)$ to $B^{X^{\rm ext}/r}(p,1)$
given in \eqref{eq:retracionH} (see also 
Corollary \ref{cor:dp(dotH)=-1}).
Then $H\circ\pi\circ\alpha$ is a curve joning
$v$ and $v_i^r$ in $S^{X^{\rm ext}/r}(p,1)$
of length $<7\delta$.
\end{proof}


We show that 
\beq \label{eq:ratio(vivj)}
       \frac{|v_i^r,v_j^r|_{S^{X^{\rm ext}}(p,r)/r}}{|\tilde v_i,\tilde v_j|_{\Sigma_p(X)}}\to 1 \quad \text{as $r\to 0$}
\eeq
for all $1\le i\neq j\le I$.
Choose 
\[
    0<\e\ll\min\{ \delta,\e_0,\sigma_0\},
\]
and take $a_i>0$ for each $1\le i\le I$ such that 
$\tilde v_{i,\e}:=(1+a_i)f_r(v_i^r)$ has norm $1+\e$ in 
$T_p(X)$.
For fixed $1\le i\neq j\le I$, let $\tilde\gamma_{\e}(s)$\,$(0\le s\le |\tilde v_{i,\e},\tilde v_{j,\e}|)$ be a minimal geodesic joining
$\tilde v_{i,\e}$ to $\tilde v_{j,\e}$ in 
$\Sigma^{\e}_p(X):=\Sigma_p(X)\times\{ 1+\e\}\subset T_p(X)$.
Set 
$$
\zeta_{\e}:=g_r\circ\tilde\gamma_{\e}.
$$
Taking small enough $\e$,   we may assume
\[
 \tilde \gamma_{\e} \subset T_p(X)_{\e_0}^{\rm reg}.
\]

\pmed
\begin{figure}
\begin{center}
\begin{tikzpicture}
[scale = 0.4]

\draw [thick] (-1.5,3.5) to [out=-65, in=70] (-1.5,-3.5);
\draw [thick] (2.5,4.5) to [out=-60, in=70] (2.7,-2.7);
\draw [thick] (-1.5,3.5)--(2.5,4.5);
\draw [thick] (-1.5,-3.5)--(2.7,-2.7);
\draw [thick] (0.2,2.5) to [out=-70, in=70] (0.3,-2);
\fill (0.17,2.8) circle(2pt) node[left]
{\tiny{$v_{i}^r$}};
\fill (0.26,-2.4) circle(2pt) node[left]
{\tiny{$v_{j}^r$}};
\draw [thick] (1.5,2.5) to [out=-70, in=70] (1.5,-2.2);
\draw [thin] (1.5,2.5) --(0.2,2.5);
 \draw [thin] (1.5,-2.2)--(0.3,-2);

\draw [thin] (1.8,1.56) --(0.5,1.6);
\draw [thin] (2,0.62) --(0.7,0.7);
\draw [thin] (2,-0.32) --(0.7,-0.2);
\draw [thin] (1.8,-1.26) --(0.55,-1.1);

\fill (1.5,2.5) circle (2pt) node [right] 
{\tiny{$v_{i,\e}$}};
\fill (1.5,-2.2) circle (2pt);
 \fill (1.3,-2.2) circle (0pt) node [right] 
{\tiny{$v_{j,\e}$}};
\fill (1,0) circle (0pt) node [left] 
{\tiny{$\sigma_\e$}};
\fill (1.8,0) circle (0pt) node [right] {\tiny{$\zeta_{\e}$}};
\draw [thick,dotted] (-2.5,3.8) to [out=-80, in=70] (-2.5,-3.8);
\draw [thick,dotted] (3.5,4.8) to [out=-60, in=70] (3.7,-3);
\draw [thick,dotted] (-2.5,-3.8) to [out=-70, in=240] (3.7,-3);
\draw [thick,dotted] (3.5,4.8) to [out=140, in=70] (-2.5,3.8);
\fill(-3,0)  circle (0pt) node [left] 
{\small{$A_r$}};

\draw[->, thick] (6,1)--(9,1);
\fill(7.5,1)  circle (0pt) node [above] 
{\footnotesize{$f_r$}};
\draw[<-, thick] (6,-1)--(9,-1);
\fill(7.5,-1)  circle (0pt) node [below] 
{\footnotesize{$g_r$}};

\draw[shift={(14,0)}] [thick] (-1.5,3.5) to [out=-65, in=70] (-1.5,-3.5);
\draw[shift={(14,0)}] [thick] (2.5,4.5) to [out=-60, in=70] (2.7,-2.7);
\draw[shift={(14,0)}] [thick] (-1.5,3.5)--(2.5,4.5);
\draw[shift={(14,0)}] [thick] (-1.5,-3.5)--(2.7,-2.7);
\fill[shift={(14,0)}] (0.2,2.3) circle (2pt) node [left] {\tiny{$\tilde v_{i}$}};
\fill[shift={(14,0)}] (0.28,-1.8) circle (2pt) node [left] {\tiny{$\tilde v_{j}$}};

\draw[shift={(14,0)}] [thick] (1.5,2.5) to [out=-70, in=70] (1.5,-2.2);
\draw[shift={(14,0)}] [thin] (1.5,2.5) --(0.8,2.5);
\fill[shift={(14,0)}] (0.8,2.5) circle (2pt) node [above] {\tiny{$f_r(v_{i}^r)$}};
 \draw[shift={(14,0)}] [thin] (1.5,-2.2)--(0.8,-2.1);
\fill[shift={(14,0)}] (0.8,-2.1) circle (2pt) ;
\fill[shift={(14,0)}] (0.8,-2.0) circle (0) node [below] {\tiny{$f_r(v_{j}^r)$}};

\fill[shift={(14,0)}] (1.5,2.5) circle (2pt) node [right] {\tiny{$\tilde v_{i,\e}$}};
\fill[shift={(14,0)}] (1.5,-2.2) circle (2pt);
 \fill [shift={(14,0)}](1.3,-2.2) circle (0pt) node [right] {\tiny{$\tilde v_{j,\e}$}};
\fill [shift={(14,0)}](1.8,0) circle (0pt) node [right] {\tiny{$\tilde\gamma_{\e}$}};

\draw[shift={(14,0)}] [thick,dotted] (-2.5,3.8) to [out=-80, in=70] (-2.5,-3.8);
\draw [shift={(14,0)}] [thick,dotted] (3.5,4.8) to [out=-60, in=70] (3.7,-3);
\draw[shift={(14,0)}]  [thick,dotted] (-2.5,-3.8) to [out=-70, in=240] (3.7,-3);
\draw[shift={(14,0)}]  [thick,dotted] (3.5,4.8) to [out=140, in=70] (-2.5,3.8);
\fill[shift={(14,0)}] (7,0)  circle (0pt) node [left] 
{\small{$\tilde A$}};
\end{tikzpicture}
\end{center}
\vspace{-0.3cm}
\caption{}
\end{figure}

In what follows, we use the directions
$\uparrow_{\rho_{\e}(s)}^{\rho_{\e}(s')}$
and
$\uparrow_{\zeta_{\e}(s)}^{\zeta_{\e}(s')}$
for any $s'$ close enough to $s$ in place of considering "directions $\dot\rho_\e(s)$ and 
$\dot\zeta_\e(s)$ " 
that are not necessarily defined.
We now show   
\begin{slem}\label{slem:<zeta,dp>=0}
For any fixed $s\in [0,|\tilde v_{i,,\e},\tilde v_{j,\e}|]$ and any $s'$ close enough to $s$, we have 
\begin{enumerate}
\item
\begin{enumerate}
 \item $|\angle(\uparrow_{\zeta_{\e}(s)}^{\zeta_{\e}(s')},\dot\gamma^Y_{\zeta_{\e}(s),p}(0)-\pi/2|< \tau_p(r,\e_0)$, \quad
 \item $|\zeta_{\e}(s'),\zeta_{\e}(s)|/|s'-s|>1-\tau_p(r,\e_0)\,;$
\end{enumerate}
\item $|\angle(\uparrow_{\zeta_{\e}(s)}^{\zeta_{\e}(s')},\ \xi_{\zeta_{\e}(s)}^+)-\pi/2|<\tau_p(r,\e_0)$ whenever
$\zeta_{\e}(s)\in X_0$.
\end{enumerate}
\end{slem}
\begin{proof} 
Let $\mu_0>0$ be the $(m,\delta)$-strain radius of 
$A_r$ in the sense of \cite{Ya:conv}, and let $0<\sigma\ll \min\{ \sigma_0,\mu_0\}$ be a new parameter.
Put $\rho_{\e}:=g_r^Y\circ\tilde\gamma_{\e}$.
For any fixed $s\in [0,|\tilde v_{i,\e},\tilde v_{j,\e}|]$, set $\tilde y:=\tilde\gamma_{\e}(s)$ and 
$y:=g_r^Y(\tilde y)$ for  simplicity.

 First we show (1) and (2)  for 
$\rho_{\e}(s)$ in place of $\zeta_{\e}(s)$,
where (2) should be replaced in an obvious way.
Let $\tilde w:=\tilde\gamma_{\e}(s+\sigma)$ 
and $w:=\rho_\e(s+\sigma)$.
Note that $\angle(\dot{\tilde\gamma}_{\e}(s), \uparrow_{\tilde y}^{\tilde w})<\tau(\sigma)$.
By the second half of Lemma 4.6 of \cite{Ya:conv}), we have
\beq \label{eq:angle(rhoss)}
\angle(\uparrow_{\rho_{\e}(s)}^{\rho_{\e}(s')},\uparrow_y^w)<\tau_p(r,\e_0,\sigma)
\eeq
for any $s'$ close enough to $s$.
 It should be noted that we can ignore another 
parameter $\sigma_1$ appeared in 
\cite[Lemma 4.6]{Ya:conv} by taking
$\sigma_1\ll \sigma$, $r\ll \sigma_1$.

Take a point $q$ with $|p,q|=2r$ and $\wangle pyq >\pi-\tau_p(r)$.

\eqref{eq:|f-varphi|}  implies  
\[
|\wangle pyw -\pi/2|\le 
|\wangle pyw -\wangle o_p\tilde y\tilde w|+
|\wangle o_p\tilde y\tilde w|-\pi/2|
\le \tau_p(r,\sigma).
\]
Similarly we have 
\[
|\wangle qyw -\pi/2|\le \tau_p(r,\sigma).
\]
By Lemma \ref{lem:comparison} applied to the
$(1,\tau_p(r))$-strainer $(p,q)$ at $y$, this implies 
\begin{align*}
  |\angle pyw -\pi/2| &\le 
 |\angle pyw -\wangle pyw|+ |\wangle pyw -\pi/2|
    <\tau_p(r,\sigma),
 \end{align*}
and similarly
\[
 |\angle qyw -\pi/2| <\tau_p(r,\sigma).
\]
From \eqref{eq:angle(rhoss)}, we then have 
\beq \label{angle(prho=pi/2}
\begin{aligned}
&|\angle p\rho_{\e}(s) \rho_{\e}(s')-\pi/2|<\tau_p(r,\e_0,\sigma), \\
&|\angle q\rho_{\e}(s) \rho_{\e}(s')-\pi/2|<\tau_p(r,\e_0,\sigma).
\end{aligned}
\eeq
On the other hand, by Lemma \ref{eq:dpi-x},
we obtain
\begin{align} \label{eq:angle(dpi)}
   \angle(d\pi(\uparrow_y^p), \uparrow_z^p)<\tau_p(r),\quad
\angle(d\pi(\uparrow_y^q), \uparrow_z^q)<\tau_p(r).
\end{align}
Since $\angle pzq>\pi-\tau_p(r)$, \eqref{eq:angle(dpi)}  implies
\begin{align}\label{eq:angle(dpi,dpi)}
   \angle(d\pi(\uparrow_y^p), d\pi(\uparrow_y^q)) >\pi-\tau_p(r).
\end{align}
Let $\xi^+(y)\in\Sigma_y(Y)$ denote the
direction of the perpendicular 
through $y$.
Since $|y,z|<\tau_p(r)$ and 
since $|\angle(\uparrow_y^p,\xi^+(y))-\pi/2|<\tau_p(r)$ and
$|\angle(\uparrow_y^q,\xi^+(y))-\pi/2|<\tau_p(r)$, we have 
\beq\label{eq:dpi(yp)}
   |d\pi(\uparrow_y^p)|>1-\tau_p(r), \quad
    |d\pi(\uparrow_y^q)|>1-\tau_p(r).
\eeq
Moreover since $\pi$ is locally
$\phi(\tau_p(r))^{-1}$-Lipschitz at $y$,
 so is $d\pi$  at $y$.
It follows from \eqref{angle(prho=pi/2} and \eqref{eq:dpi(yp)} that 
\begin{align*}
&\angle(d\pi(\uparrow_y^p),
 d\pi\bigl(\uparrow_{\rho_{\e}(s)}^{ \rho_{\e}(s')}\bigr))<\pi/2+\tau_p(r,\e_0,\sigma),\quad \\
&\angle(d\pi(\uparrow_y^q),
 d\pi\bigl(\uparrow_{\rho_{\e}(s)}^{ \rho_{\e}(s')}\bigr))<\pi/2+\tau_p(r,\e_0,\sigma).
\end{align*}
\eqref{eq:angle(dpi,dpi)} then implies 
that the above inequalities are almost equal sign.
Thus together with 
\eqref{eq:angle(dpi)},  we have 
\beq \label{eq:angle(zp,dpi(rho))}
|\angle(\uparrow_z^p, d\pi\bigl(\uparrow_{\rho_{\e}(s)}^{ \rho_{\e}(s')}\bigr))-\pi/2|<
\tau_p(r,\e_0,\sigma),
\eeq
which implies  (1)(a).

\par\n
(2)\,
By \cite[Lemma 1.8]{Ya:conv}, choose $a:=\gamma_z^+(t)$ and $b\in {\rm int}X$
with $|a,z|>\sigma$ and $|b,z|>\sigma$  such that $\wangle azb>\pi-\tau(\e_0,\sigma)$.
Let 
$\tilde a:=f_r^Y(a)$ and $\tilde b:=f_r^Y(b)$.
Since $|\wangle \tilde a\tilde y\tilde w-\pi/2|<\tau(r,\sigma)$ and 
$|\wangle \tilde b\tilde y\tilde w-\pi/2|<\tau(r,\sigma)$, replacing $p,q$ by $a,b$,
in a way similar to \eqref{angle(prho=pi/2},
we have 
\beq \label{angle(abrho=pi/2}
\begin{aligned}
&|\angle a\rho_{\e}(s) \rho_{\e}(s')-\pi/2|<\tau_p(r,\e_0,\sigma), \\
& |\angle b\rho_{\e}(s) \rho_{\e}(s')-\pi/2|<\tau_p(r,\e_0,\sigma).
\end{aligned}
\eeq
Replacing $p,q$ by $a,b$, 
in a way simiar to \eqref{eq:angle(zp,dpi(rho))},
we obtain 
\beq \label{eq:angle(za,dpi(rho))}
|\angle(\uparrow_z^a, d\pi\bigl(\uparrow_{\rho_{\e}(s)}^{ \rho_{\e}(s')}\bigr))-\pi/2|
<\tau_p(r,\e_0,\sigma),
\eeq
yielding (2).

(1)(b) is an immediate consequence of 
\eqref{angle(abrho=pi/2}.
This completes the proof.
\end{proof}

\psmall
Given $1\gg\theta\gg \delta>0$, let $r_0:=r_0(p,\theta,\delta)$ be small enough.
Starting wih a maximal $\theta r_0$-discrete
set of $S^Y(p,r_0)$, we apply
 Proposition \ref{prop:1-Lip-contractibleAlex} to 
$x=p$ and $\e=r$ and $r=r_0$, 
we obtain a  domain $W_r$ containing $B^Y(p,2r)$ and a 
$(1,(1+\zeta)r)$-Lipschitz strong deformation retraction of 
$W_r$  to $p$. Taking $1/r$-rescaling into account, we have 
$(1,(1+\zeta))$-Lipschitz strong deformation retraction
$\Phi:W_r\times [0,(1+\zeta)]\to W_r$.
Let $(W_r)_X:=W_r\cap X$, and let
\[
H:=\pi\circ\Phi:
(W_r)_X\times [0,(1+\zeta)]\to (W_r)_X
\]
be a $(1+\tau_p(\theta,\zeta,r/r_0)r, 1+\zeta)$-Lipschitz strong 
deformation retraction of $(W_r)_X$ to $p$
as in \eqref{eq:retracionH}.

\pmed
By Lemma \ref{lem:gradH-6}, for each $s\in [0,1]$, 
we choose $t_s\in (0,2\e)$ in such a way that 
$$
\sigma_\e(s):=H(\zeta_{\e}(s),t_s)\in S^{X^{\rm ext}}(p,r).
$$


\begin{slem}
\beq \label{eq:L(plroj)}
      | L(\sigma_\e)/L(\zeta_{\e})-1|<\tau_p(r,\e_0).
\eeq
\end{slem}
\begin{proof}
We first show that 
\beq\label{eq:ts/ts}
\text{$|t_{s'}-t_s|/|s'-s|<\tau_p(r,\e_0)$ for any $s,s'\in [0,1]$,}
\eeq
if $\theta, \zeta, r/r_0$ are small enough.
By Corollary \ref{cor:dp(dotH)=-1},
we have 
 \[
 ||d_p^Y(H(\zeta_{\e}(s'),t_{s'}))-d_p^Y(H(\zeta_{\e}(s'),t_{s}))|/|t_{s'}-t_s| -1|\le \tau_p(\theta,\zeta,r/r_0)
\]
Moreover, from the law of cosines with
Sublemma \ref{slem:<zeta,dp>=0}, we get
\begin{align*}
 |d_p^Y(H(\zeta_{\e}(s')&,t_{s}))-d_p^Y(H(\zeta_{\e}(s),t_{s}))|\\
&\le (\tau_p(r,\e_0)+|s-s'|)|s-s'|\le \tau_p(r,\e_0)|s-s'|
\end{align*}
for small enough $|s-s'|$.
Thus combining these inequalities, we have 
\begin{align*}
  0&=|d_p^Y(H(\zeta_{\e}(s'),t_{s'}))-d_p^Y(H(\zeta_{\e}(s'),t_{s}))\\
&\hspace{0.5cm}
+d_p^Y(H(\zeta_{\e}(s'),t_{s}))-d_p^Y(H(\zeta_{\e}(s),t_{s}))| \\
&\ge |d_p^Y(H(\zeta_{\e}(s'),t_{s'}))-d_p^Y(H(\zeta_{\e}(s'),t_{s}))|  \\
&\hspace{0.5cm}
  -|d_p^Y(H(\zeta_{\e}(s'),t_{s}))-d_p^Y(H(\zeta_{\e}(s),t_{s}))| \\
&\ge (1-\tau_p(\theta,\zeta,r/r_0))|t_{s'}-t_s|-\tau_p(r,\e_0)|s'-s|,
\end{align*}
from which \eqref{eq:ts/ts} follows.
\pmed
Now we obtain from \eqref{eq:ts/ts} that
\begin{align*}
 |\sigma_\e&(s),\sigma_\e(s')|_Y  \\
 & \le |H(\zeta_{\e}(s),t_s), H(\zeta_{\e}(s'),t_s)|_Y+
    |H(\zeta_{\e}(s'),t_s), H(\zeta_{\e}(s'),t_{s'})|_Y\\
& \le(1+\tau_p(r,\e_0))|\zeta_{\e}(s),\zeta_{\e}(s')| +
      (1+\tau_p(\theta,\zeta,r/r_0))\tau_p(r,\e_0)|s-s'|  \\
& \le (1+\tau_p(r,\e_0))|s-s'|,
\end{align*}
for  small enough $\theta, \zeta, r/r_0$.
In a similar way, we have the opposite inequality to conclude
\[
| |\sigma_\e(s),\sigma_\e(s')|_Y/|s-s'|-1|
<\tau_p(r,\e_0).
\]
Since
\[
| |\zeta_{\e}(s), \zeta_{\e}(s')|_Y/|s-s'|-1|<\tau_p(r,\e_0),
\]
we get the required estimate \eqref{eq:L(plroj)}.
\end{proof}

\pmed

 From construction, the distances between $v_i^r, v_j^r$ and the 
corresponding endpoints of $\zeta_{\e}$ are less than
$c\e$ for a uniform constant $c$. 
Therefore one can verify that 
the intrinsic distances in
$S^{X^{\rm ext}}(p,r)/r$ between $v_i^r, v_j^r$ 
and the corresponding endpoints of $\sigma_{\e}$ are less than $c'\e+\tau_p(r)$ as in the proof of Sublemma \ref{slem:delta-dense}.

Using  \eqref{eq:L(plroj)}, we have 
\begin{align*}
  |v_i^r,v_j^r|_{S^{X^{\rm ext}}(p,r)/r}&\le L(\sigma_\e) +c'\e+\tau_p(r)
    \le (1+\tau_p(r,\e_0))L(g^Y_r\circ\tilde\gamma_{\e})+c\e \\
  &\le (1+\tau_p(r,\e_0))L(\tilde\gamma_{\e})+c'\e+\tau_p(r)  \\
& \le  (1+\tau_p(r,\e_0))|\tilde v_i, \tilde v_j|_{\Sigma_p(X)}+c'\e+\tau_p(r),
\end{align*}
and hence 
\[\frac{
|v_i^r,v_j^r|_{S^{X^{\rm ext}}(p,r)/r}}{|\tilde v_i, \tilde v_j|_{\Sigma_p(X)}}\le 1+\tau_p(r,\e_0) +(c'\e+\tau_p(r))/\delta.
\]
Making use of the map $f_r$ and the canonical gradient flow 
for $d_{o_p}$ on $T_p(Y)\setminus \{ o_p\}$ in place of 
$g_r$ and $\Phi$, we have the reverse inequality in a similar way. We omit the detail.
Thus we obtain \eqref{eq:ratio(vivj)}.

Next suppose that $p\in \pa Y$. Taking  the double $D(Y)$, and considering the convergence 
\[
B^{D(Y)}(p,2r)/r  \xrightarrow{GH} T^2_p(D(Y))
 \quad \text{as $r\to 0$}
\]
we get \eqref{eq:ratio(vivj)} in a similar way.
This completes the proof of Proposition \ref{prop:int-metric-sphereGH}.
\end{proof}

\pmed

For the proof of Theorem \ref{thm:metric-sphere}, 
in view of Proposition \ref{prop:int-metric-sphereGH}, it suffices to show that 
$S^{X^{\rm ext}}(p,r)/r$ is locally $C$-doubling and  locally $C$-Lipschitz contractible for large enough uniform constant $C$ and small enough $0<r\le r_p$.

First we consider the latter.
Let $Y$ be an Alexandrov space extending $X$.
Theorem \ref{thm:X1-f} together with 
\eqref{eq:volume-Sigma} shows 
$ \ca H^{{\rm dim}Y-1}(\Sigma_p(Y)\ge c_1(v)>0$, and hence 
\beq\label{eq:Vol(B(2r))}
              \ca H^{{\rm dim} Y}(B^Y(p,2r)/r)\ge c_2(v)>0
\eeq
for any small enough $r\le r_p$,
where $c_i(v)$\,$(i=1,2)$ are constants 
depending only on $\dim Y$ and the lower volume bound
$v$ in \eqref{eq:volume-Sigma}.
Therefore we can apply Theorem \ref{thm:conv-cover}
to each point of $B^Y(p,r)/r$.

Let $C=C(n,\kappa,\nu,\lambda,d,m,c_2(v))>1$ be the constant provided in  
Theorem \ref{thm:conv-cover}.
For any $x\in S^{X^{\rm ext}}(p,r)/r$ and
$0<\e\ll C^{-1}$, 
choose a $(C,C\e)$-Lipschitz deformation retraction
$F:W\times [0,C\e]\to W$ of $W$ to a point with $B^Y(x,\e)\subset W$
provided in Theorem \ref{thm:conv-cover}.
If $\e$ is small enough, we have 
\[
W\subset  A^Y(p, r-C\e,r+C\e)/r.
\]
Choose $q\in X$ with $|p,q|_Y=2r$ such that 
\beq \label{eq:wangle(pxq)8}
\wangle^Y pyq>\pi-\tau_p(r,\e)
\eeq
for all $y\in W$.
Let $\Psi$ denote the $d_p^Y$-gradient flow
on $A^Y(p, r-C\e,r+C\e)/r$.  
\eqref{eq:gradient} and \eqref{eq:wangle(pxq)8} yield
\[
\angle(\dot\Psi(y,t),\uparrow_{\Psi(y,t)}^q)<\tau_p(r,\e), \quad
|\dot\Psi(y,t)|>1-\tau_p(r,\e).
\]

For simplicity, we put $U:=W\cap X$.
In a way similar to \eqref{eq:|z,X|<small},
we have 
\[
         |\Psi(y,t), X|\le \tau_p(r,\e) t
\]
for all $(y,t)$ with $y\in U$ and $\Psi(y,t)\in
A^Y(p, r-C\e,r+C\e)/r$.
Let us consider 
\[
      \psi:=\pi\circ\Psi:U\times\R_+\to X
\]
By Lemmas \ref{lem:derivative-pi} and \eqref{lem:derivative-pi-0},
$\psi(y,t)$ has the right derivative with respect to $t$.

\begin{lem}\label{lem:dp-increasing}
$d^Y_p$ is increasing along the homotopy
$t\mapsto \psi(y,t)$ as long as $\psi(y,t)\in
U$.
\end{lem}
\begin{proof}
In a way similar to Lemma \ref{eq:dpi-x}, 
we obtain 
\[
\angle( d\pi(\uparrow_{\Psi(y,t)}^q),
\uparrow_{\psi(y,t)}^q) < \tau_p(r,\e),
\]
from which we have
\begin{align*}
\angle(\dot\psi(y,t),\uparrow_{\psi(y,t)}^q)
&\le \angle(\dot\psi(y,t), d\pi(\uparrow_{\Psi(y,t)}^q)) +
\angle( d\pi(\uparrow_{\Psi(y,t)}^q),
\uparrow_{\psi(y,t)}^q) \\
&\le \tau_p(r,\e).
\end{align*}
\eqref{eq:wangle(pxq)8} then yields the conclusion.
\end{proof}

Now we define a map
\beq\label{eq:map-alpha}
\alpha_x:U\to S^{X^{\rm ext}}(p,r)/r
\eeq
as follows. Set
\[
    U_{-}:=U\cap \{ d_p^Y\le r\},\quad  U_{+}:=U\cap \{ d_p^Y\ge r\}, \quad
    \Pi:=U\cap S^{X^{\rm ext}}(p,r).
\]
For each $y\in U_{-}$, there is a unique $t_y\ge 0$ such that 
$\psi(y,t_y)\in S^{X^{\rm ext}}(p,r)/r$.
Define $\alpha_-:U^-\to S(p,r)/r$ by
\[
     \alpha_-(y)=\psi(y, t_y).
\]
In a similar way, using $d_q^Y$-gradient flow,
we construct a map $\alpha_+:U^+\to S^{X^{\rm ext}}(p,r)/r$.
Combining $\varphi_+$ and $\varphi_-$, we obtain 
the  map $\alpha_x$ in \eqref{eq:map-alpha}.

\begin{slem}\label{slem:alpha-lip}
$\alpha_x$ is $C_0$-Lipschitz for some uniform constant
$C_0$.
\end{slem}
\begin{proof}
First we show that $t_y$ is $(1+\tau_p(r,\e))$-Lipschitz. 
Since $d_p^Y$ is $\lambda$-concave for some 
uniform constant $\lambda>0$, by Proposition \ref{prop:lip},
$\varphi$ is $e^{3\lambda/2}$-Lipschitz
on $A^Y(p, r/2, 2r)/r$.
Thus we have  
\begin{align*}
  0&=|d_p^Y(\varphi(y',t_{y'}))-d_p^Y(\varphi(y',t_{y})) 
+d_p^Y(\varphi(y',t_{y}))-d_p^Y(\varphi(y,t_{y}))| \\
&\ge |d_p^Y(\varphi(y',t_{y'}))-d_p^Y(\varphi(y',t_{y})) | 
  -|d_p^Y(\varphi(y',t_{y}))-d_p^Y(\varphi(y,t_{y}))| \\
&\ge (1-\tau_p(r))|t_{y'}-t_y|-e^{C\e}|y,y'|,
\end{align*}
and hence $|t_{y'}-t_y|\le (1+\tau_p(r,\e))|y,y'|$.
The conclusion follows immediately.
\end{proof}

By Theorem \ref{thm:conv-cover},
$H:=\pi\circ F$ provides with a $(C,C\e)$-Lipschitz
deformation retraction of $U=W\cap X$ to a point.
We replace $H$ to another one suitable for
the map $\alpha_x$.

\begin{lem}\label{lem:varphiget/9}
For given $1>a>0$, 
there exist $C'>1$ satisfying the following:
For any 
$p\in X$ with \eqref{eq:volume-Sigma},
choose $r_p>0$ satisfying \eqref{eq:Vol(B(2r))}.
 Then for any small enough $0<r\ll r_p$,  any $x\in S^{X^{\rm ext}}(p,r)/r$ and $0<\e<(C')^{-1}$,
there is a 
$(C', C'\e)$-Lipschitz deformation retraction
$H:U\times [0,C'\e]\to U$ to a point with
$B^{X^{\rm ext}/r}(x,\e)\subset U$ 
such that 
\beq \label{eq:apart-varphi}
  |\alpha_x\circ H(y,t), \pa U|_{Y/r}\ge (1-a)(t+|y,\pa U|_{Y/r}),
\eeq
for all $(y,t)\in U\times [0,C'\e]$ with  
$H(y,t)\in U\setminus B^{{X^{\rm ext}}/r}(x,\e/100)$, where 
$\alpha_x$ is defined as in \eqref{eq:map-alpha} for 
$x$.
 \end{lem}

\begin{proof}
Now suppose  the assertion does not hold.
Then for some $1>b>a$, 
we have sequences $p_i\in X_i\subset Y_i$ with  \eqref{eq:volume-Sigma}, $r_i$ with $r_i/r_{p_i}\to 0$, 
$0<\e_i<C_i^{-1}$ and $x_i\in S^{X_i^{\rm ext}}(p_i,r_i)/r_i$ 
such that 
for any $(C_i,C_i\e_i)$-Lipschitz deformation retractions 
$H_i:U_i\times [0,C_i\e_i]\to U_i$ to a point with
$B^{{X^{\rm ext}}/r_i}(x_i,\e_i)\subset U_i$ constructed in 
Theorem \ref{thm:conv-cover}(4),
\eqref{eq:apart-varphi} does not hold.
Namely we have
\beq \label{eq:|alpha,paU|<1/9}
  |\alpha_i\circ H_i(y_i,t_i), \pa U_i|_{Y_i/r_i} \le (1-a)(t_i+|y_i,\pa U_i|_{{Y_i}/r_i})
\eeq
for some $y_i\in U_i$ with $H_i(y_i,t_i)\in U_i\setminus B^{Y_i/r_i}(x_i, \e_i/100)$.
Here we assume that the parameters $\theta=\theta_i$, $\zeta=\zeta_i$ and 
$R=R_i$ defining the convex domain $W_i$ of $Y_i$
with $U_i=W_i\cap X_i$ satisfy $\theta_i\to 0$, $\zeta_i\to 0$ and 
$R_i\to\infty$ as $i\to \infty$ (see Step 1 in the proof
of Theorem \ref{thm:conv-cover}).
Taking smaller $r_i\ll r_{p_i}$, we may also assume
that there is $q_i$ with $|p_i,q_i|_{Y_i/r_i}=2$
such that $\wangle p_ix_iq_i>\pi-o_i$.
Let $z_i:=H_i(y_i,t_i)$ and $s_i:=|z_i, \pa U_i|_{Y_i/r_i}$,  and 
consider the rescaling limit
\beq \label{eq:conv(s-1Y)}
    (s_i^{-1}(Y_i/r_i), z_i) \to (Y_\infty, z_\infty).
\eeq
By the splitting theorem,  $Y_\infty$ is isometric to a product
$\R\times Z_\infty$,
where $S^Y(p_i,r_i)$ converges to $\{ 0\}\times Z_\infty$.
Let 
$U_\infty$, $H_\infty$, $\alpha_\infty$ and
$y_\infty$ 
be the respective limits of 
$U_i$, $H_i$, $\alpha_i$ and $y_i$ 
under the above convergence. 

We show that 
\beq \label{eq:zinfty}
   z_\infty\in \{ 0\}\times Z_\infty.
\eeq
Since $y_\infty\in \{ 0\}\times Z_\infty$, 
we assume $y_\infty\neq z_\infty$.
Let $\Phi_i$ be the gradient flow such that 
$H_i=\pi_i\circ \Phi_i$ outside $B^{Y_i/r_i}(x_i,\e_i/100)$. 
Let $w_i:=\Phi_i(y_i,t_i)$.
By \eqref{eq:|z,X|<small},
we get 
\[
       |w_i,z_i|<o_i t_i.
\]
Lemma \ref{lem:almost-grad-phi} implies
$s_i\ge (1-o_i)+|y_i,\pa U_i|$.
Thus in the limit under \eqref{eq:conv(s-1Y)}, we have
$w_\infty=z_\infty$.
Take $x_i'\in\gamma_{x_i,y_i}^{Y_i}$ such that 
$|x_i',y_i|=2t_i$.
Sublemmas \ref{slem:grad-f-Phi} and  \ref{slem:Phi-gamma} imply 
\[
\angle    x_i'y_iw_i <o_i
\]
This shows that $z_\infty$ is on the geodesic
joining $x_\infty'$ and $y_\infty$.
On the other hand, $d_{p_i}$ and $d_{q_i}$ are 
$o_i$-concave near $S^{Y_i}(p_i,r_i)/r_i$ under
the $1/s_i$-rescaling.
 This yields $x_\infty'\in \{ 0\}\times Z_\infty$,
and hence $z_\infty\in \{ 0\}\times Z_\infty$
showing \eqref{eq:zinfty}.
\psmall

We may assume that 
 $t_i/s_i$ converges to $t\in  [0,1]$.
By Lemma \ref{lem:almost-grad-phi}, 
we have in the limit
\beq \label{eq:|zZ|=1}
\begin{aligned}
1=|z_\infty,\pa U_\infty|& \ge  t+|y_\infty,\pa U_\infty| \\
           & > (1-a)(t+|y_\infty,\pa U_\infty|).
\end{aligned}
\eeq
Note that  
$\alpha_\infty$ must be the canonical projection
to   $\{ 0\}\times Z_\infty$ along the $\R$-factors, and therefore from \eqref{eq:zinfty}, we have
$\alpha_\infty(z_\infty)=z_\infty$.
However, \eqref{eq:|alpha,paU|<1/9} and \eqref{eq:|zZ|=1} 
yield a contradiction:
\beq \label{eq:alpha,paU|<9}
1= |\alpha_\infty (z_\infty),  \pa U_\infty| 
\le (1-a)(t+|y_\infty, \pa U_\infty|) < 1.
\eeq
This completes the proof of Lemma \ref{lem:varphiget/9}.
\end{proof}

 The idea of the proof of Theorem \ref{thm:metric-sphere} is indicated in 
Fig. \ref{fig.idea}.
\pmed
\begin{figure}
\begin{center}
\begin{tikzpicture}
[scale = 0.33]

 \draw(0,0) circle[radius=2.6];
\fill[thick]  (0.14,0) circle (3pt);
\draw [thick] (-1,3) to [out=-50, in=50] (-1,-3);
\fill (-3.5,0) circle (0pt) node [above] {\small{$U$}};
\fill (-2,-3) circle (0pt) node [below] {\footnotesize{$S^{X^{\rm ext}}(p,r)/r$}};

\draw [->,thick] (-2,1)--(-0.2,1);
\draw [->,thick] (-2,0)--(-0.2,0);
\draw [->,thick] (-2,-1)--(-0.2,-1);
\fill (-1,1) circle (0pt) node [above] {\footnotesize{$\alpha$}};

\draw [->,thick] (2,1)--(0.3,1);
\draw [->,thick] (2,0)--(0.3,0);
\draw [->,thick] (2,-1)--(0.3,-1);

\draw(10,0) circle[radius=2.6];
\fill (13.5,0) circle (0pt) node [above] {\small{$U$}};

\draw [->,thick] (10,2)--(10,0.5);
\fill (10.2,1.8) circle (0pt) node [left] {\tiny{$H$}};
\draw [->,thick] (10,-2)--(10,-0.5);
\draw [->,thick] (8,0)--(9.5,0);
\draw [->,thick] (12,0)--(10.5,0);
\draw [->,thick] (8.6,1.3)--(9.7, 0.28);
\draw [->,thick] (8.6,-1.3)--(9.7, -0.28);
\draw [->,thick] (11.4,1.3)--(10.3, 0.28);
\draw [->,thick] (11.4,-1.3)--(10.3, -0.28);

\fill[thick]  (10,0) circle (2pt);

\fill[shift={(0,-1)}] (6,-3.5) circle (0pt) node [left] 
{\large$\Downarrow$};

\draw[shift={(0,-2)}](5,-7) circle[radius=2.6];
\fill[shift={(0.2,-2)}](5,-7) circle (0pt) node[left]
{\tiny{$\alpha\circ H$}};
\draw[shift={(0,-2)}] [ thick] (4.5,-4.4) to [out=-65, in=65] (4.5,-9.6);

\draw[shift={(0,-2)}] [->, thick] (4.5,-4.4) to [out=-60, in=100] (5.2,-6.5);
\draw[shift={(0,-2)}] [->, thick] (4.5,-9.6) to [out=60, in=-100] (5.2, -7.5);
\fill[shift={(0,-2)}][thick]  (5.13, -7) circle (3.5pt);
\fill[shift={(0,-2)}] (4,-9.7) circle (0pt) node [below] {	\footnotesize{$\Pi$}};
\end{tikzpicture}
\end{center}
\caption{Idea of the proof of Theorem \ref{thm:metric-sphere}}\label{fig.idea}
\end{figure}

\begin{proof}[Proof of Theorem \ref{thm:metric-sphere}]
Let $C'$ be as in Lemma \ref{lem:varphiget/9},
and use the same symbol $C:=C_0(C')^2$. We show that 
$S^{X^{\rm ext}}(p,r)/r$ is locally $C$-Lipschitz 
contractible. 
For any $x\in S^{X^{\rm ext}}(p,r)/r$ and 
any $0<\e<(C')^{-1}$,  
let $H:U\times [0,C'\e]\to U$ be the 
$(C', C'\e)$-Lipschitz deformation retraction
with $B(x,\e)\subset U$ as in Lemma \ref{lem:varphiget/9}.
Set $\Pi:=U\cap S^{X^{\rm ext}}(p,r)/r$.
Let $\iota:\Pi\to U$ be the inclusion, and 
consider the $(C_0C',C'\e)$-Lipschitz homotopy
$$
\Psi:=\alpha_x\circ H\circ(\iota, {\rm id}):\Pi\times [0,C'\e]\to S^{X^{\rm ext}}(p,r)/r.
$$
Lemma \ref{lem:varphiget/9} confirms that
$\Psi_t(\Pi)\subset \Pi$ for all $t\in[0,C\e]$.
Thus
$\Psi:\Pi\times [0,C'\e]\to\Pi$
provides $(C_0C',C'\e)$-Lipschitz deformation retraction to a point.
This completes the proof of the
local $C$-Lipschitz 
contractibility of $S^{X^{\rm ext}}(p,r)/r$.

Using the $e^{3\lambda/2}$-Lipschitz retraction $\varphi$,  we also verify that 
the intrinsic metric (length metric) and the extrinsic  metrc of $S^{X^{\rm ext}}(p,r)/r$
are $C(\lambda)$-biLipschitz homeomorphic to each other.
Since ${X^{\rm ext}}/r$ is locally $C_1$-doubling
with uniform $C_1$, so is 
$S^{X^{\rm ext}}(p,r)/r$. 
This completes the proof of Theorem \ref{thm:metric-sphere}.       
\end{proof}  

\pmed\n
{\bf Lipschitz homotopy stability of small punctured balls}.
\pmed\n

\begin{proof}[Proof of Theprem \ref{thm:punctured-ball}]
Let $Y$ be an Alexandrov space extending $X$.
As before,  we have 
$$
              \ca H^{{\rm dim} Y}(B^Y(p,2r)/r)\ge c_0(v)>0
$$
for any small enough $r$.
Therefore we can apply the method developed in 
\cite{FMY} together with Theorem \ref{thm:conv-cover}
to $B^Y(p,2r)/r$.
We proceed along the same line as in
\cite{FMY}. A main difference is that
we have to deal with infinite open coverings with different rescalings.
Therefore we will give only an outline emphasizing the main differences.

For each integer $i\ge 0$, let us consider the annuli 
\begin{align*}
    & \tilde A_i:=A^{T_pX}(o_p, r/2^{i+1}, r/2^{i}), \quad
        A_i:=A^{X^{\rm ext}}(p, r/2^{i+1}, r/2^{i}), \\
     & \tilde A_*:=A^{T_pX}(o_p, 1/2, 1).
\end{align*}
For a small enough $\e>0$, we take a maximal $\e/2$-discrete set
$\{ \tilde x_j\}_{j=1}^J$ in $\tilde A_*$, and set $\tilde V_j:=B(\tilde x_j,2\e)$.
Consider the covering
$\tilde{\ca V}_*=\{\tilde V_j\}$ of $\tilde A_*$.
In what follows, we use the abbraviation 
\[
     a U:=\{ au\,|\,  u\in U\} \subset T_pX
\]
for $U\subset T_pX$ and $a>0$. 
We also set $\e_i:=2^{-i}\e$
For each $i\ge 1$, set $\tilde V_j^i:=(r/2^{i})^{-1} \tilde V_j$
and 
\[ 
        \tilde{\ca V}_i:=\{ \tilde V_j^i\}_j,
\]
which is a covering of $\tilde A_i$.
Let $\tilde{\ca V}$ be the collection of all elements in $\tilde{\ca V}_i$ for all
$i\ge 0$. .

Note that $d_{GH}(r^{-1}B^{X^{\rm ext}}(p,r),T_p^1X)<\tau_p(r)$, where we may assume
that $\tau_p(r)$ is increasing in $r$.
Therefore by \cite{FMY}, there is small enough $r$ satisfying 
the following: For any $i\ge 0$, consider the convergence 
\[
  (*)_i\quad     (r/2^i)^{-1}B^{X^{\rm ext}}(p, r/2^i)\to T_p^1.
\]
We choose metric balls $V_j^i$ corresponding to $\tilde V_j$
under $(*)_i$.
Note that  $\ca V_i:=\{ V_j^i\}_j$  is a covering of $A_i$.

Let $\ca V$ be the collection of all elements in $\ca V_i$ for all
$i\ge 0$, and let $\ca N$ be the nerve of  $\ca V$.
We define the \textit{$\{ \epsilon_i r\}$-geometric realization} 
$|\ca N|$ of $\ca N$  as follows.
Let $\{ v_j \}_{j=1}^J$ be the set of all vertices
of $\ca N$.
For each  $v_j$, find $i$ such that 
$v_j$ corresponds  to an element 
of the covering $\ca V_i$. We identify $v_j$ with the $j$-th standard vector of norm $\epsilon_i r$ in Euclidean space, i.e.,
\[
v_i=(0,\dots,\overset{j}{\epsilon_i r},\dots)\in\mathbb R^\infty.
\]
The set $|\ca N|$ is defined as the union of  all the convex hulls of 
$\{v_i\}_i$ corresponding the simplices of $\ca N$.
We consider the intrinsic metric on $|\ca N|$ induced by the Euclidean metric.

Define a $C$-Lipschitz map $\Theta:B^{X^{\rm ext}}(p,r)\setminus \{ p\}\to |\ca N|$
by using a partition of unity associated with $\ca V$.
For each simplex $\sigma$ of $\ca N$, we choose a $(C,\e_ir)$-Lipschitz
contractible domain  $U_\sigma$ in a systematic way as in \cite{FMY}.
More precisely, we choose $U_\sigma$ by reverse induction on  
the dimension of $\sigma$ as follows. For any $\sigma\in\ca N$, 
let $B_\sigma$ denote the corresponding nonempty intersection of
elements of $\ca V$.
For each vertex $v$ of $\ca N$, define $i=i_v$ 
by the condition that $v$ 
corresponds to an element of $\ca V_i$.
Let  $i_\sigma$ be the maximum of $i_v$ for all
the vertices $v$ of $\sigma$.
For $\sigma$ of maximal dimension, we choose a 
$(C,r\e_{i_\sigma})$-Lipschitz
contractible domain  $U_\sigma$ containing $B_\sigma$.
For nonmaximal simplex $\tau$, we choose a
$(C,r\e_{i_\tau})$-Lipschitz
contractible domain  $U_\tau$ containing 
$$
       B_\tau\cup(\cup_{\tau\subset\sigma} U_\sigma).
$$

Using $\{ U_\sigma\}_{\sigma\in\ca N}$ based on the notion of domination by polyhedra (\cite{Peter}),
we construct a $C$-Lipschitz map
$\zeta:|\ca N|\to B(p,(1+C\e)r)\setminus\{ p\}$ such that 
$\zeta\circ\Theta$ is $(C,r\e)$-Lipschitz homotopic to the inclusion
$\iota:B^{X^{\rm ext}}_*(p,r)\to 
B^{X^{\rm ext}}_*(p,(1+C\e)r)$ (see \cite{FMY}).
Note that this is possible since the diamters of $\sigma$ and $U_\sigma$ are comparable
for all $\sigma\in \ca N$.

We can also construct a $C$-Lipschitz map $f:|\ca N|\to T_p^{(1+C\e)r}X_*$ as in \cite{FMY}.
Set 
$$
F:=\tilde\chi\circ f\circ\Theta:B^X_*(p,r)\to
T_p^rX_*,
$$
where $\tilde\chi:T_p^{2r}X \to T_p^rX$ denotes the 
canonical retraction along the geodesics to $o_p$.

Let $\tilde{\ca N}$ be the nerve of  $\tilde{\ca V}$.
Let $\chi:B^{X^{\rm ext}}(p,2r)\to B^{X^{\rm ext}}(p,r)$
be the $C$-Lipschitz deformation retraction 
defined via the flow $H$ in \eqref{eq:retracionH}.
Replacing $\ca N$ and $\tilde{\chi}$  by $\tilde{\ca N}$
and $\chi$,
we define a $C$-Lipschitz map  $\tilde F:T_p^rX_*\to
B_*(p,r)$ in a similar way.
As in \cite{FMY}, we can show that $F$ and $\tilde F$  give $(C,Cr\e)$-Lipschitz homotopy equivalences.
This completes the proof.
\end{proof}

\pmed\n
{\bf Diffeomorphism stability.}\,
We discuss a stability  in  $\ca M(n,\kappa,\nu,\lambda,d)$.
For $v>0$, let $\ca M(n,\kappa,\nu,\lambda,d,v)$ be
the non-collapsing moduli space consisting of all
$M\in \ca M(n,\kappa,\nu,\lambda,d)$ with
${\rm vol}(M)\ge v$.
Let us consider an example of a flat disk with an inside open disk removed.
If the inside disk approaches to the outside boundary, then we have a singular surface in the limit.
This shows that in contrast to the closed manifold case, the topological stability 
does not holds in the closure $\overline{\ca M(n,\kappa,\nu,\lambda,d,v)}$.
It was proved in \cite[Theorem 1.5]{wong0}  that if $M_i\in \ca M(n,\kappa,\nu,\lambda,d,v)$ converges to 
$N$, then $M_i$ has the same Lipschitz homotopy type as $N$ for large $i$
 (see also Proposition \ref{prop:retractionYX}).

In this context, we have the following diffeomorphism stability {\it  within} 
 $\ca M(n,\kappa,\nu,\lambda,d)$.

\begin{prop} \label{prop:w-stability}
For a given 
$M\in \ca M(n,\kappa,\nu,\lambda,d)$, there exists a positive number $\e=\e_M>0$ such that
if $M'\in \ca M(n,\kappa,\nu,\lambda,d)$ satisfies $d_{GH}(M,M')<\e$, then it is 
$\tau_M(\e)$-almost isometric and diffeomorphic to $M$, 
where $\lim_{\e\to 0}\tau_M(\e)=0$.
\end{prop}

Proposition \ref{prop:w-stability} solves a question raised in \cite[2.3.Questions]{wong0}.
\pmed

\begin{proof}
Suppose the conclusion does not hold.
Then we have a sequence $M_i \in \ca M(n,\kappa,\nu,\lambda,d)$ such that
it converges to an element  $M$  of $\ca M(n,\kappa,\nu,\lambda,d)$ 
while $M_i$ are not diffeomorphic to $M$.
From the convergence $M_i\to M$,
we have a lower volume bound
${\rm vol}(M_i)\ge v>0$ for a uniform constant
$v$.

Passing to a subsequence, we may assume that 
$\tilde M_i$ converges to an Alexandrov space $Y$ of dimension $n$, and that
$M_i$ converges to a closed subset $X$ of $Y$ under this convergence.
By Stability Theorem \ref{thm:stability}, $\tilde M_i$ is homeomorphic to $Y$ for large enough $i$.
Proposition \ref{prop:intrinsic} then shows that $X^{\rm int}$ is isometric to $M$.
It follows that $X_0$, the topological boundary  $\pa X$ of $X$ in $Y$, is isometric to $\pa M$,
and hence the singular set $\ca S$ is empty. 
Lemma \ref{lem:X02} implies that $X_0=X_0^1$. It follows from Lemma \ref{lem:eta'} that $\eta_0:C_0 \to X_0$ 
is an isometry and $Y$ is isometric to $\tilde M$.   Since $D(\tilde M_i)$ converges to $D(\tilde M)$, 
it follows from the equivariant fibration theorem in \cite{Ya:four} that  $\tilde M_i$ is diffeomorphic to $\tilde M$,
and hence  $M_i$ is diffeomorphic to $M $. This completes the proof.
\end{proof} 

\setcounter{equation}{0}
\pmed
\section{Behavior of boundaries}
\label{sec:behavior}

In this subsection, we discuss the behavior 
of $\pa M$ or its components under 
the Gromov-Hausdorff convergence in 
$\ca M=\ca M(n,\kappa,\nu,\lambda,d)$,
from several points of view.
\pmed
\n
{\bf Volume convergence.}\,

\begin{lem} \label{lem:vol-conv}
Let a sequence $M_i \in \ca M(n,\kappa,\nu,\lambda,d)$ converge to a 
geodesic space $N$ with respect to the  Gromov-Hausdorff  distance. Then 
$(M_i,\mu^n_{M_i})$ converges to $(N,\mu^n_N)$
with respect to the measured Gromov-Hausdorff
topology.

More generally, if $X_i\in\ol{\ca M}(m)$ converges to $X$, then 
$(X_i,\mu^m_{X_i})$ converges to $(X,\mu^m_X)$.
 In particular, $\ol{\ca M}(m,v)$
is compact.
\end{lem}

For the definition of the measured Gromov-Hausdorff topology, see Fukaya \cite{Fuk}.

\begin{proof}[Proof of Lemma \ref{lem:vol-conv}]
If  ${\rm vol}(M_i)\to 0$, we easily have 
$\dim N\le n-1$,
and hence we have nothing to do.
In what follows, we assume ${\rm vol}(M_i)\ge v>0$ for a uniform constant $v$.
By Heintz-Karcher\cite{HK} together with 
Proposition \ref{prop:cpn-diam},  we have 
\[
\vol (B(\pa M_i,\delta))\le\vol (\pa M_i)\tau_{n,\kappa,\lambda}(\delta)\le \tau_{n,\kappa,\lambda,d}(\delta).
\]
Applying \cite[Theorem 10.8]{BGP} to the convergence
$M_i\setminus\mathring{B}(\pa M_i,\delta)\to
N\setminus\mathring{B}(N_0,\delta)$,
we obtain the conclusion.

For the second half of the statement, 
let $Y_i$ be an extension of $X_i$.
We may assume that 
$Y_i$ converges to an  Alexandorov
space $Y$, which is an extension of $X$.
Since $\ca H^m(C^{Y_i}_{t_0})$
converges to  $\ca H^m(C^{Y}_{t_0})$,
$\ca H^m(Y_i\setminus X_i)$
converges to  $\ca H^m(Y\setminus X)$.
The conclusion follows from the convergence
$\ca H^m(Y_i)\to \ca H^m(Y)$.
\end{proof}

\begin{proof} [Proof of Theorem \ref{thm:vol-conv}]
 If ${\rm vol}(\pa M_i)\to 0$, then 
$\dim C_0=\dim N_0\le n-2$,
and hence we have nothing to do.
In what follows, we assume ${\rm vol}(\pa M_i)\ge w>0$ for a uniform constant $w$.
Applying \cite[Theorem 10.8]{BGP} to the convergence
$\pa M_i\to C_0$, we have 
 a measurable map
$\tilde\varphi_i:(\pa M_i)^{\rm int}\to C_0$
such that 
\beq\label{eq:H(paM)toN}
(\tilde\varphi_i)_*(\mu^{n-1}_{\pa M_i})
\to \mu^{n-1}_{N_0}.
\eeq
Consider the measurable map
$$
\varphi_i:=\eta_0\circ\tilde\varphi_i\circ \iota_i^{-1}:(\pa M_i)^{\rm ext}\to X_0,
$$
where $\iota_i:(\pa M_i)^{\rm int}\to(\pa M_i)^{\rm ext}$ denotes the identical map.
Note that $\varphi_i$ is an $o_i$-approximation.
Consider the 
decomposition
\[
     C_0^1=({\rm int}\, C_0^1) \cup {\tilde{\mathcal S}}^1.
\]
 Lemma \ref{lem:eta'} shows that 
\begin{align} \label{eq:pavol1}
   (\eta_0)_*(\mu^{n-1}_{{\rm int}\, C_0^1})= 
   \mu^{n-1}_{{\rm int}\, N_0^1}. 
\end{align}
Theorem \ref{thm:dim(metric-sing)} shows that 
both $ {\tilde{\mathcal S}}^1$ and  $\ca S^1$ have $\ca H^{n-1}$-measure zero,
which implies that 
\begin{align}
     (\eta_0)_*(\mu^{n-1}_{C_0^1})= 
   \mu^{n-1}_{N_0^1}.  \label{eq:pavol2}
\end{align}
Similarly consider the decomposition  
\[
    C_0^2=({\rm int} \,C_0^2) \cup  {\tilde{\mathcal S}}^2.
\]
Lemma \ref{lem:eta'} shows that 
\begin{align}
  (\eta_0)_*(\mu^{n-1}_{{\rm int} \,C_0^2})=
  2\mu^{n-1}_{N_0^2}.     \label{eq:pavol3}
\end{align}

 On the other hand, since $\ca S^1$ is closed
(Lemma \ref{lem:S1-closed}) and 
$\ca H^{n-1}(\ca S^1)=0$, we have  the decomposition:
\[
     {\tilde{\mathcal S}}^2 =  {\tilde{\mathcal S}}^2 \cap \pa({\rm int}\,C_0^2) 
                              \cup ( {\tilde{\mathcal S}}^2\cap  \pa({\rm int}\,C_0^1) \setminus  \pa({\rm int}\,C_0^2)).
\]
By Proposition \ref{prop:eta-2cover}, we have 
\begin{equation}
\begin{cases}
 \begin{aligned}
& \eta_0:  {\tilde{\mathcal S}}^2 \cap \pa({\rm int}\,C_0^2) \to \ca S^2\cap  \pa({\rm int}\,N_0^2)\hspace{3.5cm}  \\
     &  \text{is locally almost isometric double covering,} \hspace{2.5cm} \label{eq:double-cov1}
 \end{aligned}
\end{cases}
\end{equation}
which implies that 
\begin{align}
(\eta_0)_*(
  \mu^{n-1}_{\tilde{\ca S^2} \cap \pa({\rm int}\,C_0^2)}) 
  = 2 \mu^{n-1}_{\ca S^2\cap\pa({\rm int}\,N_0^2)}. 
                     \label{eq:pavol4}
\end{align}
Again Proposition \ref{prop:eta-2cover} implies that
\begin{equation}
\begin{cases}
 \begin{aligned}
     &\eta_0:  ({\tilde{\mathcal S}}^2\cap  \pa({\rm int}\,C_0^1) 
                     \setminus  \pa({\rm int}\,C_0^2)) \to 
                \ca S^2\cap  \pa({\rm int}\,N_0^1) \setminus  \pa({\rm int}\,N_0^2) \\
&\text{is a locally  almost  isometric double covering,} \hspace{2.5cm} 
                                     \label{eq:double-cov2}
 \end{aligned}
 \end{cases}
\end{equation}
which implies that 
\begin{align} (\eta_0)_*(
  \mu^{n-1}_{({\tilde{\mathcal S}}^2\cap  \pa({\rm int}\,C_0^1) \setminus  \pa({\rm int}\,C_0^2)})
 = 2 \mu^{n-1}_{\ca S^2\cap  \pa({\rm int}\,N_0^1) \setminus  \pa({\rm int}\,N_0^2)}.
                     \label{eq:pavol5}
\end{align}
From \eqref{eq:pavol3},  \eqref{eq:pavol4} and   \eqref{eq:pavol5}, we obtain 
\begin{align}
    (\eta_0)_*(\mu^{n-1}_{C_0^2})
= 2\mu^{n-1}_{N_0^2}. \label{eq:pavol6}
\end{align}

For arbitrary measurable sets $A_k\subset N_0^k$
\,$(k=1,2)$,
combining  \eqref{eq:H(paM)toN}, \eqref{eq:pavol2} and \eqref{eq:pavol6},
we get 
\[
\lim_{i\to\infty} \vol(\varphi_i^{-1}(A_1))=\ca H^{n-1}(A_1), \quad
\lim_{i\to\infty} \vol(\varphi_i^{-1}(A_2))= 2\ca H^{n-1}(A_2).
\]
This yields the conclusion of Theorem \ref{thm:vol-conv}.
\end{proof}

\pmed\n
{\bf   Geometric size of boundary components.}\,
First we verify the uniform boundedness for volume ratio of boundary components of $M$ 
in $\ca M(n,\kappa,\nu,\lambda, d)$.

\begin{prop} \label{prop:volume-ratio}
There are positive constants 
$v_i=v_i(n,\kappa,\nu,\lambda,d)$\,$(i=1,2)$\, such that 
for every $M\in \ca M(n,\kappa,\nu,\lambda, d)$ and for every components 
$\pa^{\alpha} M$ and $\pa^{\beta} M$ of $\pa M$, we have 
\[
          v_1 \le \frac{\vol (\pa^{\alpha} M)}{\vol(\pa^{\beta} M)} \le v_2.
\]
\end{prop}

\begin{proof}
By Proposition \ref{prop:cpn-diam}, we have  $\diam(\pa^{\alpha} M)\le d_0$ for some $d_0=d(n,\kappa,\nu,\lambda, d)>0$.
Take the constants $t_0$ and $\e_0$
such that $t_0=2d_0$ and $\e_0=1/10$,
and choose a proper warping function $\phi:[0, t_0]\to [1,\e_0]$ so as to satisfy 
\eqref{eq:phi}.
As before, we carry out the extension procedure to obtain $\tilde M$.
The glued warped product $C_M$ can be expressed as 
\[
      C_M= \bigcup_{\alpha\in A} \,C_{M,\alpha}, \,\,\, C_{M,\alpha} = \pa^{\alpha} M\times_{\phi}[0, t_0].
\]
Take a point $p_{\alpha}\in\pa^{\alpha} M$, 
and  set $\tilde p_\alpha:=(p_\alpha,t_0/2)\in C_{M,\alpha}$.  Note that 
\[
                                  C_{M,\alpha}\supset B^{\tilde M}(\tilde p_{\alpha}, t_0/2),
\]
where 
\begin{align*}
   \vol (C_{M,\alpha}) =\int_0^{t_0} \phi(t)^{n-1}\,\vol (\pa^{\alpha} M)\, dt =c(n, t_0)\vol (\pa^{\alpha} M).
\end{align*}
Since $\diam(\tilde M)\le 2t_0+d$,
it follows from the Bishop-Gromov volume
comparison theorem that 
\begin{align*}
   1 & \le \frac{\vol(\tilde M)}{c(n, t_0)\vol (\pa^{\alpha} M)}  = \frac{\vol(\tilde M)}{\vol (C_{M,\alpha})} \\
       & \le   \frac{B(\tilde p_{\alpha},2t_0+d)}{B(\tilde p_{\alpha}, t_0/2)}  \le \frac{b_K^n(2t_0+d)}{b_K^n(t_0/2)},
\end{align*}
where $K=K(\kappa,\nu,\lambda,d)$ is the lower sectional curvature bound of $\tilde M$ and $b_K^n(r)$ denotes the volume 
of $r$-ball in the $n$-dimensional complete simply connected space of constant sectional curvature $K$.
Thus we have 
\[
        0< c_1(n,\kappa,\nu,\lambda, d) \le \frac{\vol(\tilde M)}{\vol(\pa^{\alpha}M)} \le c_2(n,\kappa,\nu,\lambda, d) <\infty,
\]
from which the conclusion follows immediately.
\end{proof}

\begin{cor} \label{cor:vol-ratio}
There are positive constants $w_i=w_i(n,\kappa,\nu,\lambda,d)$ satisfying the following$:$
Let a sequence $M_i\in \ca M(n,\kappa,\nu,\lambda, d)$ GH-converge to a compact geodesic space $N$,
and $N_0^{\alpha}$ be the limit of $\pa^{\alpha} M_i$ under the above convergence, $\alpha\in A$. 
Then  for arbitrary $\alpha,\beta\in A$, we have 
\[
          w_1\, \le \,\frac{\ca H^{m-1} (N_0^{\alpha})}{ \ca H^{m-1} (N_0^{\beta})} \, \le \, w_2,
\]
where $m=\dim N$.
\end{cor}

\begin{proof}
We proceed along the same line as the proof of Proposition \ref{prop:volume-ratio}
for the limit space $Y$of $\tilde M_i$ instead of $\tilde M$, and hence 
we omit the detail.
\end{proof}

\begin{rem}
It should be noted that in Corollary \ref{cor:vol-ratio}, $N_0^{\alpha}$ does not necessarily coincide with a component 
of $N_0$. Namely,  $N_0^{\alpha}$ may have nonempty intersection with 
 $N_0^{\beta}$ for distinct $\alpha, \beta\in A$.
For instance,   consider a flat disk $D$ with an inside open disk $E$ removed, and let 
$\e:=|\pa D, \pa E|$. Then take the limit of 
$M_\e:=D\setminus E$  as $\e\to 0$. 
\end{rem}

\begin{prop} \label{prop:diam-thin}
There exists a positive number $\e=\e(n,\kappa,\nu,\lambda,d)$ such that if $M\in \ca M(n,\kappa,\nu,\lambda, d)$
has a boundary component  $\pa^{\alpha} M$ satisfying $\diam(\pa^{\alpha} M)^{\rm int}<\e$, then we have
 \begin{enumerate}
  \item the number $k$ of components of $\pa M$ is at most two;
  \item if $k=2$, then $M$ is diffeomorphic to a product $\pa^{\alpha} M\times [0,1]$.
 \end{enumerate}
\end{prop}

\begin{proof}
We proceed by contradiction. Suppose the result does not hold. Then we have a 
sequence $M_i\in  \ca M(n,\kappa,\nu,\lambda, d)$ such that 
$\diam(\pa^{\alpha} M_i)\to 0$ for a component of $\pa M_i$ and 
one of $(1)$ and $(2)$ does not hold for $M_i$.
Passing to a subsequence, we may assume that $\tilde M_i$ converges to $Y$.
From the assumption, the component of $C_{t_0}^Y$corresponding to $\pa^\alpha M_i$
is a point. Thus we have $\dim Y = 1$, and $Y$ must be a segment,
which implies that the number of components of $\pa M_i$ is at most two
for large enough $i$.

Suppose $M_i$ has  disconnected boundary 
\[
          \pa M_i = \pa^{\alpha} M_i \cup \pa^{\beta} M_i
\]
for large enough $i$, and let $Y=[a, b]$. Since $\pa \tilde M_i$ converges to $\{ a, b\}$, 
it follows from the fibration theorem in \cite{Ya:pinching} that for some $\e>0$, 
$\tilde M_{i,\e}:=M_i\cup \pa M_i\times_{\phi} [0, t_0-\e]$ is diffeomorphic 
to $L_i \times [a+\e, b-\e]$, where $L_i$ is an $(n-1)$-dimensional closed manifold, 
and hence $M_i$ is diffeomorphic to $L_i\times [0,1]$, a contradiction to the  
assumption.
\end{proof}

\begin{rem} \label{rem:diam-thin}
(1)\,Note that Proposition \ref{prop:diam-thin} does not hold if one drops the upper diameter bound $\diam(M)\le d$.
\par\n
(2)\,Proposition \ref{prop:diam-thin} has a similarity with \cite[Theorem 1.5]{YZ:inradius}.
However the inradius collapse is not assumed in Proposition \ref{prop:diam-thin}.
 It should also be pointed out that 
Proposition \ref{prop:diam-thin} is a generalization of  \cite[Theorem 5]{wong2}, 
where a similar result was shown under the  assumption of small 
diameter $\diam(M)<\e$.
\end{rem}

\pmed\n
{\bf Criterions for inradius collapse.}\,
We provide two sufficient conditions for inradius collapse. One is about the limit spaces,
and the other is about boundary components.

\begin{prop} \label{prop:inradius}
Let $M_i$ in $\ca M(n,\kappa,\nu,\lambda,d)$  converge to a compact  geodesic space $N$  with respect to the Gromov-Hausdorff distance,
and suppose that $N$ is a  closed topological manifold or a closed Alexandrov space.
Then ${\rm inrad}(M_i)$ converges to zero. 
\end{prop}

\begin{rem} \upshape
In \cite{wong2}, Wong proved Proposition \ref{prop:inradius}
when $N$ is a Poincare duality space.
 Our method is different from the one
in \cite{wong2}.
\end{rem}

\begin{proof}[Proof of Proposition \ref{prop:inradius}]
We assume that $N$ is a closed Alexandrov space with curvature bounded below.
The case when $N$ is a closed topological manifold is similar.
Suppose that Proposition  \ref{prop:inradius} does not hold so that $r_i:={\rm inrad}(M_i)$ is uniformly  bounded away from $0$.
Take 
a point  $p_i \in\inte M_i$  and $q_i\in\partial M_i$  such that 
$r_i=|p_i, q_i|=|p_i, \pa M_i|$.
Passing to a  subsequence, we may assume that $(B(p_i,r_i), q_i)$ converges to a metric 
ball $(B(x_0,r), y_0)$ in $X$ under the convergence $\tilde M_i\to Y$, where $r>0$.

Take minimal geodesics $\gamma$ and  $\gamma_{y_0}^{+}$ from $y_0$ to $x_0$ and $C_{t_0}^Y$
respectively. 
Note that $\gamma$ and $\gamma_{y_0}^{+}$ form a minimal geodesic in $Y$.
It follows from Lemma \ref{lem:single-interior} that $y_0\in {\rm int}\, X_0^1$.
Take $p_0\in C_0$ with $\eta_0(p_0)=y_0$,
and fix  neighborhoods $U_0$ of  $p_0$ in $C_0$ and $V_0$ of $y_0$ in $X_0$ respectively in such a way 
that $\eta_0:U_0\to V_0$ is a homeomorphism.
Since $C_0$ is an Alexandrov space of dimension $m-1$,  
 one can take a point $y_1\in V_0$, which is close to $y_0$,  having 
a neighborhood $V_1\subset V_0$ homeomorphic to  $\mathbb R^{m-1}$.
Let $W_1$ be the neighborhood of 
$w_1:=\gamma_{y_1}^{+}(t_0)$ in $C_{t_0}^Y$ corresponding to $V_1$.
Take a small neighborhood $W$ of $W_1$, and consider the distance function 
$d_W$ from $W$. Taking $y_1$ close to $y_0$,
we may assume that $\tilde\angle w_1 y_1 x_0 >\pi-\delta$ for small $\delta>0$. Thus,
$d_W$ is regular on a neighborhood  $V_1$ 
of $y_1$ in $Y$.
Perelman's fibration theorem \cite{Pr:alexII} 
implies that $y_1$ has a neighborhood $P$ in $Y$
such that $P$ is a product neighborhood in the sense that we have a 
commutative diagram:
$$\xymatrix{
  P \ar[rr]^{\cong} \ar[dr]_{d_W}
                &  &    L\times[a,b] \ar[dl]^{\pi_2}    \\
                &  [a,b]                }
$$
where $L=P\cap V_1\simeq \mathbb R^{m-1}$.
On the other hand, from the assumption and Theorem \ref{thm:dim-sing}, there is a point $v\in L$ having a neighborhood in $X$ homeomorphic to $\mathbb R^m$.   
This contradicts to 
the invariance of domain theorem in $\mathbb R^m$. This completes the proof.
 \end{proof}

\begin{thm}\label{thm:close-components}
For $M\in\ca M(n,\kappa,\nu,\lambda)$,
suppose that there are distinct components 
$\pa^\alpha M$ and $\pa^\beta M$ of
$\pa M$ such that the Hausdorff distance 
$d_{H}^M (\pa^\alpha M,\pa^\beta M)$ in $M$ 
is less than $\e$
for some $\e=\e_n(\kappa,\lambda)>0$.
Then $M$ is diffeomorphic to 
$\pa^\alpha  M\times I$.
\end{thm} 

\begin{proof}
For $M\in\ca M(n,\kappa,\nu,\lambda)$, let us assume
that  
$$
d_{H}^M (\pa^\alpha M,\pa^\beta M)<\e
$$ 
for distinct components $\pa^\alpha M$ and $\pa^\beta M$ of $\pa M$. 
By Proposition \ref{prop:diam-thin}, we may assume that both  $\pa^\alpha M$ and $\pa^\beta M$ have diameters $\ge d_0$
for a uniform positive constant $d_0$.
Let $W$  be the $10\e$-neighborhood of 
$\pa^\alpha M\cup \pa^\beta M$ in $M$.

First we show 

\begin{slem}\label{slem:thetax=pi}
For any $x\in {\rm int} \,M\cap W$, let $\gamma_\alpha$ and $\gamma_\beta$ be 
minimal geodesics from $x$ to 
$\pa^\alpha M$ and $\pa^\beta M$ respectively, and let $\theta_x$ denote
the angle between them at $x$. Then we have
\[
        \theta_x >\pi-\tau_x(\e).
\]
\end{slem}
\begin{proof}
The basic strategy is the same as
 \cite[Sublemma 7.3]{YZ:part1}
and \cite[Lemma 8.10]{YZ:part1}.
Suppose that the sublemma does not hold. Then we have a sequence $M_i$ in
$\ca M(n,\kappa,\nu,\lambda)$ such that for two distinct 
components $\pa^\alpha M_i$ and $\pa^\beta M_i$ of $\pa M_i$, the following holds: 
\begin{itemize}
\item $d_H^{M_i}(\pa^\alpha M_i,\pa^\beta M_i)<\e_i$ with $\lim_{i\to\infty} \e_i=0,\,$
\item for some $x_i\in {\rm int} \,M_i\cap W_i$,
\[
      \theta_{x_i}\le \pi-\theta_0,
\]
for a constant $\theta_0>0$ independent of $i$, where $W_i$ and $\theta_{x_i}$ are defined for $\pa^\alpha M_i$,
$\pa^\beta M_i$ and $x_i$ as above using $\e_i$.
\end{itemize}
Let $\delta_i:=\max\{ |x_i,\pa^\alpha M_i|,
|x_i,\pa^\beta M_i|\}$, and let us consider
$$
\hat M_i:=M_i\bigcup_{\pa M_i} \pa M_i
\times_{\phi}
[0,\delta_i]\subset \tilde M_i.
$$
Since $\delta_i\le 11\e_i$, we obtain $\lim_{i\to\infty}\delta_i=0$.
Let $\pa^\alpha\hat M_i$ and 
$\pa^\beta \hat M_i$ be the components
of $\pa M_i$ corresponding to 
$\pa^\alpha M_i$ and $\pa^\beta M_i$
respectively. Obviously, the minimal geodesics
$\gamma_\alpha^i$ and  $\gamma_\beta^i$ joining $x_i$ to $\pa^\alpha M_i$ and $\pa^\beta M_i$ 
extends to minimal geodesics
$\hat\gamma_\alpha^i$, 
$\hat\gamma_\beta^i$ to
$\pa^\alpha \hat M_i$ and $\pa^\beta \hat M_i$ respectively.
Passing to a subsequence, we may assume that 
$(\frac{1}{\delta_i}\hat M_i, x_i)$
converges to a pointed space $(\hat Y,y)$.
Note that $\hat Y$ is a complete noncompact 
Alexandrov space with nonnegative curvature.
Let $\hat\gamma_\alpha^\infty$ and 
$\hat\gamma_\beta^\infty$ be the limits of 
$\hat\gamma_\alpha^i$ and  
$\hat\gamma_\beta^i$ respectively.
Clearly we have 
\beq \label{eq:hatangle=pi}
\angle_{y} (\hat\gamma_\alpha^\infty,\hat\gamma_\beta^\infty)\le \pi-\theta_0.
\eeq
In a way similar to the proof of 
 \cite[Sublemma 7.3]{YZ:part1}, we reach the contradiction 
to \eqref{eq:hatangle=pi}. We omit the detail since 
the argument is essentially the same.
\end{proof}

Let $U$ be the $5\e$-neighborhood
of $\pa^\alpha M\cup \pa^\beta M$ in $M$.
Using Sublemma \ref{slem:thetax=pi},
we can construct a vector field $V$
such that 
\begin{itemize}
\item $U\subset {\rm supp} V\subset {\rm int}\, \hat M\cap W\,;$
\item $V$ is almost gradient for the distance 
function $d_{\pa^\alpha \hat M}$ on $U$.
\end{itemize}
Using the flow curves of $V$, we obtain 
an open domain $W_0$ contained in $W$ that  is diffeomorphic to 
$\pa^\alpha M\times [0,1]$.
Since $M$ is assumed to be connected, we conclude $M=W_0$.
This completes the proof of Theorem \ref
{thm:close-components}.
\end{proof}
\pmed

\pmed\n
{\bf Obstruction to collapse.}\,
Finally, we discuss simplicial volumes as an obstruction of 
general collapse in $\mathcal M(n,\kappa,\nu,\lambda)$.
For a closed orientable $n$-manifold $N$, the {\it simplicial volume} $||N||$ introduced by Gromov \cite{G:vol}  is defined as 
\[
      ||N|| := \inf_{c\in [N]}\,\, ||c||_1,
\]
where $[N]\in H_n(N;\mathbb R)$ is a fundamental class of $N$, and $||c||_1=\sum_i |a_i|$
for the real cycles $c=\sum_i a_i\sigma_i$ representing $[N]$
(see \cite{G:vol}).
When $N$ is non-orientable,  $||N||$ is defined as the half of 
the simplicial volume of the orientable double covering of $N$.

\begin{prop} \label{prop:simplicial2}
For given $n,\kappa, \nu,\lambda>0$, there exists
$\e>0$ such that 
if every unit ball in $M\in \mathcal M(n,\kappa,\nu,\lambda)$ has volume $<\e$,
 then $||D(M)|| =0$.
\end{prop}

\begin{proof}
Suppose that the theorem does not hold. Then  for a sequence
$v_i\to 0$, there would exist a sequence
$M_i$ in $\mathcal M(n,\kappa,\nu,\lambda)$ such that 
every unit ball in $M_i$ has volume $<v_i$ while
$||D(M_i)||>0$.
Take $\alpha_i\to\infty$ satisfying
\[
     \lim_{i\to\infty} \alpha_i^n v_i = 0.
\]
Consider the rescaling $M_i' := \alpha_i M_i$.
Note that
\beq
 K_{M_i'}\ge -o_i, \quad K_{\pa M_i'}\ge -o_i,
\quad \Pi_{\pa M_i'} \ge  o_i,  \label{eq:rescale-small}
\eeq
where $\lim_{i\to\infty} o_i=0$. Moreover, we have 
\beq \label{eq:rescale-small}
 \text{any metric $\alpha_i$-ball in $M_i'$ has volume less than $\alpha_i^n v_i\to 0$.}\hspace{0.5cm}
\eeq
By \eqref{eq:rescale-small}, we can take the warping function 
$\phi_i:[0, t_{0,i}]\to [\e_{0,i}, 1]$ of  \eqref{eq:phi}  so that the height 
$t_{0,i}$ and the scaling factor $\e_{0,i}=\phi_i(t_{0,i})$ of the end of 
the glued cylinder $C_{M_i'}$  satisfy 
\[
    \lim_{i\to\infty} t_{0,i}=0, \quad  \lim_{i\to\infty} \e_{0,i}=1.
\]
Let $\tilde M_i'$ be the result of the gluing of $M_i'$ and $C_{M_i'}$ as before.

\begin{ass}  \label{ass:small-vol}
For any $\e>0$ there exists an $I$ such that for any $i\ge I$ and $p_i\in M_i$,
we have 
\[
             \vol B^{\tilde M_i'}(p_i,1) <\e.
\]
\end{ass}
\begin{proof}
Otherwise, we have a subsequence $\{ j\}$ of $\{ i\}$ such that 
\beq
                  \vol B^{\tilde M_j'}(p_j,1) \ge \e_1>0  \label{eq:large-ballvol}
\eeq
for some $p_j\in M_j$, 
where $\e_1$ is a uniform positive constant.
Passing to a subsequence, we may assume that 
$(\tilde M_j', p_j)$ converges to an Alexandrov space $(Y,y)$.
It follows from \eqref{eq:large-ballvol} that $\dim Y=n$.
Since $d_{GH}(M_j'^{\rm ext}, \tilde M_j')<t_{0,i}$ and since
$M_j'$ and $M_j'^{\rm ext}$ are $1/\e_{0,i}$-bi-Lipschitz homeomorphic 
to each other, 
both limits $(N,q)$ and  $(X,x)$ of $(M_j',p_j)$ and $(M_j'^{\rm ext},p_j)$ respectively 
isometric to $(Y,y)$. It turns out that 
\[
            \liminf \vol\, B^{M_j'}(p_j, 1) \ge c>0
\]
for some uniform constant $c$. This is a contradiction to \eqref{eq:rescale-small}.
\end{proof}

Now Isolation Theorem (\cite{G:vol})  
yields that $||D(\tilde M_j')||=0$ for all large $j$.
Here it should be noted that we can certainly apply the proof of  Isolation Theorem in \cite{G:vol} to $D(\tilde M_j')$ since $D(\tilde M_j')$ has a triangulation and is an Alexandrov space with a uniform lower curvature bound, although it is only a $C^0$-Riemannian manifold.
This completes the proof.
\end{proof}

We have the following immediately.

\begin{cor} \label{cor:simplicial}
Let $M$ be a compact $n$-manifold with boundary.
If the double of $M$ has non-vanishing  simplicial volume  $||D(M)||>0$,
then $M$ does not carry Riemannian metrics which  
collapse in $\mathcal M(n,\kappa,\nu,\lambda)$ for any fixed $\kappa,\nu,\lambda$.
\end{cor}

\begin{rem} \label{rem:Ricci-Gromovinv}
Proposition \ref{prop:simplicial2} actually holds for the following family 
 $\ca R(n,\kappa,\eta, \lambda)$ of $n$-dimensional compact 
Riemannian manifolds  $M$ whose Ricci curvatures are 
uniformly bounded below:
\beq\label{eq:Ricci-bounds}
      {\rm Ricci}_M \ge \kappa, \,\,
 {\rm Ricci}_{\pa M} \ge \eta,\,\,
  \Pi_{\partial M}\ge-\lambda.
\eeq
Actually applying Perelman's
gluing theorem (\cite{P97}, see also \cite{BWW19})
to the above argument, we get the conclusion 
in a similar way.
See \cite{HY} on  some results for inradius 
collapsed manifolds satisfying \eqref{eq:Ricci-bounds} with 
uniformly bounded diameters.
\end{rem}

\begin{ex}
(1)\, Let $M$ be a punctured torus of negative curvature 
having totally geodesic boundary.
Then $\e M$ collapses to a point as $\e\to 0$ while $||D(M)||>0$.
This example shows that Proposition \ref{prop:simplicial2}  does not hold
if one drops the lower curvature bound $\kappa$.

(2)\,
Let $N \subset \mathbb R^2$
be the union of the unit circle $\{ (x,y)\,|\, x^2+y^2=1\}$ and the segment $\{ (x,y)\,|\, x=0, -1\le y\le 1\}$.
Let $M_{\e}$ be the closed $\e$-neighborhood of $N$ in $\mathbb R^2$.
After slight smoothing of $M_{\e}$, it is a compact surface with smooth boundary
such that 
$K_{M_{\e}}\equiv 0$ and 
$\inf {\rm II}_{\pa M_{\e}} \to -\infty$ as $\e\to 0$, and
$M_{\e}$ collapses to $N$ while $||D(M_{\e})||>0$.
This example shows that Proposition \ref{prop:simplicial2}  does not hold
if one drops the lower bound $\Pi_{\pa M}\ge -\lambda$.
\end{ex}

\begin{prob}
Determine if Proposition \ref{prop:simplicial2} still holds
if one drops the lower bound $K_{\pa M}\ge \nu$.
\end{prob}


\end{document}